\documentclass[11pt,a4wide]{article}
\usepackage{etex}
\usepackage[utf8]{inputenc}
\usepackage[english]{babel}
\usepackage{enumitem}
\usepackage{amsmath,amssymb,amsrefs,mathrsfs,mathtools}
\usepackage{esint} 

\usepackage{comment} 
\usepackage{tikz}
\usepackage{booktabs,ctable, multirow}

\usepackage{fullpage}
\title{}
\author{}
\usepackage{cite}
\usepackage{commath}
\usepackage{multirow}

\numberwithin{equation}{section}  
\usepackage{epsfig,amsbsy,amsthm}
\usepackage{float,epsfig}
\usepackage{fixmath,textcomp}
\usepackage{graphics}
\usepackage{pgfpages,authblk}
\usepackage{appendix}
\usepackage{graphicx}
\usepackage{pst-all}
\allowdisplaybreaks 

\usepackage{soul}
\usepackage{ulem}

\newtheorem{theorem}{\bf Theorem}[section]

\newtheorem{definition}{\bf Definition}[section]

\newtheorem{example}[theorem]{\bf Example}
\theoremstyle{remark} 
\newtheorem{remark}[theorem]{\bf Remark}
\newtheorem{lemma}[theorem]{\bf Lemma}

\newcommand*{\rom}[1]{\expandafter\@slowromancap\romannumeral #1@}
\newcommand{\trinl}{\ensuremath{|\!|\!|}}
\newcommand{\trinr}{\ensuremath{|\!|\!|}}

\title{\bf LOWEST-ORDER NONSTANDARD FINITE ELEMENT METHODS FOR TIME-FRACTIONAL BIHARMONIC PROBLEM} 
\author{Shantiram Mahata\textsuperscript{1}\thanks{\texttt{shantiram@math.iitb.ac.in}}\;\;, \; 
Neela Nataraj\textsuperscript{1}
\thanks{\texttt{neela@math.iitb.ac.in}}\;\;,\; and
Jean-Pierre Raymond\textsuperscript{2}
\thanks{\texttt{raymond@math.univ-toulouse.fr}}\\
\textsuperscript{1} {\small Department of Mathematics, Indian Institute of Technology Bombay, Mumbai, Maharashtra - 400076, India}\\
\textsuperscript{2}{\small  Universit{\'e} Paul Sabatier Toulouse {\rm III} \& CNRS, Institut de Math{\'e}matiques, 31062 Toulouse Cedex 9, France; Visiting Professor, Indian Institute of Technology Bombay, Mumbai, Maharashtra - 400076, India}}

\theoremstyle{definition}

\newtheorem{rem}{Remark}[section]

\numberwithin{equation}{section}

\newcommand{\cT}{\mathcal T}

\newcommand{\pw}{\mathrm{pw}}
\newcommand{\nc}{\mathrm{M}}

\newcommand{\half}{\frac{1}{2}}

\DeclareMathOperator{\E}{\mathcal{E}}
\newcommand{\M}{\mathrm{M}}

\newcommand{\T}{\mathcal{T}}

   \usepackage{xparse}
   
   \newcounter{const}
\NewDocumentCommand{\constant}{o}
 {
  \IfValueTF{#1}%
  {C_{#1}}%
  {\refstepcounter{const}%
  C_{\theconst}}%
 }

\usepackage{tikz,pgfplots}
\usetikzlibrary{calc,angles,positioning,intersections,quotes,decorations.markings}
\usepackage{tkz-euclide}
\pgfplotsset{compat=1.5}
\begin{document}
\date{}

\maketitle


\begin{abstract}
     In this work, we consider an initial-boundary value problem for a time-fractional biharmonic equation in a bounded polygonal domain with a Lipschitz continuous boundary in $\mathbb{R}^2$ with clamped boundary conditions. After establishing the well-posedness, we focus on some regularity results of the solution with respect to the regularity of the problem data. The spatially semidiscrete scheme covers several popular lowest-order piecewise-quadratic finite element schemes, namely, Morley,  discontinuous Galerkin, and $C^0$ interior penalty methods, and includes both smooth and nonsmooth initial data. Optimal order error bounds with respect to the regularity assumptions on the data are proved for both homogeneous and nonhomogeneous problems. The numerical experiments validate the theoretical convergence rate results.
\end{abstract}



\section{Introduction}\footnote{ In the first version, even if the results are mathematically correct, they are not optimally written since we assumed that $\Omega$ is convex and introduced some results which are also valid in the case of nonconvex polygons while some other results are valid only in the case of convex polygons. This version extends the results to the case when $\Omega$ is a polygonal domain with a Lipschitz continuous boundary.}
Let $\Omega$ be a bounded  polygonal domain with a Lipschitz continuous boundary $\partial\Omega$ in $\mathbb{R}^2$ and $T$ be a fixed positive real number. For $0<\alpha<1$, consider the following  initial-boundary value problem for time-fractional biharmonic equation that seeks {$u(t)\in H_0^2(\Omega)$} such that 
\begin{align}\label{bhmaineqn}
\begin{aligned}
\partial_t^{\alpha}u(t)+Au(t)&= f(t)  \text{ in }  \Omega, \; 0<t\le T, \\
u(0)&= u_0 \text{ in } \Omega,\\
\end{aligned}
\end{align}
where $A$ is the unbounded operator in $L^2(\Omega)$ corresponding to the biharmonic operator $\Delta^2$ with clamped boundary conditions.
For an absolutely continuous real function $u$ over $[0,T]$, $\partial_t^{\alpha}u(t)$  denotes the Caputo fractional derivative defined by
\begin{align}\label{caputoderiv}
\partial_t^{\alpha}u(t)= \frac{1}{\Gamma(1-\alpha)}\int_{0}^{t}(t-s)^{-\alpha}u'(s)\,ds
=\mathcal{I}^{1-\alpha}u'(t), \;\; 0<\alpha<1,
\end{align} 
where $u'(t):=\partial_tu(t)$ and $\mathcal{I}^{\beta}$, $0<\beta<\infty$, denotes the Riemann-Liouville fractional integral defined by 
\begin{align}\label{rlfracintegral}
\mathcal{I}^{\beta}v(t)=\int_{0}^{t}\kappa_{\beta}(t-s)v(s)\,ds, \quad \kappa_{\beta}(t):=t^{\beta-1}/\Gamma(\beta),
\end{align}
with the gamma function $\Gamma(z)=\int_{0}^{\infty}e^{-t}t^{z-1}\;dt,\; \text{Re}(z)>0$.  In Definition \ref{weak-sol}, we will use \eqref{caputoderiv} for absolutely continuous functions with values in a Hilbert space $V^*$ introduced later. 


It is well-known that fractional order models describe physical phenomena more accurately compared to the usual integer-order models (cf. \cite{podlubny}). The equation in model (\ref{bhmaineqn}) represents a particular case of the time-fractional Cahn-Hilliard equation \cite{huangstynes21,maskarikaraa22}.

The main goal of this work is to study the convergence analysis of semidiscrete approximation with lowest-order nonstandard finite element methods (FEMs).  Prior to the development of numerical methods, we first prove the well-posedness of problem \eqref{bhmaineqn}.

To present our contributions, we introduce below some standard notations.  For any real $r$,  $H^r(\Omega)$ denotes the standard Sobolev space associated with the Sobolev-Slobodeckii semi-norm $|\cdot|_{H^{r}(\Omega)}$ (cf. \cite{grisvard92}). Let $H^r_0(\Omega)$ denote the closure of $D(\Omega)$ in $H^r(\Omega)$, $V:= H^2_0(\Omega) $, and let $H^{-r}(\Omega)=(H^r_0(\Omega))^{*}$ be the dual of $H^r_0(\Omega)$. The notations $(\cdot,\cdot)$ and $\|\cdot\|$  denote the $L^2$ scalar product and $L^2$ norm in  ${\Omega}$.
For a Hilbert space $H$ with norm $\|\cdot\|_H$, let $L^{p}(0,T;H)$, $1\le p<\infty$ denote the standard Bochner space equipped with the norm $\|g\|_{L^{p}(0,T;H)}=(\int_{0}^{T}\|g(t)\|_H^p\,dt)^{1/p}$. If $p=\infty$, the space $L^{\infty}(0,T;H)$ has the norm $\|g\|_{L^{\infty}(0,T;H)}=\inf\{ C: \|g(t)\|_H\le C \text{ almost everywhere on } (0,T)\}$. In addition, for $m\ge 0$, $1\le p\le \infty$, 
$W^{m,p}(0,T;H)$ $:=\{g:(0,T)\rightarrow H \text{ such that } \frac{d^jg}{dt^j}\in L^{p}(0,T;H) \text { for } 0\le j\le m\}$ is equipped with the norm $\|g\|_{W^{m,p}(0,T;H)}=\sum_{j=0}^{m}\|\frac{d^jg}{dt^j}\|_{L^{p}(0,T;H)}$ (cf. \cite{evans10}).
Hereafter, $\gamma_0\in (1/2,1]$ is the elliptic regularity index of the biharmonic problem introduced in Section \ref{bhpreliminaries}. Throughout the paper, $a\lesssim b$ denotes $a\le C b$, where $C$ is a positive generic constant that may depend on fractional order $\alpha$ and the final time $T$ but is independent of the spatial mesh size. 

 The main contributions of this work are summarized below.
\begin{itemize}
    \item [1.] We study the existence, uniqueness, and stability of the solution of time-fractional biharmonic problem \eqref{bhmaineqn} with clamped boundary conditions. In  Theorem \ref{bhwellposedness-thm}, we prove that when initial data $u_0$ is simply an element of $L^2(\Omega)$ and source function $f$ belongs to $W^{1,1}(0,T;V^{*})$, problem \eqref{bhmaineqn} admits a unique weak solution   with
\begin{align*}
\|u\|_{L^1(0,T;V)}+\|u'\|_{L^1(0,T;V^*)}\lesssim  \|u_0\| +  \|f\|_{W^{1,1}(0,T;V^*)}.    
\end{align*}
In addition, we also derive regularity results for the solutions of both homogeneous and nonhomogeneous problems useful for the error analysis. 

\item [2.] In \eqref{bhritzprjn}, we introduce a novel Ritz projection $\mathcal{R}_h: V \rightarrow V_h$, where $V_h$ denotes the space of lowest-order piecewise-quadratic polynomials for \eqref{bhmaineqn} {with clamped boundary conditions } in the proposed semidiscrete schemes and includes  popular Morley, discontinuous Galerkin (dG), and $C^0$ interior penalty ($C^0$IP) methods. We establish the quasi-optimal approximation property displayed below in  Lemma \ref{bhritzprjn_approx}. 
$$\|{v-\mathcal{R}_hv}\| + h^{\gamma}\norm[0]{v-\mathcal{R}_hv}_h\lesssim h^{2\gamma}\norm[0]{v}_{H^{2+\gamma }(\Omega)}, $$
 for all $\gamma\in [0,\gamma_0]$  and for all  $v\in H^{2+\gamma }(\Omega)$, where the regularity index $\gamma_0$ belongs to $(1/2,1)$ if $\Omega$ is not convex and $\gamma_0=1$ if $\Omega$ is convex. These properties are crucial to obtaining optimal order error bounds for both smooth and nonsmooth initial data. 

\item[3.] We use an energy argument in a unified framework for the nonstandard FEM analysis with lowest-order piecewise-quadratic polynomials for \eqref{bhmaineqn}.
Since the solution of model \eqref{bhmaineqn} reflects a singular behavior near $t=0$ (see below Theorems \ref{bhhomop_reg} and \ref{bhnonhomop_reg}), a straightforward analysis of numerical scheme is problematic. This is addressed by multiplication by weights of type $t^j$, $j=1,2$ in the semidiscrete scheme. For initial data in $L^2(\Omega)$ and a source term equal to zero, in Theorem \ref{bhsemidisctesmooth_finl} $(i)$ we obtain the following error bound in the $L^2(\Omega)$ and energy norms for each fixed time $t\in (0,T]$,
\begin{align*}
\norm[0]{u(t)-u_h(t)}&\lesssim  \big( h^{2\gamma_0}t^{-\alpha} +h^2t^{-(1+\alpha)/2}\big)\norm[0]{u_0},\\
\norm[0]{u(t)-u_h(t)}_h&\lesssim h^{\gamma_0}t^{-\alpha}\norm[0]{u_0}+ \big( h^{2\gamma_0}t^{-3\alpha/2} +h^2t^{-(\alpha+1/2)}\big)\norm[0]{u_0},
\end{align*} 
where $h$ is the mesh size of the FEM. Similar error bounds are also proved for the nonhomogeneous problem with an appropriate assumption on $f$ and $u_0\in D(A)$ in Theorem \ref{bhsemidisctesmooth_finl} $(ii)$ for $t\in (0,T]$.

\item[4.] Numerical experiments that validate the theoretical results are presented for examples with smooth and nonsmooth initial data.    
\end{itemize}

\medskip \noindent The article is organized as follows:
\begin{itemize}
    \item [-] In Section 2, we introduce some notations, establish the well-posedness of problem  (\ref{bhmaineqn}), and prove regularity results of the exact solution. 
    \item[-] This is followed by a discussion on the semidiscrete schemes and error analysis in Sections 3 and 4. Optimal order error bounds are proved for both smooth and nonsmooth initial data. 
    \item[-] Finally, Section 5 presents the results of the numerical implementations that validate the theoretical estimates. 
\end{itemize}
Below, we compare our findings with those already known from the literature and review the existing results for \eqref{bhmaineqn}.
\begin{itemize}
    \item[-] The well-posedness result in Point (1) above is inspired by the work in \cite{sy11} for second-order fractional diffusion-wave equations. For \eqref{bhmaineqn}, the well-posedness result is summarized in \cite[Appendix]{lietal2022_cnc}. However, in \cite[(A.1) and (A.2)]{lietal2022_cnc}, the well-posedness and regularity results are stated for {\it higher} regularity assumptions on the problem data, for example, initial data $u_0\in D(\Delta^{j+2})$ and source function $f$ such that $\frac{\partial^{l} f (t)}{\partial t^{l}}\in D(\Delta^{j})$, where $j\in \mathbb{N}$ and $l=0,1,2,3$. These studies to \eqref{bhmaineqn} in our work are done comparatively less smoother assumptions on the problem data (cf. Theorems \ref{bhwellposedness-thm}, \ref{bhhomop_reg}, and \ref{bhnonhomop_reg}). 
    \item[-] The introduction of the Ritz projection for lowest-order piecewise-quadratic based FEM  on $H_0^2(\Omega)\cap H^{2+\gamma}(\Omega)$ for all  $\gamma\in [0,\gamma_0]$  and its quasi-optimal approximation properties on $H^{2+\gamma}(\Omega)$ in $L^2(\Omega)$ and energy norms are new. The idea of this operator is borrowed from \cite{cn2022} for the analysis of FEMs for biharmonic plates. In \cite{lietal2022_cnc}, the authors have introduced the Ritz projection on $H_0^2(\Omega)\cap H^{j+1}(\Omega)$, $j\ge 2$ for the analysis of virtual element methods, and have proved optimal projection errors which demand at least $H^3(\Omega)$ regularity of the continuous solution $u(t)$. 
    We also have to mention that the idea of using a Ritz projection is suggested, without proof, in \cite[Remark 3.1]{danumjayaetal2021}, for the Morley finite element approximation of fourth-order nonlinear reaction-diffusion problems.

     There are technical difficulties that arise in the error analysis because (i) the space discretization is not conforming and (ii) the error estimates are established under lower regularity of the exact solution. To overcome these challenges, we introduce a novel Ritz projection $\mathcal{R}_h$ using a smoother (or companion) operator $Q$ in \eqref{bhritzprjn}. This leads to an error equation in \eqref{bherreqn} that is different from those obtained for the conforming second-order problems \cite{mustapha18,karaamustaphapani18}. 
     The idea is also completely new in comparison to the available works for time-fractional fourth-order problems in the literature.  

    \item[-] The semidiscrete convergence analysis in this work is inspired by the technique in \cite{karaamustaphapani18, mustapha18, smrksenergy} for time-fractional parabolic problems with and without memory. These works discuss only second-order problems and cover only the conforming case, where the discrete trial space is a finite-dimensional subspace of $H_0^1(\Omega)$. In contrast, the analysis of conforming FEMs for fourth-order problems demands $C^1$ continuity and requires a high polynomial degree, which imposes many conditions on the vertices and edges of an element, and hence, is prohibitively expensive. Most of the available literature for \eqref{bhmaineqn} discuss {\it simply supported boundary conditions} ($u=\Delta u=0$) that enables a mixed formulation via an auxiliary variable $v=\Delta u$  and subsequently, the reduction of the problem to a system of two second-order equations.    
    
    As far as we know, the semidiscrete error analysis with $u_0\in L^2(\Omega)$ is a totally new result for the time-fractional biharmonic problem with clamped boundary conditions.

    \item[-]  The work in \cite{cuifd2021} deals with a compact finite difference scheme with clamped plate problem, where the author has discretized the spatial derivatives by the Stephenson scheme on uniform meshes, and fractional derivative by L1 scheme on graded time meshes, and established the stability of the scheme and convergence analysis under the assumptions $\|\frac{\partial^{i+j}u(t)}{\partial x^{i}\partial y^{j}}\|\le C$ with $0\le i+j\le 8$ and $\|\frac{\partial^{j}u(t)}{\partial t^{j}}\|\le C  (1+t^{\alpha - j})$, $0\le j\le 2$, for some positive constant $C$. In contrast, the semidiscrete finite element approximation in this work assumes a lower regularity on the solution $u(t)$ and includes the case $u_0\in L^2(\Omega)$.

\end{itemize}
Several numerical methods to \eqref{bhmaineqn} and its nonlinear variants with {\it simply supported } boundary conditions (that is,  $u=\Delta u=0$) are studied in literature with stability and convergence analysis. These methods include, for example, Adomian decomposition method \cite{jafarietal08}, finite difference method \cite{zhang_pu2017},  weak Galerkin method \cite{yazdaniweak2022}, mixed FEMs \cite{liuetal14,liuetal15, liuetal2015, huangstynes21,huangetal22},  
local discontinuous Galerkin methods \cite{weihe14, duetal2017dg}, Petrov-Galerkin method \cite{pg2020}, and spline-based methods \cite{hadhoud19, zhang2020}. 
The FEM based works in \cite{liuetal14,liuetal15, liuetal2015, huangstynes21,huangetal22} for fourth-order problems are studied for simply supported boundary conditions. An auxiliary variable $v=\Delta u$ is introduced for the mixed formulation, and subsequently, the problem reduces to a system of two second-order equations with boundary conditions $u=v=0$.

\section{Well-posedness and regularity}
We introduce function spaces and norms, and present some useful properties of the  Mittag-Leffler functions in Subsection \ref{bhpreliminaries}. The well-posedness results for (\ref{bhmaineqn}) are discussed in Subsection \ref{bhwellposedness}. In Subsection \ref{bhregularityresults}, regularity results for the solutions of both homogeneous and nonhomogeneous problems are derived. 
\subsection{Preliminaries}\label{bhpreliminaries}
\subsection*{The elliptic operator $A$}
To study the well-posedness results, we start with the solution representation. The results are based on the eigenfunction expansion of the corresponding elliptic operator associated with homogeneous boundary conditions. Consider the following eigenvalue problem 
\begin{align}\label{bheigenvluep}
\begin{aligned}
&\Delta^2\phi =\mu\phi \text{ in } \Omega,\; \phi  =\frac{\partial\phi}{\partial \nu}=0  \text{ on } \partial\Omega.
\end{aligned} 
\end{align}
Here $\nu$ is the outward unit normal to $\partial \Omega$. It is well-known \cite[p. 761]{babuskaosborn91} that there exists a family $(\mu_j,\phi_j )_{j=1}^{\infty}$ such that $(\mu_j,\phi_j)$ is solution of  (\ref{bheigenvluep}),  $0<\mu_1\le\mu_2\le\ldots$, $\mu_j\rightarrow\infty$ as $j\rightarrow\infty$, and the corresponding family of eigenfunctions $\{\phi_j\}_{j=1}^{\infty}$ is an orthonormal basis of $L^2(\Omega)$.  
For further development of this work, define the unbounded operator $(A, D(A))$ in $L^2(\Omega)$ by 
\begin{equation*}
D(A) = \{\phi\in H_0^2(\Omega) \mid \Delta^2\phi\in L^2(\Omega)\} \mbox{ and } A\phi=\Delta^2\phi  \text{ for all } \phi\in D(A).
\end{equation*}
It is known that \cite{grisvard92} 
\begin{equation}\label{index}
D(A)\subset V\cap H^{2+\gamma^{*}}(\Omega) \mbox{ with } V=H_0^2(\Omega) \text{ and } \gamma^{*}\in (1/2, 2], 
\end{equation}  
where the index $\gamma^*$ belongs to $[1,2]$ if $\Omega$ is convex and $\gamma^*\in (1/2,1)$ if $\Omega$ is not convex. The value of $\gamma^*$ can be related to the greatest angle of the polygon. Define the domain of the fractional power $A^{r},\,r\in\mathbb{R}^+$ \cite[Chapter 2]{pazybook}, of the operator $A$, by  
{
\begin{equation} \label{deftn:power}
D(A^r)=\Big\{ v = \sum_{j=1}^{\infty}v_j\phi_j \mid v_j\in{\mathbb R},\  \sum_{j=1}^{\infty} \mu_j^{2r}|v_j|^2 < \infty \Big \}, \; A^rv= \sum_{j=1}^{\infty} \mu_j^{r}v_j\phi_j =\sum_{j=1}^{\infty} \mu_j^{r}(v,\phi_j)\phi_j.  
\end{equation}}
The space $D(A^r)$ is 
equipped with the norm 
\begin{align} \label{eqn:norm}
\norm{v}_{D(A^r)}= \bigg(\sum_{j=1}^{\infty} \mu_j^{2r}|(v,\phi_j)|^2\bigg)^{1/2}.
\end{align}
The orthonormality of the basis $\{\phi_j\}_{j=1}^{\infty}$ is used above and throughout in this section. In particular, for $r=0$, we obtain $\norm[0]{\cdot}_{D(A^0)}=\norm[0]{\cdot}_{L^2(\Omega)}=\norm[0]{\cdot}$. 
When $r=1/2$, we have $D(A^{1/2})=V$.  
From \cite[Part II, Chapter 1, Proposition 6.1]{Bensoussan-etal-2007}, it follows that
\begin{equation*}
D(A^r)=[L^2(\Omega),D(A)]_r  \text{ for all } r\in[0,1]. 
\end{equation*}
Therefore, with \eqref{index} and interpolation, we can prove that  
\begin{equation*}
D(A^r)\subset V\cap H^{2+\gamma^{*}(2r-1)}(\Omega) \text{ for all } r\in[1/2,1]. 
\end{equation*}

Now, for $r>0$, let $D(A^{-r}):=(D(A^r))^{*}$ be the dual of $D(A^r)$, consisting of all bounded linear functionals on $D(A^r)$.  Identifying the dual of $L^2(\Omega)$ with itself, we write $D(A^r)\subset L^2(\Omega)\subset D(A^{-r})$. The space $D(A^{-r})$ is a Hilbert space with the norm
\begin{align}\label{bhdualspacenorm}
\norm{g}_{D(A^{-r})}= 
\bigg( \sum_{j=1}^{\infty} \frac{1}{\mu_j^{2r}}|\langle g,\phi_j\rangle|^2 \bigg)^{1/2},
\end{align}
where for $g\in D(A^{-r})$ and $\phi\in D(A^{r})$, the symbol $\langle g,\phi\rangle$ stands for the duality pairing of $g$ with $\phi$. Further, with $g\in L^2(\Omega)$ and $\phi\in D(A^{r})$, we write $\langle g,\phi\rangle=(g,\phi)$ (cf. \cite[Chapter V]{brezis}). 

\begin{rem}\label{index_remark}

   In addition to the index $\gamma^*$ introduced in \eqref{index}, we need to introduce the so-called regularity index (as defined in \cite{brennersung2005} or in \cite{cn2022}) which is an index $\gamma_0\in (1/2,1]$ for which the operator $A$ is an isomorphism from $H^{2+\gamma_0}\cap V$ into $H^{-2+\gamma_0}$. If $\Omega$ is convex, $\gamma_0=1$, while if $\Omega$ is not convex we have $\gamma_0\in (1/2,1)$. If $\Omega$ is not convex, we can choose $\gamma_0$ and $\gamma^*$ such that $\gamma_0=\gamma^*$, and the value of $\gamma_0$ can be related to the greatest angle of the polygon.
   
\end{rem}
Let $a(\cdot,\cdot):V\times V\rightarrow\mathbb{R}$ be the bilinear form associated with $\Delta^2$ defined by 
\begin{align*}
\begin{aligned}
a(v ,w)&=\int_{\Omega}D^2v:D^2w\,dx
=\int_{\Omega}\sum_{i,j=1}^{2}\frac{\partial ^2 v}{\partial x_{i}\partial  x_{j}}\frac{\partial ^2 w}{\partial x_{i}\partial  x_{j}}\,dx, 
\end{aligned}
\end{align*}
where $D^2v$ denotes the Hessian matrix of $v$. Then, we can verify that
\begin{align*}
(Av,w)=a(v,w) \text{ for all } v\in D(A) \text{ and for all } w\in V.    
\end{align*}
{In the convergence analysis of the semidiscrete scheme, we need the following version of Gronwall's lemma (cf. \cite[Lemma 2.1]{chenshih98}).
\begin{lemma}[Gronwall's lemma]\label{gronwalllem}
Assume that $\phi,\psi$, and $\chi$ are three non-negative integrable functions on $[0,T]$. If $\phi$ satisfies 
\begin{align*}
\phi(t)\le \psi(t)+\int_{0}^{t}\chi(s)\phi(s)\,ds \text{ for }   t\in (0,T), 
\end{align*}
then,
\begin{align*}
\phi(t)\le \psi(t)+\int_{0}^{t}\psi(s)\chi(s)e^{\int_{s}^{t}\chi(\tau)\,d\tau }\,ds \text{ for }   t\in (0,T).
\end{align*}
\end{lemma}}
\subsection*{Mittag-Leffler functions and their properties} 
{The Mittag-Leffler functions play a very important role in the analysis of fractional differential equations.} 
The two-parameter Mittag-Leffler function $E_{\alpha,\beta}(\cdot)$ is defined \cite[p. 42]{kilbasbook} by 
\begin{align*}
E_{\alpha,\beta}(z)=\sum_{j=0}^{\infty}\frac{z^j}{\Gamma(j\alpha+\beta)}, \;\; z\in \mathbb{C},\;\; \alpha>0,\, \beta\in \mathbb{R}.
\end{align*}
{It is a generalization of the exponential function $e^{z}$ in the sense that $E_{1,1}(z)=e^{z}$. Further, we can directly verify that it is an entire function in $z$.} The following properties of the Mittag-Leffler function $E_{\alpha,\beta}(\cdot)$ \cite[(1.8.28)]{kilbasbook} or \cite[Theorem 1.4]{podlubny}, and \cite[Lemmas 3.2 and 3.3]{sy11} are essential in Subsections \ref{bhwellposedness} and \ref{bhregularityresults}.
\begin{lemma}[Properties of Mittag-Leffler functions]
Let $r$ be a real number with $\frac{\alpha\pi}{2}<r< \min\{ \pi, \alpha\pi\}$, where $\alpha\in (0,2)$. Then for any real number $\beta$ and $r\le |\arg z|\le \pi$,
\begin{align}\label{bhmlestmte}
|E_{{\alpha},\beta}(z)|  \lesssim  
\begin{cases}
\displaystyle{\frac{1}{1+|z|^2}},\;\; (\beta-\alpha)\in \mathbb{Z}^{-}\cup \{ 0\},\\ 
\displaystyle{\frac{1}{1+|z|}},\;\; \text{otherwise}. 
\end{cases}
\end{align}
Further, for $\alpha,\mu>0$ and $ k \in \mathbb{N}$,
\begin{align}\label{bhmlderivties}
\begin{aligned}
\frac{d^k}{dt^k}E_{\alpha,1}(-\mu t^{\alpha})&= -\mu t^{\alpha-k}E_{\alpha,\alpha-k+1}(-\mu t^{\alpha}),\;\; t\in (0,T],\\
E_{\alpha,\alpha}(-\mu)& \ge 0,\;\; \alpha\in (0,1), \; \mu\ge 0.
\end{aligned}
\end{align}
\end{lemma}
\subsection{Well-posedness}\label{bhwellposedness}
\begin{definition}[Weak solution]\label{weak-sol}
Let us assume that $u_0\in L^2(\Omega)$ and $f\in W^{1,1}(0,T;V^*)$. A function $u\in L^1(0,T;V)\cap W^{1,1}(0,T;V^{*})$ is called a weak solution to (\ref{bhmaineqn}) if
the following identities hold. 
\begin{align}
& \partial_t^{\alpha} \langle u(t),v \rangle +a(u(t),v)=\langle f(t),v\rangle \text{ in } L^1(0,T) \text{ for all } v \in V, \label{bhvf} \\ 
& \text{and } \nonumber \\
& \lim_{t\rightarrow 0}\norm[0]{u(t)-u_0}_{V^*}=0. \nonumber
\end{align}
\end{definition}
\begin{remark}
Since we have
\begin{align*}
\|\partial_t^{\alpha}\langle u(t),v \rangle\|_{L^1(0,T)}&= \|\int_{0}^{t}\frac{(t-s)^{-\alpha}}{\Gamma(1-\alpha)}\langle u'(s),v \rangle\,ds\|_{L^1(0,T)}\lesssim \|\langle u'(\cdot),v \rangle \|_{L^1(0,T)} \\
&\lesssim \|u'\|_{L^1(0,T;V^*)}\|v\|_{V} \text{ for all } v\in V,
\end{align*}
we can define $\partial_t^{\alpha} u$ in ${L^1(0,T;V^*)}$ by the equation 
\begin{align*}
\langle \partial_t^{\alpha}u(t),v \rangle:= \partial_t^{\alpha}\langle u(t),v \rangle \text{ for all } v\in V \text{ and all } t\in(0,T].   
\end{align*}
\end{remark}
\begin{theorem}[Well-posedness of (\ref{bhmaineqn})]\label{bhwellposedness-thm}
Assume that $u_0\in L^2(\Omega)$ and $f\in W^{1,1}(0,T;V^*)$. Then, \eqref{bhmaineqn} admits a unique weak solution determined by
\begin{align}\label{bhctssoln}
u(t)=E(t)u_0+\int_{0}^{t}F(t-s)f(s)\,ds,
\end{align}
with the operators $E(t) \in {\mathcal L}(L^2(\Omega))$ and $F(t) \in {\mathcal L}{(V^*)}$ defined by 
\begin{align}\label{bhsolnopertrs}
\begin{aligned}
E(t)u_0&=\sum_{j=1}^{\infty}(u_0,\phi_j)E_{\alpha,1}(-\mu_jt^{\alpha})\phi_j(x),\\
F(t)f(\cdot)&=\sum_{j=1}^{\infty}\langle f(\cdot),\phi_j\rangle t^{\alpha-1}E_{\alpha,\alpha}(-\mu_jt^{\alpha})\phi_j(x), \;\; t\in(0,T].
\end{aligned}
\end{align}
Moreover, we have
\begin{align}\label{bhweasolest}
\|u\|_{L^1(0,T;V)}+\|u'\|_{L^1(0,T;V^*)}\lesssim  \|u_0\| +  \|f\|_{W^{1,1}(0,T;V^*)}.  
\end{align}
\end{theorem}
The following stability properties of the solution operator $F(t)$ of the nonhomogeneous problem are essential in the proof of Theorems \ref{bhwellposedness-thm} and \ref{bhnonhomop_reg}. For related results of second-order time-fractional differential equations, we refer to \cite[Lemma 2.2]{jlpz15}.
\begin{lemma}[Stability]\label{bhnonhomosloptr_stbt}
For $q\ge -1$, $t\in(0,T]$, and $v\in D(A^{q/2})$, we have
\begin{align*}
\norm[0]{F(t)v}_{D(A^{p/2})}\lesssim t^{-1+\big(1+\frac{(q-p)}{2}\big)\alpha}\norm[0]{v}_{D(A^{q/2})},\;\; 0\le p-q\le 4.
\end{align*}
\end{lemma}
\begin{proof}
In view of (\ref{bhsolnopertrs}), \eqref{eqn:norm} (or \eqref{bhdualspacenorm}), and (\ref{bhmlestmte}), we obtain
\begin{align*}
\norm[0]{F(t)v}_{D(A^{p/2})}^2&=\sum_{j=1}^{\infty}\mu_j^p t^{2\alpha-2}|E_{\alpha,\alpha}(-\mu_jt^{\alpha})|^2| \langle v,\phi_j \rangle|^2\lesssim t^{2\alpha-2} \sum_{j=1}^{\infty}\frac{\mu_j^p}{(1+(\mu_jt^{\alpha})^2)^2} | \langle v,\phi_j \rangle |^2\\
&= t^{2\alpha-2+(q-p)\alpha}\sum_{j=1}^{\infty}\frac{(\mu_jt^{\alpha})^{p-q}}{(1+(\mu_jt^{\alpha})^2)^2}\mu_j^q| \langle v,\phi_j \rangle |^2\le t^{2\alpha-2+(q-p)\alpha}\norm[0]{v}_{D(A^{q/2})}^2,  
\end{align*}
where we have used $\frac{(\mu_jt^{\alpha})^{p-q}}{(1+(\mu_jt^{\alpha})^2)^2}\le 1$ for $0\le p-q\le 4$ and $j\in \mathbb{N}$. This concludes the proof. 
\end{proof}

\noindent\textit{Proof of Theorem \ref{bhwellposedness-thm}.}   \textit{Step 1 (Representation for solution).}  If the function $u$ is a solution to (\ref{bhmaineqn}) in the sense of Definition \ref{weak-sol}, then for all $j\in {\mathbb N}$, $(u(\cdot),\phi_j)$ satisfies 
\begin{align}\label{bhnonhomosol_efn}
\begin{aligned}
\partial_t^{\alpha} \langle u(t),\phi_j\rangle +\mu_j(u(t),\phi_j) &= \langle f(t),\phi_j\rangle \text{ for all }  t\in(0,T],
\\  \lim_{t \rightarrow 0} (u(t),\phi_j)&=(u_0,\phi_j).
\end{aligned}
\end{align}
It is well-known that for all $j\in {\mathbb N}$ \cite[p. 242]{gorenflomainardi97} (see also \cite{podlubny, kilbasbook,sy11}),  \eqref{bhnonhomosol_efn} has a unique solution $(u(\cdot),\phi_j)\in C([0,T])$ defined by \begin{equation*}
(u(t),\phi_j) = (u_0,\phi_j)E_{\alpha,1}(-\mu_jt^{\alpha})+ \int_{0}^{t} \langle f(s),\phi_j\rangle (t-s)^{\alpha-1}E_{\alpha,\alpha}(-\mu_j(t-s)^{\alpha})\,ds.
\end{equation*} 
Let us set 
\begin{align*}
\displaystyle u(t)& =\sum_{j=1}^{\infty}(u(t),\phi_j)\phi_j(x),  \; \displaystyle v(t)=\sum_{j=1}^{\infty}(u_0,\phi_j)E_{\alpha,1}(-\mu_jt^{\alpha})\phi_j(x), \\
\text{and } \displaystyle w(t)& =\sum_{j=1}^{\infty}\Big(\int_{0}^{t}\langle f(s),\phi_j\rangle (t-s)^{\alpha-1}E_{\alpha,\alpha}(-\mu_j(t-s)^{\alpha})\,ds\Big)\phi_j(x).
\end{align*}
In the next steps, we show that $u(t)=v(t)+w(t)$ is indeed a weak solution. 

\noindent\textit{Step 2 (Proofs of $v,w\in L^1(0,T;V)$).}   From \eqref{deftn:power}, the definition of $v(t)$,  and the second inequality in (\ref{bhmlestmte}) we obtain
\begin{align*}
\norm[0]{A^{1/2}v(t)}^2=\sum_{j=1}^{\infty}|\mu_j^{1/2}(u_0,\phi_j)E_{\alpha,1}(-\mu_jt^{\alpha})|^{2}\lesssim t^{-\alpha}\sum_{j=1}^{\infty}|(u_0,\phi_j)|^2\frac{\mu_jt^{\alpha}}{(1+\mu_jt^{\alpha})^2}\le t^{-\alpha}\norm[0]{u_0}^2,
\end{align*}
with $\frac{\mu_jt^{\alpha}}{(1+\mu_jt^{\alpha})^2} \le 1$ and \eqref{eqn:norm} in the last step above. This and $D(A^{1/2})=V$ show that $\norm[0]{v(t)}_V\lesssim t^{-\alpha/2}\| u_0\|$ and hence $v\in L^1(0,T;V)$. 

\noindent Next, to show $w\in L^1(0,T;V)$, we first invoke the last inequality and then the first identity in (\ref{bhmlderivties}) to arrive at
\begin{align}\label{bhnonhomost_1}
\begin{aligned}
\int_{0}^{\eta}|t^{\alpha-1}E_{\alpha,\alpha}(-\mu_jt^{\alpha})|\,dt &= \int_{0}^{\eta}t^{\alpha-1}E_{\alpha,\alpha}(-\mu_jt^{\alpha})\,dt=-\frac{1}{\mu_j}\int_{0}^{\eta}\frac{d}{dt}E_{\alpha,1}(-\mu_jt^{\alpha})\,dt\\
&= \frac{1}{\mu_j}(1-E_{\alpha,1}(-\mu_j\eta^{\alpha})),\;\; \eta >0.
\end{aligned}
\end{align} 
Since $W^{1,1}(0,T;V^*)\subset L^{\infty}(0,T; V^*)$,   \eqref{deftn:power}, the definition of $w(t)$, and  \eqref{bhnonhomost_1} show
\begin{align*}
\begin{aligned}
&\norm[0]{A^{1/2}w(t)}^2=\sum_{j=1}^{\infty}\Big|\int_{0}^{t}\langle f(s),\phi_j\rangle\mu_j^{1/2}(t-s)^{\alpha-1}E_{\alpha,\alpha}(-\mu_j(t-s)^{\alpha})\,ds\Big|^2\\
&\lesssim
\sum_{j=1}^{\infty}\frac{1}{\mu_j}\sup_{0\le \tau\le T}|\langle f(\tau),\phi_j\rangle|^2\Big|\int_{0}^{t}\mu_j(t-s)^{\alpha-1}E_{\alpha,\alpha}(-\mu_j(t-s)^{\alpha})\,ds\Big|^2\\
&= \sum_{j=1}^{\infty}\frac{1}{\mu_j}\sup_{0\le \tau\le T}|\langle f(\tau),\phi_j\rangle|^2\big(1-E_{\alpha,1}(-\mu_j(t-s)^{\alpha})\big)^2 \lesssim \|f\|_{L^{\infty}(0,T;V^*)}^2,
\end{aligned}
\end{align*}
with the boundedness of $E_{\alpha,1}(\cdot)$ (cf. \eqref{bhmlestmte}) in the last step. This shows $w\in L^1(0,T; V)$.

\noindent\textit{Step 3 (Proof of $v\in W^{1,1}(0,T;V^{*})$).}  The first equality in (\ref{bhmlderivties}) leads to
\begin{align*}
v'(t)= - t^{\alpha-1}\sum_{j=1}^{\infty}\mu_j(u_0,\phi_j)E_{\alpha,\alpha}(-\mu_jt^{\alpha})\phi_j(x). 
\end{align*}
This with (\ref{bhdualspacenorm}) and (\ref{bhmlestmte}) show
\begin{align*}
\|v'(t)\|_{V^{*}}^2 =\sum_{j=1}^{\infty}\frac{t^{2\alpha-2}}{\mu_j}\mu_j^2|E_{\alpha,\alpha}(-\mu_jt^{\alpha})|^2|(u_0,\phi_j)|^2\lesssim t^{\alpha-2}\sum_{j=1}^{\infty}  \frac{\mu_jt^{\alpha}}{(1+(\mu_jt^{\alpha})^2)^2}|(u_0,\phi_j)|^2. 
\end{align*}
Since $\frac{\mu_jt^{\alpha}}{(1+(\mu_jt^{\alpha})^2)} \le 1$, the last displayed inequality and  \eqref{eqn:norm} reveal $\|v'(t)\|_{V^*}\lesssim t^{\alpha/2-1}\|u_0\|$, and hence $v\in W^{1,1}(0,T;V^{*})$.

\noindent {\it Step 4 (Proof of $\displaystyle w(t)=\int_{0}^t F(t-s) f(s) \: ds $)}. 
Set $$S_N(t,s)=\sum_{j=1}^{N}\langle f(s),\phi_j\rangle (t-s)^{\alpha-1}E_{\alpha,\alpha}(-\mu_j(t-s)^{\alpha})\phi_j(x), \; t>s.$$ 
An application of Lemma \ref{bhnonhomosloptr_stbt} with $p=-1=q$ yields  
\begin{align*}
&\|S_N(t,s)\|_{V^*}^2= \sum_{j=1}^{N}\frac{1}{\mu_j}|\langle f(s),\phi_j\rangle (t-s)^{\alpha-1}E_{\alpha,\alpha}(-\mu_j(t-s)^{\alpha})|^2\lesssim  (t-s)^{2\alpha-2}\|f\|_{L^{\infty}(0,T;V^*)}^2,
\end{align*}
which implies $\|S_N(t,s)\|_{V^*}\lesssim (t-s)^{\alpha-1}\|f\|_{L^{\infty}(0,T;V^*)}$ for all $N\in \mathbb{N}$. In view of the Lebesgue-B{\"o}chner dominated convergence theorem, we obtain
\begin{align*}
\lim_{N\rightarrow\infty}\int_{0}^{t}S_N(t,s)\,ds= \int_{0}^{t}\lim_{N\rightarrow\infty}S_N(t,s)\,ds,
\end{align*}
and this concludes the proof.

\noindent {\it Step 5 (Proof of $w\in W^{1,1}(0,T;V^{*})$)}.
A differentiation of $w$ with respect to $t$ yields 
\begin{align}\label{bhnonhomofstdrvt}
\begin{aligned}
w'(t)=\frac{d}{dt}\Big(\int_{0}^{t}F(s)f(t-s)\,ds\Big)&=\int_{0}^{t}F(s)f'(t-s)\,ds+F(t)f(0).
\end{aligned}
\end{align} 
That is,
\begin{align*}
\begin{aligned}
\|w'(t)\|_{V^*}&\le \int_{0}^{t}\|F(s)f'(t-s)\|_{V^*}\,ds+\|F(t)f(0)\|_{V^*}\\
&\lesssim \int_{0}^{t}s^{\alpha-1}\|f'(t-s)\|_{V^*}\,ds+t^{\alpha-1}\|f(0)\|_{V^*},
\end{aligned}   
\end{align*}
with a use of Lemma \ref{bhnonhomosloptr_stbt} for $p=-1=q$ in the last step above.
Since $f\in W^{1,1}(0,T;V^{*})$, from Young's inequality for the convolution product to bound the first term followed by an application of $t\le T$ for both the terms, we obtain 
\begin{align*}
\int_{0}^{T}\|w'(t)\|_{V^*}\,dt \lesssim  \|f'\|_{L^1(0,T;V^*)}+\|f(0)\|_{V^*}.  
\end{align*}
This establishes $w\in W^{1,1}(0,T;V^{*})$. 

\noindent\textit{Step 6 (Proof of $\lim_{t\rightarrow 0}\norm[0]{u(t)-u_0}_{V^*}=0$).}  Since $w\in W^{1,1}(0,T;V^{*})$,  $w$ belongs to $C([0,T];V^{*})$. Moreover, $w(0)=0$. Therefore, 
\begin{align*}
\lim_{t\rightarrow 0}\|w(t)\|_{V^*}=0.    
\end{align*}
Next, we show that   
\begin{align}\label{bhhomosolzero}
\lim_{t\rightarrow 0}\norm[0]{v(t)-u_0}=0.
\end{align}
Employ the definition of  $v(t)$ and \eqref{deftn:power} for $r=0$ to obtain
\begin{align*}
\norm[0]{v(t)-u_0}^2=\sum_{j=1}^{\infty}|(u_0,\phi_j)|^2\big(E_{\alpha,1}(-\mu_jt^{\alpha})-1\big)^{2}. 
\end{align*}
Also, for each $j\in \mathbb{N}$, $\lim_{t\rightarrow 0}\big(E_{\alpha,1}(-\mu_jt^{\alpha})-1\big)=0$.  Moreover, \eqref{bhmlestmte} applied to the first inequality below shows, for $0\le t\le T$,
\begin{align}\label{bhhomosolzero_1}
\sum_{j=1}^{\infty}|(u_0,\phi_j)|^2\big(E_{\alpha,1}(-\mu_jt^{\alpha})-1\big)^{2}\lesssim \sum_{j=1}^{\infty}\Big( 1+\frac{1}{(1+\mu_jt^{\alpha})^2}\Big)|(u_0,\phi_j)|^2\le 2 \|u_0\|^2<\infty.
\end{align}
Using \eqref{bhmlestmte}, we have 
\begin{align*}
|(u_0,\phi_j)|^2\big(E_{\alpha,1}(-\mu_jt^{\alpha})-1\big)^{2}\lesssim |(u_0,\phi_j)|^2 \text{ for all } j\in \mathbb{N}.  
\end{align*}
Since $u_0\in L^2(\Omega)$, the series on the left-hand side of \eqref{bhhomosolzero_1} is normally convergent, and so  \eqref{bhhomosolzero} is proved.

\noindent Since $u=v+w$, a combination of the limits obtained for $v$ and $w$ leads to
\begin{align*}
\lim_{t\rightarrow 0}\|u(t)-u_0\|_{V^*}\le \lim_{t\rightarrow 0}\|v(t)-u_0\|_{V^*}+\lim_{t\rightarrow 0}\|w(t)\|_{V^*}=0.    
\end{align*}
Finally, utilizing \eqref{bhctssoln} and \eqref{bhweasolest}, we can show \eqref{bhvf}.  Thus, we have shown that $u$ defined by \eqref{bhctssoln} satisfies all the conditions of a weak solution.   \hfill{$\Box$}
\subsection{Regularity results}\label{bhregularityresults}
The following regularity results for the solutions of both homogeneous and nonhomogeneous problems are employed in the error analysis. {The proofs of these results are based on the solution representation in \eqref{bhctssoln}, the properties of Mittag-Leffler functions in  \eqref{bhmlestmte} and \eqref{bhmlderivties}, and Lemma \ref{bhnonhomosloptr_stbt}  }.
\begin{theorem}[Regularity for homogeneous case]\label{bhhomop_reg}
Let $u$ be the solution of \eqref{bhmaineqn} with $f=0$. Then, for $t\in(0,T]$, the following results hold true:
\begin{itemize}
\item [(i)] (Nonsmooth initial data) For $u_0\in L^2(\Omega)$, 
we have
\begin{align}
\norm[0]{u^{(i)}(t)}_{D(A^p)}&\lesssim  t^{-(i+\alpha p)}\norm[0]{u_0} \text{ for }  i \in \{ 0,1 \}, 0\le p\le 1, \label{bhhomopfn_func_ns}\\
\norm[0]{u''(t)}&\lesssim 
t^{-2}\norm[0]{u_0}, \label{bhhomopfn_ddtve_ns}\\
\|\partial_t^{\alpha}u(t)\|&\lesssim t^{-\alpha}\|u_0\|. \label{caputo_nsest}
\end{align}
In addition, $\|u\|_{C([0,T];L^2(\Omega))}\lesssim \|u_0\|$.
\item [(ii)] (Smooth initial data) For $u_0\in D(A)$, 
\begin{align}
\norm[0]{u(t)}_{D(A)}&\lesssim \norm[0]{u_0}_{D(A)}  \label{bhhomopfn_func_s}\\
\norm[0]{u'(t)}_{D(A^p)}&\lesssim  t^{-(1+\alpha(p-1))}\norm[0]{u_0}_{D(A)}\text{ for }  0\le p\le 1,  \label{bhhomopfn_dtve_s}\\
\norm[0]{u''(t)}&\lesssim 
t^{\alpha-2}\norm[0]{u_0}_{D(A)}, \label{bhhomopfn_ddtve_s}\\
\|\partial_t^{\alpha}u(t)\|& \lesssim  \|u_0\|_{D(A)}. \label{caputo_sest}
\end{align}
Furthermore,  $\|u\|_{C([0,T];D(A))}\lesssim \|u_0\|_{D(A)}$.
\end{itemize}
\end{theorem}
\noindent{\it{ Proof of case (i).}}
In view of \eqref{deftn:power}, (\ref{bhctssoln}) for $f=0$, and  (\ref{bhmlestmte}) it is easy to see that for $p=0,1$,
\begin{align*}
\norm[0]{A^{p}u(t)}^2=\sum_{j=1}^{\infty}|\mu_j^{p}(u_0,\phi_j)E_{\alpha,1}(-\mu_jt^{\alpha})|^{2}\lesssim \sum_{j=1}^{\infty}|(u_0,\phi_j)|^2\frac{\mu_j^{2p}}{(1+\mu_jt^{\alpha})^2}\le t^{-2p\alpha}\norm[0]{u_0}^2.
\end{align*}
Thus for $i=0$, \eqref{bhhomopfn_func_ns} follows  for  $0\le p\le 1$ by means of interpolation.

\noindent The first identity in (\ref{bhmlderivties}) with $k=1$ and (\ref{bhctssoln}) with $f=0$ lead to 
\begin{align*}
u'(t)= - t^{\alpha-1}\sum_{j=1}^{\infty}\mu_j(u_0,\phi_j)E_{\alpha,\alpha}(-\mu_jt^{\alpha})\phi_j(x). 
\end{align*}
In the next two bounds we utilize $\frac{(\mu_jt^{\alpha})^{i-l}}{(1+(\mu_jt^{\alpha})^2)^2}\le 1$ for $0\le i-l\le 4$ and $j\in \mathbb{N}$. The estimate in (\ref{bhmlestmte}) and the last displayed equality lead to 
\begin{align*}
\norm[0]{u'(t)}^2&
\lesssim 
t^{-2}\sum_{j=1}^{\infty}|(u_0,\phi_j)|^2\frac{(\mu_jt^{\alpha})^2}{(1+(\mu_jt^{\alpha})^2)^2}
\le
t^{-2}\sum_{j=1}^{\infty}(u_0,\phi_j)^2=t^{-2}\norm[0]{u_0}^2.
\end{align*}

\noindent Analogous arguments  for the case $i=p=1$ in \eqref{bhhomopfn_func_ns} reveal
\begin{align*}
\norm[0]{Au'(t)}^2&\lesssim t^{2\alpha-2}\sum_{j=1}^{\infty}|(u_0,\phi_j)|^2 \frac{\mu_j^4}{(1+(\mu_jt^{\alpha})^2)^2}= 
t^{-2(1+\alpha)}\sum_{j=1}^{\infty}|(u_0,\phi_j)|^2\frac{(\mu_jt^{\alpha})^4}{(1+(\mu_jt^{\alpha})^2)^2}\\
&\le t^{-2(1+\alpha)}\sum_{j=1}^{\infty}|(u_0,\phi_j)|^2=t^{-2(1+\alpha)}\norm[0]{u_0}^2.
\end{align*}
 The above displayed estimates show $\|u'(t)\|\lesssim t^{-1}\|u_0\|$ and $\norm[0]{Au'(t)} \lesssim
 t^{-(1+\alpha)}\norm[0]{u_0}$. An interpolation argument  completes the proof of \eqref{bhhomopfn_func_ns}.

\medskip \noindent  The first identity in (\ref{bhmlderivties}) with $k=2$  and \eqref{bhctssoln} with $f=0$ show 
\begin{align*}
u''(t)= - t^{\alpha-2}\sum_{j=1}^{\infty}\mu_j(u_0,\phi_j)E_{\alpha,\alpha-1}(-\mu_jt^{\alpha})\phi_j(x). 
\end{align*}
Now, proceed as in the case for $i=1, p=0$ in \eqref{bhhomopfn_func_ns} to obtain  (\ref{bhhomopfn_ddtve_ns}). To prove \eqref{caputo_nsest}, 
 use   \eqref{bhmaineqn} and \eqref{bhhomopfn_func_ns} with $i=0, p=1$ to conclude 
\begin{align*}
\|\partial_t^{\alpha}u(t)\|\le \|Au(t)\|\lesssim t^{-\alpha}\|u_0\|.    
\end{align*}

\noindent Proceeding as in \textit{Step 6} in the proof of Theorem \ref{bhwellposedness-thm}, we can show that $u\in C([0,T];L^2(\Omega))$, and this completes the proof of {$(i)$}.  \qed

\noindent{{\it Proof of  (ii).}} Analogous arguments as in the proof of $(i)$, but now with $u_0\in D(A)$, lead to 
\begin{align*}
&\|u(t)\|^2_{D(A)}=\sum_{j=1}^{\infty}|\mu_j(u_0,\phi_j)E_{\alpha,1}(-\mu_jt^{\alpha})|^{2}\lesssim \sum_{j=1}^{\infty}\mu_j^{2}|(u_0,\phi_j)|^2\frac{1}{(1+\mu_jt^{\alpha})^2}\le \norm[0]{u_0}_{D(A)}^2, \\
& \norm[0]{u'(t)}^2 \lesssim t^{2\alpha-2}\sum_{j=1}^{\infty}\mu_j^2|(u_0,\phi_j)|^2 \frac{1}{(1+(\mu_jt^{\alpha})^2)^2}\le t^{2\alpha-2}\norm[0]{u_0}_{D(A)}^2, \\
\text{ and} & \norm[0]{Au'(t)}^2\lesssim t^{-2}\sum_{j=1}^{\infty}\mu_j^2|(u_0,\phi_j)|^2 \frac{(\mu_jt^{\alpha})^2}{(1+(\mu_jt^{\alpha})^2)^2}
\le t^{-2}\norm[0]{u_0}_{D(A)}^2.
\end{align*}
The last three displayed inequalities and an interpolation conclude the proof of \eqref{bhhomopfn_func_s}-\eqref{bhhomopfn_dtve_s}. 
The proof of \eqref{bhhomopfn_ddtve_s} is similar to that of \eqref{bhhomopfn_ddtve_ns} and is skipped. 
The proof of \eqref{caputo_sest} is straightforward from \eqref{bhmaineqn} and \eqref{bhhomopfn_func_s}. 
Following the \textit{Step 6} of Theorem \ref{bhwellposedness-thm}, we can show that $u\in C([0,T];D(A))$, and this completes the proof of  $(ii)$. \hfill{$\Box$}
\begin{remark}
In view of the estimates in Theorem \ref{bhhomop_reg} $(i)$ and $(ii)$, by means of interpolation we obtain for time $t\in(0,T]$,
\begin{align*}
\|u(t)\|_{D(A^p)}&\lesssim t^{-\alpha(p-q)}\|u_0\|_{D(A^q)} \text{ for } 0\le q\le p\le 1,  \\
\|u'(t)\|_{D(A^p)}&\lesssim t^{-(1+\alpha(p-q))}\|u_0\|_{D(A^q)}  \text{ for } 0\le p, q\le 1,\\
\|u''(t)\|&\lesssim t^{q\alpha-2}\|u_0\|_{D(A^q)} \text{ for } 0\le q\le 1, \text{ and},\\
\|\partial_t^{\alpha}u(t)\|&\lesssim t^{-(1-q)\alpha}\|u_0\|_{D(A^q)} \text{ for } 0\le q\le 1.
\end{align*}
\end{remark}
\begin{theorem}[Regularity for nonhomogeneous case with zero initial data]\label{bhnonhomop_reg}
Let $u$ be the solution of (\ref{bhmaineqn}) with $u_0=0$ and $f\in W^{1,\infty}([0,T];D(A^{q/2}))\cap W^{2,1}(0,T;L^2(\Omega))$, $q\in[-1,1]$. Then, for all $\epsilon\in (0,1)$, we have 
\begin{align}
& \norm[0]{u(t)}_{D(A^{q/2+1-\epsilon/2})} \lesssim \epsilon^{-1}t^{\alpha\epsilon/2}\norm[0]{f}_{L^{\infty}(0,t;D(A^{q/2}))},\;\; t\in[0,T], \label{bhnonhomo_fnest}\\
& \norm[0]{u'(t)}_{D(A^{q/2+1-\epsilon/2})}\lesssim  \epsilon^{-1}t^{\alpha\epsilon/2}\norm[0]{f'}_{L^{\infty}(0,t;D(A^{q/2}))}+ t^{-1+\alpha\epsilon/2}\norm[0]{f(0)}_{D(A^{q/2})}, \;\; t\in(0,T],\label{bhnonhomo_dtve}\\
& \norm[0]{u''(t)}\lesssim \mathcal{I}^{\alpha}\norm[0]{f''(\cdot)}(t)+t^{\alpha-1}\norm[0]{f'(0)}+t^{\alpha-2}\norm[0]{f(0)},\;\; t\in(0,T], \label{bhnonhomo_ddtve}\\
&\|\partial_t^{\alpha}u(t)\|\lesssim \epsilon^{-1}t^{\alpha\epsilon/2}\norm[0]{f}_{L^{\infty}(0,t;D(A^{\epsilon/2}))}+\|f(t)\|, \;\; t\in[0,T]. \label{caputo_estnhprblm}
\end{align}
\end{theorem}
\noindent{\it{Proof of (\ref{bhnonhomo_fnest}).}}
The representation \eqref{bhctssoln} for $u_0=0$ and  Lemma \ref{bhnonhomosloptr_stbt} establish
\begin{align*}
& \norm[0]{u(t)}_{D(A^{q/2+1-\epsilon/2})}\le \int_{0}^{t}\norm{F(t-s)f(s)}_{D(A^{q/2+1-\epsilon/2})}\,ds\\
&\qquad \qquad \lesssim \int_{0}^{t}(t-s)^{-1+\alpha \epsilon/2}\norm[0]{f(s)}_{D(A^{q/2})}\,ds\lesssim \epsilon^{-1}t^{\alpha\epsilon/2}\norm[0]{f}_{L^{\infty}(0,t;D(A^{q/2}))}.
\end{align*}
This concludes the proof of (\ref{bhnonhomo_fnest}).

\noindent{\it{Proof of (\ref{bhnonhomo_dtve}).}}
Apply Lemma \ref{bhnonhomosloptr_stbt} with $p=q+2-\epsilon$, $-1\le q\le 1$ in (\ref{bhnonhomofstdrvt}) to yield, for positive time $t$,
\begin{align*}
\begin{aligned}
\norm[0]{u'(t)}_{D(A^{q/2+1-\epsilon/2})}&\le \int_{0}^{t}\norm[0]{F(t-s)f'(s)}_{D(A^{q/2+1-\epsilon/2})}\,ds+\norm[0]{F(t)f(0)}_{D(A^{q/2+1-\epsilon/2})}\\
&\lesssim \int_{0}^{t}(t-s)^{-1+\alpha\epsilon/2}\norm[0]{f'(s)}_{D(A^{q/2})}\,ds+t^{-1+\alpha\epsilon/2}\norm[0]{f(0)}_{D(A^{q/2})}\\
&\lesssim \epsilon^{-1}t^{\alpha\epsilon/2}\norm[0]{f'}_{L^{\infty}(0,t;D(A^{q/2}))}+t^{-1+\alpha\epsilon/2}\norm[0]{f(0)}_{D(A^{q/2})}.
\end{aligned}
\end{align*}
This concludes the proof.

\noindent{\it{Proof of (\ref{bhnonhomo_ddtve}).}}
Differentiate (\ref{bhnonhomofstdrvt}) with respect to time and use the first identity in (\ref{bhmlderivties}) with $k=2$ to obtain
\begin{align}\label{bhnonhomo2nddrvt}
\begin{aligned}
u''(t)&=\int_{0}^{t}F(t-s)f''(s)\,ds+F(t)f'(0)+F'(t)f(0)\\
&=\int_{0}^{t}F(t-s)f''(s)\,ds+F(t)f'(0)
+ t^{\alpha-2}\sum_{j=1}^{\infty}(f(0),\phi_j)E_{\alpha,\alpha-1}(-\mu_jt^{\alpha})\phi_j(x).
\end{aligned}
\end{align} 
An application of Lemma \ref{bhnonhomosloptr_stbt} with $p=q=0$ shows
\begin{align*}
\begin{aligned}
\norm[0]{\int_{0}^{t}F(t-s)f''(s)\,ds}&\lesssim \int_{0}^{t}(t-s)^{\alpha-1}\norm[0]{f''(s)}\,ds\lesssim \mathcal{I}^{\alpha}\norm[0]{f''(\cdot)}(t),\\
 \norm[0]{F(t)f'(0)}&\lesssim t^{\alpha-1}\norm[0]{f'(0)}.
\end{aligned}
\end{align*} 
To bound the last term in (\ref{bhnonhomo2nddrvt}),  proceed as the proof of the case $i=1, p=0$ in (\ref{bhhomopfn_func_ns}) to obtain 
\begin{align*}
\norm[0]{F'(t)f(0)}\lesssim t^{\alpha-2}\norm[0]{f(0)}.
\end{align*}
A combination of the last three estimates in (\ref{bhnonhomo2nddrvt}) completes the proof of (\ref{bhnonhomo_ddtve}).

\noindent{\it{Proof of (\ref{caputo_estnhprblm}).}} The proof is immediate from  \eqref{bhmaineqn} and \eqref{bhnonhomo_fnest} with $q=\epsilon$. \hfill{$\Box$}
\begin{remark}\label{regularity_rmk}
The combination of \eqref{bhhomopfn_dtve_s} with $\beta\in [1/2,1)$ and \eqref{bhnonhomo_dtve} with $q\in [2\beta-2+\epsilon,1]$ results in the following estimate for the solution $u$ of \eqref{bhmaineqn}: For $t\in(0,T]$,
\begin{align*}
\|u'(t)\|_{D(A^{\beta})}\lesssim  t^{-1+\alpha(1-\beta)}\|u_0\|_{D(A)}+ \epsilon^{-1}t^{\alpha\epsilon/2}\norm[0]{f'}_{L^{\infty}(0,t;D(A^{q/2}))}+ t^{-1+\alpha\epsilon/2}\norm[0]{f(0)}_{D(A^{q/2})}.  
\end{align*}
\end{remark}
\section{Semidiscrete scheme}
In this section, we describe the lowest-order finite element discretization schemes for the spatial variable in Subsection 3.1. The Ritz projection operator and its approximation properties are stated in Subsection 3.2.
\subsection{Lowest-order finite element discretizations}

Let $\T$ denote a shape regular triangulation of the polygonal Lipschitz domain into compact triangles.  Associate its piecewise constant mesh-size $h_\T \in P_0(\T)$ with $h_K:= h_\T|_K:= {\rm diam} (K) \approx |K|^{1/2}$ in any triangle $K \in \T$ of area $|K|$ and its maximal mesh-size $h:= {\rm max} \; h_\T$. Let ${\mathcal V}$ (resp. ${\mathcal V}(\Omega)$ or ${\mathcal V}(\partial \Omega)$) denote the set of all (resp. interior or boundary) vertices in $\T$. Let ${\mathcal E}$ (resp. ${\mathcal E}(\Omega)$ or ${\mathcal E}(\partial \Omega)$) denote the set of all (resp. interior or boundary) edges. The length of an edge $e$ is denoted by $h_e$.  
\medskip {Let the  Hilbert space  $H^m(\T)\equiv \prod_{K\in \T} H^m(K)$}. {The edge-patch $\omega(e):=\text{\rm int}(K_+\cup K_-)$ of the interior edge $e=\partial K_+\cap\partial K_-\in\E(\Omega)$ is the interior of union $K_+\cup K_-$ of the neighboring triangles $K_+$ and $K_-$;   the jump and average of $\varphi$ are defined by $\big[\!\!\big[\varphi\big]\!\!\big]:=\varphi|_{K_+}-\varphi|_{K_-}$ and  $\big\{\!\!\!\big\{\varphi\big\}\!\!\!\big\}:=\half\left(\varphi|_{K_+}+\varphi|_{K_-}\right)$ across the interior edge 
$e$ of the adjacent triangles  $K_+$ and $K_-\in\T$ in an order such that the unit normal vector  
$\nu_{K_+}|_e= \nu_e= - \nu_{K_-}|_e$ along the edge $e$
has a fixed orientation and points outside $K_+$ and inside $K_-$; $\nu_K$ is the outward unit normal of $K$ along $\partial K$.}
Further for $e\in \mathcal{E}(\partial\Omega)$, define $\omega(e):=\text{int}(K)$, and the jump and average by $\big[\!\!\big[\varphi\big]\!\!\big]:=\varphi|_e$ and $\big\{\!\!\!\big\{\varphi\big\}\!\!\!\big\}:=\varphi|_e$.
For functions in $H^2_0(\Omega)$, the notation $\trinl \cdot \trinr:=|\cdot|_{H^{2}(\Omega)}$ stands for the energy norm.  The notation $\trinl \cdot \trinr_{\text{pw}}:=|\cdot|_{H^{2}(\cT)}:=\| D^2_\text{pw}\cdot\|$ 
refers to the piecewise energy norm with the piecewise Hessian $D_\text{pw}^2$. 
Define the piecewise polynomials space $P_r(\mathcal{T})$ of degree $r\in \mathbb{N}$ by $P_r(\mathcal{T})=\{ v\in L^2(\Omega): v|_{K}\in P_r(K) \text{ for all } K \in \mathcal{T} \}$. 

\medskip
\noindent The nonconforming Morley finite element space $\text{M}(\mathcal{T})$ \cite{ciarlet78} is defined as 
\begin{align*}
\text{M}(\mathcal{T}):=\{  & \chi_h\in P_2(\mathcal{T}): \chi_h \text{ is continuous at the vertices and its normal derivative } {\partial \chi_h}/{\partial \nu} \text{ is }\\ 
& \text{ continuous at the midpoints of interior edges}, \chi_h \text{ vanishes at the vertices on }  \partial\Omega \\
& \text{ and its normal derivative }  {\partial \chi_h}/{\partial \nu} \text{ vanishes at the midpoints of boundary edges}   \}.
\end{align*}
On a finite-dimensional space $V_h\subset H^2(\mathcal{T})$,  define a mesh-dependent broken norm \cite{cgn2015}  by
\begin{align}\label{bhdisctenoms}
\begin{aligned}
\norm[0]{\chi_h}_{h}^2&=\sum_{K\in \mathcal{T}}|\chi_h|_{H^2(K)}^2 + \sum_{e\in\mathcal{E}}\sum_{z\in\mathcal{V}(e)}h_e^{-2}\big|\big[\!\!\big[   \chi_h\big]\!\!\big](z)\big|^2+ \sum_{e\in\mathcal{E}}\Big| \fint_{e} \Big[\!\!\Big[  \frac{\partial \chi_h}{\partial \nu}\Big]\!\!\Big]\,ds \Big|^2,
\end{aligned}
\end{align} 
where $\fint_{e}$ denotes the integral mean over the edge $e$. In particular, for $\chi_h \in \M(\T)$, $\|\chi_h\|_h = \trinl \chi_h \trinr_{\rm pw}$ as the jump terms in \eqref{bhdisctenoms} vanish.  Further, the discrete bilinear form $a_h:(V_h+\text{M}(\mathcal{T}))\times (V_h+\text{M}(\mathcal{T}))\rightarrow \mathbb{R}$ in all the examples in this paper has the form
\[ a_h(\cdot,\cdot) = a_{\pw}(\cdot,\cdot)+b_h(\cdot,\cdot) +c_h(\cdot,\cdot),\]
and satisfies {\bf (H)} below.
\begin{description}
\item[(H)] $a_h(\cdot,\cdot)$ is symmetric, positive-definite, and
continuous on $V_h$ with respect to the discrete norm
$\norm[0]{\cdot}_h$, i.e.,
$\exists $ constants $\beta_1, \beta_2>0$ independent of $h$ such that,
for all $w_h,\chi_h\in V_h$, $a_h(w_h,\chi_h)=a_h(\chi_h,w_h)$ and
\begin{align}\label{bhbilinearcoer_cts}
a_h(\chi_h,\chi_h)\ge \beta_1 \norm[0]{\chi_h}_{h}^2 \text{ and } 
a_h(w_h,\chi_h)\le \beta_2\norm[0]{w_h}_h\norm[0]{\chi_h}_h.
\end{align}
\end{description}
The semidiscrete problem that corresponds to the weak formulation in Definition  \ref{weak-sol} seeks $u_h\in W^{1,1}(0,T;V_h)$ such that 
\begin{equation}\label{bhsemidisscheme}
\begin{aligned}
\partial_t^{\alpha}(u_h(t),\chi_h)+a_h( u_h(t), \chi_h)&=(f(t),\chi_h)  \text{ for all } \chi_h\in V_h, 0<t\le T,\\ 
 u_h(0)&=P_hu_0 \in V_h.
\end{aligned}
\end{equation}
\nopagebreak
Here $P_h:L^2(\Omega)\rightarrow V_h$ denotes the $L^2$-projection defined by $(P_hv,\chi_h)=(v,\chi_h)$ for all $\chi_h\in V_h$.

\noindent Now we present Examples \ref{bhmorleyscheme}-\ref{bhcoipscheme} below for which the discrete bilinear form $a_h(\cdot,\cdot)$ satisfies {\bf(H)} \cite[Section 5]{cn2022}.
\begin{example}[Morley]\label{bhmorleyscheme}
In this case, $V_h:=\rm M(\mathcal{T})$, and { for all } $w_{h},\chi_{h}\in \rm M(\mathcal{T})$, 
\begin{align*}
\begin{aligned}
a_{h}(w_{h},\chi_{h}):=a_{\rm pw}(w_{h},\chi_{h}):= \int_{\Omega} D_{\pw}^2w_{h}:D_{\pw}^2\chi_{h}\,dx.
\end{aligned}
\end{align*}
The discrete Morley norm on $V_h$ is defined by $\trinl \cdot \trinr_h=\trinl \cdot \trinr_\pw= a_\pw(\cdot,\cdot)^{1/2}$. Notice that this norm is equivalent on $V_h$  to that introduced in \eqref{bhdisctenoms}. 
\end{example}
\begin{example}[dG]\label{bhdgscheme}
Choose $V_h:=P_2(\mathcal{T})$, and for all $w_h,\chi_h\in V_h$, let
\begin{align*}
\begin{aligned}
a_{h}(w_h,\chi_h)&:= a_{\rm pw}(w_{h},\chi_{h})+b_h(w_{h},\chi_{h})+c_{\rm dG}(w_{h},\chi_{h}),\\
b_h(w_{h},\chi_{h})&:=-\mathcal{J}(w_{h},\chi_{h})-\mathcal{J}(\chi_{h},w_{h}),
\mathcal{J}(w_{h},\chi_{h}):= \sum_{e\in \mathcal{E}}\int_{e} \big[\!\!\big[  \nabla w_h\big]\!\!\big]\cdot\big\{\!\!\!\big\{ D_{\pw}^2\chi_h \big\}\!\!\!\big\}\nu \,ds,\\
c_{\rm dG}(w_{h},\chi_{h})&:=\sum_{e\in\mathcal{E}}\Big(\frac{\sigma_{\rm  dG}^1}{h_e^3}\int_{e}\big[\!\!\big[  w_h\big]\!\!\big]\big[\!\!\big[  \chi_h\big]\!\!\big]\,ds+ \frac{\sigma_{\rm dG}^2}{h_e}\int_{e}\Big[\!\!\Big[  \frac{\partial w_h}{\partial \nu}\Big]\!\!\Big]\Big[\!\!\Big[  \frac{\partial \chi_h}{\partial \nu}\Big]\!\!\Big]\,ds\Big),    
\end{aligned}
\end{align*}
where $\sigma_{\rm  dG}^1$, $\sigma_{\rm  dG}^2 >0$ are the penalty parameters and
 $a_{\rm pw}(\cdot,\cdot)$ is as defined in Example \ref{bhmorleyscheme}. The dG norm $\|\cdot\|_{\rm dG}$ on $P_2(\T)$ is defined  by $\|w_h\|_{\rm dG}=\Big( \trinl w_h \trinr_\pw^2+ c_{\rm dG}(w_h,w_h)\Big)^{1/2}$, $w_h\in V_h$. {As in the previous example, this norm is equivalent on $V_h$  to that introduced in \eqref{bhdisctenoms}. }
\end{example}
\begin{example}[$C^{0}$IP]\label{bhcoipscheme}
Choose $V_h:=P_2(\mathcal{T})\cap H_0^1(\Omega)$ and for all $w_h,\chi_h\in V_h$, define
\begin{align*}
\begin{aligned}
a_{h}(w_h,\chi_h)&:= a_{\rm pw}(w_{h},\chi_{h})+b_h(w_{h},\chi_{h})+c_{\rm IP}(w_{h},\chi_{h}),\\
c_{\rm IP}(w_{h},\chi_{h})&:= \sum_{e\in\mathcal{E}}\frac{\sigma_{\rm IP}}{h_e}\int_{e}\Big[\!\!\Big[  \frac{\partial w_h}{\partial \nu}\Big]\!\!\Big]\Big[\!\!\Big[  \frac{\partial \chi_h}{\partial \nu}\Big]\!\!\Big]\,ds,
\end{aligned}
\end{align*}
where $\sigma_{\rm IP}$ is a positive parameter,  $a_{\rm pw}(\cdot,\cdot)$ and $b_h(\cdot,\cdot)$ are as defined in Examples \ref{bhmorleyscheme} and \ref{bhdgscheme}. The discrete norm $\|\cdot\|_{\rm IP}$ on the space $V_h$ reads $\|w_h\|_{\rm IP}=\Big(  \trinl w_h \trinr_\pw^2+ c_{\rm IP}(w_h,w_h) \Big)^{1/2}$ for  $w_h\in V_h$. {In this example as well, $\|\cdot\|_{\rm IP}$ is equivalent on $V_h$ to that introduced in \eqref{bhdisctenoms}. }
\end{example}
\noindent The equivalence of the common norm $\|\cdot\|_h$ with the norms defined in the Examples \ref{bhmorleyscheme}-\ref{bhcoipscheme} (see \cite{cn2022}) is helpful  in the proofs of the approximation properties of the Ritz projection in the next section. 
\subsection{Ritz projection and its approximation properties}\label{ritz}

We introduce the Ritz projection which is the elliptic projection for the biharmonic problem and state its approximation properties. The Ritz projection $\mathcal{R}_h:V\rightarrow V_h$ is defined by
\begin{align}\label{bhritzprjn}
a_h(\mathcal{R}_hv,\chi_h)=a(v, Q\chi_h) \text{ for all } \chi_h\in V_h  \text{ and } v\in V,
\end{align}
 where $Q:=JI_{\text{M}}$ is a smoother defined from $H^2(\cT)$ to $V$, with  $I_\text{M}:H^2(\cT) \rightarrow \text{M}(\mathcal{T})$ and $J:\text{M}(\mathcal{T}) \rightarrow V$ denoting  the extended Morley interpolation operator and the companion operator, respectively (see Appendix for the details of the definition and properties of the interpolation and companion operators). 
 The Lax-Milgram lemma shows that the projection $\mathcal{R}_h$ is a well-defined operator on $V$.
 
\noindent In the semidiscrete error analysis discussed in Section \ref{bhsemierrsec}, the  error $u(t)-u_h(t)$ is split by introducing the Ritz projection $\mathcal{R}_h$ as
\[
u(t)-u_h(t)= (u(t)-\mathcal{R}_hu(t))+(\mathcal{R}_hu(t)-u_h(t)).\] Given the approximation properties for 
  the Ritz projection from Lemma \ref{bhritzprjn_approx}, the main task will be to establish bounds for  $u_h(t)-\mathcal{R}_hu(t)$ in the next section.
\begin{lemma}[Approximation properties]\label{bhritzprjn_approx}
For any $v \in V$ and the Ritz projection $\mathcal{R}_h:V\rightarrow V_h $ defined in (\ref{bhritzprjn}), the approximation properties in the energy and $L^2$ norms stated below hold. 
\begin{align*}
\begin{aligned}
& \|{v-\mathcal{R}_hv}\| + h^{\gamma}\norm[0]{v-\mathcal{R}_hv}_h\lesssim h^{2\gamma}\norm[0]{v}_{H^{2+\gamma }(\Omega)}, 
\end{aligned} 
\end{align*}
for all  $\gamma\in [0,\gamma_0]$ and for all  $v\in H^{2+\gamma }(\Omega)$, where $\gamma_0\in(1/2,1]$ is the regularity index introduced in Remark \ref{index_remark}.

\end{lemma}
\noindent The proofs are available in \cite{cn2022}; we provide an outline in the appendix for continuity in reading. 
In the error analysis, we also need the following approximation property of the operator $Q$ for piecewise quadratic functions. 
\begin{lemma}\cite[Theorem 4.5 (d)]{cn2022}\label{bhcompmorlyestmte}
For any $\chi_h\in P_2(\mathcal{T})$, the  operator $Q=JI_{\M}$ 
satisfies 
\begin{align*}
\|{\chi_h -Q\chi_h}\|_{H^s(\T)} \le C_1 h^{2-s}\min_{v\in V}\norm[0]{v- \chi_h}_h \text{ for any } 0 \le s \le 2 \text{ and a constant } C_1>0.
\end{align*}
\end{lemma}
\section{Semidiscrete error estimates}\label{bhsemierrsec}
In this section, we derive error bounds for the semidiscrete scheme for both smooth and nonsmooth initial data. {To start with, in Subsection 4.1, we state some important properties of fractional integrals that are relevant in the context}. This is followed by the main results of this section and their proofs in Subsections 4.2 and 4.3.
\subsection{Properties of Riemann-Liouville fractional integrals}
For all $\alpha\in(0,\infty)$, all $\beta\in(0,\infty)$ satisfying $\alpha+\beta\ge 1$, all $v\in L^1(0,T)$, the operators
${\cal I}^{\alpha}$ and ${\cal I}^{\beta}$ (cf. \eqref{rlfracintegral}) satisfy
\begin{align}\label{bhsemigroup}
{\cal I}^{\alpha}{\cal I}^{\beta} v(t) = {\cal I}^{\alpha+\beta} v(t), \mbox{ for almost all  }t\in(0,T).    
\end{align}
If $v\in C([0,T])$, then (\ref{bhsemigroup}) is satisfied at all points  $t \in [0,T]$ and $\alpha,\beta>0$ (cf. \cite[p. 34]{samkokilbasmarichev93}).
Recall that $\kappa_{\alpha}(t):=t^{\alpha-1}/\Gamma(\alpha)$. In the rest of this subsection, we assume that $0<\alpha<1$, and so $0<1-\alpha<1$. The  three identities below hold for $v_1(t)=tv(t)$ and $v_2(t)=t^{2}v(t)$ for $t\in[0,T]$ (see \cite[Lemma 2]{karaamustaphapani18} and \cite[Lemma 2.1]{mustapha18} for a proof).  
\begin{align}
t\mathcal{I}^{\alpha}v(t)&= \mathcal{I}^{\alpha}v_1(t)+\alpha\mathcal{I}^{\alpha+1}v(t), \label{bhFracIntrule}\\
t\mathcal{I}^{\alpha}v'(t)&= \mathcal{I}^{\alpha}(v_1)'(t)+(\alpha-1)\mathcal{I}^{\alpha}v(t)-t\kappa_{\alpha}(t)v(0), \label{bhminiest1}\\
t^2\mathcal{I}^{\alpha}v'(t)&= \mathcal{I}^{\alpha}(v_2)'(t)+2(\alpha-1)\mathcal{I}^{\alpha}v_1(t)+\alpha(\alpha-1)\mathcal{I}^{\alpha+1}v(t)-t^2\kappa_{\alpha}(t)v(0). \label{bhminiest2}
\end{align}
For $\phi, v\in L^2(0,T;L^2(\Omega))$, since $\cos(\alpha \pi/2)-(1-\alpha)\ge 0$,  the following continuity property with any positive $\vartheta$ holds for $\mathcal{I}^{1-\alpha}$ \cite[Lemma 3.1 (iii)]{mustaphaschotzau14}:  
\begin{align}\label{bhcontinuity}
\int_{0}^{t}(\mathcal{I}^{1-\alpha}\phi,v)\,ds &\le  \vartheta \int_{0}^{t}(\mathcal{I}^{1-\alpha}v,v)\,ds +\frac{1}{4\vartheta(1-\alpha)^2}\int_{0}^{t}(\mathcal{I}^{1-\alpha}\phi,\phi)\,ds.     
\end{align}
If $v:[0,T]\rightarrow L^2(\Omega)$ is a piecewise continuous function in time, then $\mathcal{I}^{\alpha}$ satisfies \cite[Lemma 3.1 (ii)]{mustaphaschotzau14}
\begin{align}\label{bhpositivt}
\int_{0}^{T}(\mathcal{I}^{\alpha}v,v)\,dt&\ge \cos(\alpha \pi/2)\int_{0}^{T}\|\mathcal{I}^{\alpha/2}v\|^2\,dt\ge 0.  
\end{align}
In addition, if $v':[0,T]\rightarrow L^2(\Omega)$ is a  piecewise continuous function in time, then for $t\in(0,T]$, {it follows from \cite[Lemma 2.1]{lemcleanmustapha16} that}
\begin{align}\label{bhfuncteste}
\norm[0]{v(t)-v(0)}^2\lesssim t^{\alpha}\int_{0}^{t}\|\mathcal{I}^{(1-\alpha)/2}v'\|^2\,ds   \lesssim t^{\alpha}\int_{0}^{t}(\mathcal{I}^{1-\alpha}v',v')\,ds,
\end{align}
where the last inequality follows from  \eqref{bhpositivt}.

\noindent The proofs of this section use the inequality below frequently. For $a,b\ge 0$, it holds that 
\begin{equation}\label{algebraic}
(a+b)^2\lesssim (a^2+b^2)\lesssim (a+b)^2. \; 
\end{equation}

\subsection{Main results}
For the semidiscrete error analysis, we split the  error $u(t)-u_h(t)$ by introducing the Ritz projection $\mathcal{R}_h$ from \eqref{bhritzprjn} as
\begin{align}\label{bhsemidissplit_err}
u(t)-u_h(t)=:\rho(t)+\theta(t), 
\end{align}
with 
\[\rho(t):= u(t)-\mathcal{R}_hu(t) \text{ and }  \theta(t):= \mathcal{R}_hu(t)-u_h(t).\]
\noindent Let
\begin{align}\label{lambda0} 
\Lambda_0(\epsilon,t) :=\norm[0]{u_0}_{D(A)}+\epsilon^{-1}t^{\alpha\epsilon/2}\norm[0]{f}_{L^{\infty}(0,T;D(A^{\epsilon/2}))},  \; \epsilon\in (0,1), \; t\in[0,T]. 
\end{align}
\noindent  Recall that  $D(A) \subset V \cap H^{2+\gamma^{*}}(\Omega) \subset V \cap H^{2+\gamma_0}(\Omega)$. For the homogeneous problem, using (\ref{bhhomopfn_func_ns}) with  $i=0$, $p=1$, and for the nonhomogeneous problem, applying  \eqref{bhhomopfn_func_s} and \eqref{bhnonhomo_fnest} with $q=\epsilon$ we obtain
\begin{align*}
\|u(t)\|_{H^{2+\gamma_0}(\Omega)} &\lesssim 
\begin{cases}
t^{-\alpha}\norm[0]{u_0} \text{ for }\; u_0\in L^2(\Omega) \text{ and } f=0,\; t\in(0,T],\\
\Lambda_0(\epsilon,t) \text{ for } u_0\in D(A)\text{ and } f\ne 0,\; t\in[0,T].
\end{cases}
\end{align*}
This and Lemma \ref{bhritzprjn_approx} establish the estimates for $\rho(t)$ as given below.  
\begin{align}\label{bhritzprjn_ns_smth}
\norm[0]{\rho(t)} + h^{\gamma_0}\norm[0]{\rho(t)}_h &\lesssim 
\begin{cases}
h^{2\gamma_0}t^{-\alpha}\norm[0]{u_0} \text{ for } u_0\in L^2(\Omega) \text{ and } f=0,\; t\in(0,T],\\
h^{2\gamma_0}\Lambda_0(\epsilon,t) \text{ for }u_0\in D(A) \text{ and } f\ne 0,\; t\in[0,T].
\end{cases}
\end{align}

Hence the main task in the remaining part of this section is to bound $\theta(t)$ in the $L^2(\Omega)$ and energy norms. 
For $t\in[0,T]$ and an $\epsilon \in (0,1)$ (determined by the smoothness of $f$), set  
\begin{align}\label{bh_thetaconstnts_smth}
\begin{aligned}
\Lambda_1(\epsilon,t)&:= t^{\frac{3}{2}}\norm[0]{u_0}_{D(A)}+ t^{\frac{3}{2}-\alpha(1-\frac{\epsilon}{2})}\norm[0]{f(0)}_{D(A^{\frac{\epsilon}{2}})}
+ t\|f\|_{L^2(0,t;L^2(\Omega))} 
\\
&\qquad   +\epsilon^{-1}t^{\frac{5}{2}-\alpha(1-\frac{\epsilon}{2})}\norm[0]{f'}_{L^{\infty}(0,T;D(A^{\frac{\epsilon}{2}}))} \\
\Lambda_2(\epsilon,t)&:= t^{\frac{3}{2}}\norm[0]{u_0}_{D(A)}+B(\frac{\alpha\epsilon}{2},1-\alpha)t^{\frac{3}{2}-\alpha(1-\frac{\epsilon}{2})}\norm[0]{f(0)}_{D(A^{\frac{\epsilon}{2}})}
+ t\|f\|_{L^2(0,t;L^2(\Omega))} \\
&\qquad  +\epsilon^{-1}t^{\frac{5}{2}-\alpha(1-\frac{\epsilon}{2})}\norm[0]{f'}_{L^{\infty}(0,T;D(A^{\frac{\epsilon}{2}}))}  \\
\mathcal{B}_0(\epsilon,t)&:=\norm[0]{u_0}_{D(A)}+ \epsilon^{-1}t^{\frac{\alpha\epsilon}{2}} \|f\|_{W^{1,\infty}(0,T; D(A^{\epsilon/2}))}\\
\mathcal{B}_1 (\epsilon,t) & := \Lambda_1(\epsilon,t)+\Lambda_2(\epsilon,t)+ t^2\|f(t)\| +t\|f\|_{L^2(0,t;L^2(\Omega))}+t^2\|f'\|_{L^2(0,t;L^2(\Omega))}\\
&\qquad 
+ t^{\frac{3}{2}}(\norm[0]{f(0)}+t\norm[0]{f'(0)}) +t^{5/2}\|f''\|_{L^1(0,T;L^2(\Omega))},  
\end{aligned}
\end{align}
where $B(\cdot,\cdot)$ denotes the standard beta function. Note that in the above expression and in the sequel, whenever the norm of the function $f$ is dependent only on the space variable, we denote the dependence of $f$ on $t$ as $f(t)$; if the norm is space-time dependent, the arguments in $f$ are omitted for notational brevity.
%
\begin{theorem}[Estimates for $\theta(t)$]\label{bh_thetafnlest} Let  $u(t)$ and $u_h(t)$ solve \eqref{bhmaineqn} and \eqref{bhsemidisscheme}, respectively. Let $\mathcal{R}_hu(t) $ denote the Ritz projection of $u(t)$ defined in \eqref{bhritzprjn}. 
Then for  $\theta(t) = \mathcal{R}_hu(t)-u_h(t)$,  the estimates in $(i)$-$(ii)$ below hold.
\begin{itemize}
\item[(i)](Nonsmooth initial data)
For $u_0\in L^2(\Omega)$ and $f=0$,
\begin{align*} 
 \norm[0]{\theta(t)}+t^{\alpha/2}\|\theta(t)\|_h&\lesssim  \big(h^{2\gamma_0}t^{-\alpha} +h^2t^{-(1+\alpha)/2}\big)\norm[0]{u_0}, \;\; t\in(0,T].  
\end{align*}

\item[(ii)](Smooth initial data) For  $u_0\in D(A)$, 
$f\in W^{1,\infty}([0,T];D(A^{\epsilon/2}))\cap W^{2,1}(0,T;L^2(\Omega))$, and for all $\epsilon\in (0,1)$,
\begin{align*}
\norm[0]{\theta(t)}+t^{\alpha/2}\|\theta(t)\|_h&\lesssim   h^{2\gamma_0}{\mathcal{B}_0(\epsilon,t) }+h^2t^{\alpha/2-2}\mathcal{B}_1(\epsilon,t), \;\; t\in(0,T],   
\end{align*} 
with {$\mathcal{B}_0(\epsilon,t)$} and $\mathcal{B}_1(\epsilon,t)$ defined in \eqref{bh_thetaconstnts_smth}. 
%
\end{itemize}
\end{theorem}
\noindent A combination of the estimates for $\rho(t)$ from \eqref{bhritzprjn_ns_smth}  and $\theta(t)$ from Theorem \ref{bh_thetafnlest}, shows the error estimates for the semidiscrete scheme for both smooth and nonsmooth initial data in Theorem \ref{bhsemidisctesmooth_finl}. 
\begin{theorem}[Error estimates]\label{bhsemidisctesmooth_finl}
Let $u(t)$ and $u_h(t)$ be solutions of the continuous and semidiscrete problems in \eqref{bhmaineqn} and \eqref{bhsemidisscheme}, respectively. 
\begin{itemize}
\item[(i)] (Nonsmooth initial data) For $u_0\in L^2(\Omega)$, source function $f=0$, $u_h(0)=P_hu_0$, and $t\in(0,T]$, it holds that 
\begin{align*}
\norm[0]{u(t)-u_h(t)}&\lesssim \big( h^{2\gamma_0}t^{-\alpha} + h^2t^{-(1+\alpha)/2}\big)\norm[0]{u_0}, \\  
\norm[0]{u(t)-u_h(t)}_h&\lesssim   h^{\gamma_0}t^{-\alpha}\norm[0]{u_0}+ \big( h^{2\gamma_0}t^{-3\alpha/2} +h^2t^{-(\alpha+1/2)}\big)\norm[0]{u_0}. 
\end{align*} 

\item[(ii)] (Smooth initial data)
For $u_0\in D(A)$, $f\in W^{1,\infty}([0,T];D(A^{\epsilon/2}))\cap W^{2,1}(0,T;L^2(\Omega))$,  $u_h(0)=P_hu_0$, $t\in(0,T]$, and for all $\epsilon \in (0,1)$, it holds that 
\begin{align*}
\norm[0]{u(t)-u_h(t)}&\lesssim   
 h^{2\gamma_0}\Lambda_0(\epsilon,t) +  h^{2\gamma_0}{\mathcal{B}_0(\epsilon,t)}+h^2t^{\alpha/2-2}\mathcal{B}_1(\epsilon,t),   \\
\norm[0]{u(t)-u_h(t)}_h&\lesssim  
 h^{\gamma_0}\Lambda_0(\epsilon,t)+ h^{2\gamma_0}t^{-\alpha/2} {\mathcal{B}_0(\epsilon,t)}+h^2t^{-2}\mathcal{B}_1(\epsilon,t),   
\end{align*} 
where $\Lambda_0(\epsilon,t)$ and 
{$\mathcal{B}_0(\epsilon,t)$}, $\mathcal{B}_1(\epsilon,t)$ are defined, respectively, in \eqref{lambda0} and \eqref{bh_thetaconstnts_smth}.
\end{itemize}
\end{theorem}
\begin{proof}
The error decomposition (\ref{bhsemidissplit_err}), the triangle inequality, the first estimate in (\ref{bhritzprjn_ns_smth}), and Theorem \ref{bh_thetafnlest} $(i)$ lead to the proof of $(i)$. The proof of  $(ii)$ follows from  (\ref{bhsemidissplit_err}), the triangle inequality, the second estimate in (\ref{bhritzprjn_ns_smth}), and Theorem \ref{bh_thetafnlest} $(ii)$. 
\end{proof}
\subsection{Proof of Theorem \ref{bh_thetafnlest}}\label{bhproofsec}
  First of all, a key inequality is proved in Lemma~\ref{keylem_bhthetafinlthm} and bounds for the terms appearing in this lemma are established in Lemmas~\ref{bhlemma1}--\ref{bhlemma4}. A combination of these results establishes the proof of Theorem \ref{bh_thetafnlest}.
  
The following notations hold throughout this subsection.
\begin{align*}
\begin{aligned}
\theta_1(t)&:=t\theta(t),\; \theta_2(t):=t^{2}\theta(t),\;\; \text{and} \\
\widehat{v}(t)&:=\mathcal{I}v(t)=\int_{0}^{t}v(s)\,ds, \; \widehat{\widehat{v}}(t):=\mathcal{I}^2v(t)=\int_{0}^{t}\int_{0}^{s}v(\tau)\,d\tau\,ds.
\end{aligned}
\end{align*} 
Recall the smoother $Q$ from Subsection~\ref{ritz}. {For each $\chi_h\in V_h$, test (\ref{bhvf}) with $Q\chi_h \in V$  and subtract (\ref{bhsemidisscheme}) from (\ref{bhvf}) to obtain}
\begin{align*}
(\partial_t^{\alpha}u(t), Q\chi_h)-(\partial_t^{\alpha}u_h(t),\chi_h)+(a(u(t),Q\chi_h)-a_h(u_h(t),\chi_h))=(f(t), (Q-I)\chi_h).
\end{align*}
Add and subtract  $(\partial_t^{\alpha}u(t),\chi_h)$, utilize  (\ref{bhritzprjn}) and $\theta(t):= R_h u(t) -u_h(t)$ on the left-hand side of the above expression to obtain  
\begin{align*}
(\partial_t^{\alpha}(u(t)-u_h(t)),\chi_h)+ (\partial_t^{\alpha}u(t), (Q-I)\chi_h)+a_h(\theta(t),\chi_h)=(f(t), (Q-I)\chi_h).
\end{align*} 
Recall $u(t)-u_h(t)= \rho(t)+ \theta(t)$ and the notation from  \eqref{caputoderiv}, that yields $\partial_t^{\alpha}\theta(t)= \mathcal{I}^{1-\alpha}\theta'(t)$, $\partial_t^{\alpha}\rho(t) =\mathcal{I}^{1-\alpha}\rho'(t)$,   and $\partial_t^{\alpha}u(t) =\mathcal{I}^{1-\alpha}u'(t)$. For ease of notation in the subsequent estimates, we denote $\varphi(t):=\mathcal{I}^{1-\alpha}u'(t)$. Then for all $\chi_h \in V_h$ and $t\in(0,T]$, the above displayed identity leads to the error equation in $\theta(t)$ as
\begin{align}\label{bherreqn}
(\mathcal{I}^{1-\alpha}\theta'(t),\chi_h)+a_h(\theta(t),\chi_h)= -(\mathcal{I}^{1-\alpha}\rho'(t),\chi_h)- (\varphi(t),(Q-I)\chi_h)+(f(t),(Q-I)\chi_h).
\end{align}

\noindent Recall the notations 
\begin{align*}
& \theta_i(t)=t^i \theta(t), \; \rho_i(t)= t^i \rho(t) \text{ for }i=1,2; 
\; \varphi_2(t)=t^2 \varphi(t), \; f_2(t)=t^2 f(t), 
\widehat{\theta}(t)=\int_0^t \theta(s) \,ds,  
\\
& \widehat{\rho}(t)=\int_0^t \rho(s) \,ds, \; \rho_2'(t) =(t^2 \rho(t))', \;
\varphi_2'(t) =(t^2 \varphi(t))', \;
f_2'(t) =(t^2 f(t))', \text{ and } \\
&\lambda(t)=-\rho_2'(t)+2\alpha\rho_1(t)+\alpha(1-\alpha)\widehat{\rho}(t).
\end{align*}
\begin{lemma}[Key inequality]\label{keylem_bhthetafinlthm}
If $\theta$ satisfies (\ref{bherreqn}),  then for time $t\in(0,T]$, it holds that 
\begin{align*}
&\norm[0]{\theta(t)}^2+t^{\alpha}\norm[0]{\theta(t)}_h^2 \lesssim t^{\alpha-4}\int_{0}^{t}\Big( (\mathcal{I}^{1-\alpha}\widehat{\theta},\widehat{\theta})\,ds+(\mathcal{I}^{1-\alpha}\theta_1,\theta_1)+ (\mathcal{I}^{1-\alpha}\lambda,\lambda)\Big)\,ds\\
&+h^4t^{\alpha-4}\Big( \|\varphi_2(t)\|^2 +\int_{0}^{t}(\|\varphi_2(s)\|^2 +\|\varphi_2'(s)\|^2)\,ds+  \|f_2(t)\|^2 +\int_{0}^{t}(\|f_2(s)\|^2 + \|f_2'(s)\|^2)\,ds  \Big).
\end{align*}
\end{lemma}
\begin{remark}
    Note the dependency of the variables $\widehat{\theta}, \widehat{\theta_1}, \lambda$, etc. on the time variable is suppressed in the above and in the sequel for notational ease when there is no chance of confusion.
\end{remark}
\begin{proof} 
To prove the assertion, we first show that $\theta_2\in W^{1,1}(0,T;V_h)$ and $\varphi_2\in W^{1,1}(0,T;L^2(\Omega))$. For nonsmooth initial data, from \eqref{bhhomopfn_func_ns}, it follows that 
$$
\|u'(t)\|_{V}= \|u'(t)\|_{D(A^{1/2})}\lesssim 
t^{-(1+\alpha/2)}\|u_0\|.
$$
From which, we deduce that 
$$
\|{\mathcal R}_hu'(t)\|_h\lesssim 
t^{-(1+\alpha/2)}\|u_0\|, 
$$
and that $t\mapsto t^2 {\mathcal R}_hu(t)$ belongs to $W^{1,1}(0,T;V_h)$. We already know that $u_h$ belongs to $W^{1,1}(0,T;V_h)$. Thus $\theta_2$ also belongs to $W^{1,1}(0,T;V_h)$. With \eqref{caputo_nsest}, we have
$$
\|\varphi(t)\|\lesssim t^{-\alpha}\|u_0\|. 
$$
Thus $\varphi_2$ belongs to $W^{1,1}(0,T;L^2(\Omega))$. A similar argument holds for smooth initial data.

Multiply both sides of (\ref{bherreqn}) by $t^2$, apply (\ref{bhminiest2}) (twice) to rewrite both $t^2 \mathcal{I}^{1-\alpha}\theta'$ and $t^2 \mathcal{I}^{1-\alpha}\rho'$, 
utilize $(\rho(0)+\theta(0),\chi_h)=0$ 
from the second identity in \eqref{bhsemidisscheme},
and the definition of  $\lambda$ 
to obtain 
\begin{align*}
\begin{aligned}
{ (\mathcal{I}^{1-\alpha}\theta_2',\chi_h)+a_h(\theta_2,\chi_h)}&= 2\alpha(\mathcal{I}^{1-\alpha}\theta_1,\chi_h)+\alpha(1-\alpha)(\mathcal{I}^{1-\alpha}\widehat{\theta},\chi_h) + (\mathcal{I}^{1-\alpha}\lambda,\chi_h)\\
&\qquad -(\varphi_2,(Q-I)\chi_h)+(f_2,(Q-I)\chi_h)  \text{ for all } \chi_h \in V_h.
\end{aligned}
\end{align*} 
Substitute $\chi_h=\theta_2'(t)$ in the last displayed equation,  integrate over $(0,t)$, and apply  (\ref{bhcontinuity}) for the first three terms on the right-hand side with the choice of $\vartheta$ as $\frac{1}{2(2\alpha+\alpha(1-\alpha)+1)}$ to obtain 
\begin{align*}
\begin{aligned}
&\int_{0}^{t}(\mathcal{I}^{1-\alpha}\theta_2',\theta_2')\,ds +\int_{0}^{t}a_h(\theta_2,\theta_2')\,ds \le \frac{1}{2}\int_{0}^{t}(\mathcal{I}^{1-\alpha}\theta_2',\theta_2')\,ds + C\int_{0}^{t}\Big( (\mathcal{I}^{1-\alpha}\widehat{\theta},\widehat{\theta})\\
&+(\mathcal{I}^{1-\alpha}\theta_1,\theta_1)+ (\mathcal{I}^{1-\alpha}\lambda,\lambda)\Big)\,ds-\int_{0}^{t}(\varphi_2,(Q-I)\theta_2')\,ds+\int_{0}^{t}(f_2,(Q-I)\theta_2')\,ds.
\end{aligned}
\end{align*}
The constant $C$ in the above inequality depends on $\alpha$.
The symmetry of $a_h(\cdot,\cdot)$ from {\bf(H)} shows $2a_h(\theta_2(t),\theta_2'(t))=\frac{d}{dt}a_h(\theta_2(t),\theta_2(t))$. This with   \eqref{bhbilinearcoer_cts},  (\ref{bhfuncteste}), and $\theta_2(0)=0$ on the left-hand side of the above inequality establishes
\begin{align}\label{bhkeylemma_est}
t^{-\alpha}\norm[0]{\theta_2(t)}^2+\beta_1\norm[0]{\theta_2(t)}_h^2 &\lesssim  \int_{0}^{t}\Big( (\mathcal{I}^{1-\alpha}\widehat{\theta},\widehat{\theta})+(\mathcal{I}^{1-\alpha}\theta_1,\theta_1)+ (\mathcal{I}^{1-\alpha}\lambda,\lambda)\Big)\,ds \nonumber \\
& \qquad +|\int_{0}^{t}(\varphi_2,(Q-I)\theta_2')\,ds|+|\int_{0}^{t}(f_2,(Q-I)\theta_2')\,ds|. 
\end{align}
Now, the task is to bound the last two terms on the right-hand side of \eqref{bhkeylemma_est}.
Since $\theta_2$ and $\varphi_2$ belong to $W^{1,1}(0,T;L^2(\Omega))$, with an integration by parts, we have
\begin{align*}
T_1:= \int_{0}^{t}(\varphi_2(s),(Q-I)\theta_2'(s))\,ds= (\varphi_2(t),(Q-I)\theta_2(t))-\int_{0}^{t}(\varphi_2'(s),(Q-I)\theta_2(s))\,ds.
\end{align*}
The H\"older inequality and Lemma \ref{bhcompmorlyestmte}
show
\begin{align*}
\begin{aligned}
|T_1|&\le C_1 h^2 \big( \norm[0]{\varphi_2(t)} \|\theta_2(t)\|_h + \int_{0}^{t}\norm[0]{\varphi_2'(s)} \|\theta_2(s)\|_h\,ds \big) \\
&\le \frac{\beta_1}{4}\|\theta_2(t)\|_h^2 + C_1^2\beta_1^{-1} h^4\norm[0]{\varphi_2(t)}^2 + \frac{1}{2}\int_{0}^{t}\|\theta_2(s)\|_h^2\,ds + \frac{C_1^2h^4}{2}\int_{0}^{t}\norm[0]{\varphi_2'(s)}^2\,ds,
\end{aligned}
\end{align*}
with an application of Young's inequality in the last step above.
A similar approach is applied to bound the term $T_2:=\int_{0}^{t}(f_2,(Q-I)\theta_2')\,ds $ and  leads to
\begin{align*}
\begin{aligned}
|T_2|&\le C_1h^2\big(\norm[0]{f_2(t)} \|\theta_2(t)\|_h + \int_{0}^{t}\norm[0]{f_2'(s)}\|\theta_2(s)\|_h\,ds \big)\\
&\le \frac{\beta_1}{4}\|\theta_2(t)\|_h^2 + C_1^2\beta_1^{-1}h^4\norm[0]{f_2(t)}^2 + \frac{1}{{2}}\int_{0}^{t}\|\theta_2(s)\|_h^2\,ds + \frac{C_1^2h^4}{2}\int_{0}^{t}\norm[0]{f_2'(s)}^2\,ds. 
\end{aligned}
\end{align*}
Substitute these bounds of $|T_1|$ and $|T_2|$ in (\ref{bhkeylemma_est}) to obtain
\begin{align*}
&t^{-\alpha}\norm[0]{\theta_2(t)}^2+\beta_1\norm[0]{\theta_2(t)}_h^2 \lesssim \int_{0}^{t}\Big( (\mathcal{I}^{1-\alpha}\widehat{\theta},\widehat{\theta})+(\mathcal{I}^{1-\alpha}\theta_1,\theta_1)+ (\mathcal{I}^{1-\alpha}\lambda,\lambda)\Big)\,ds \nonumber \\
&+h^4\Big( \|\varphi_2(t)\|^2 +\int_{0}^{t}\|\varphi_2'(s)\|^2\,ds+  \|f_2(t)\|^2 +\int_{0}^{t}\|f_2'(s)\|^2\,ds  \Big)+\int_{0}^{t}\|\theta_2(s)\|_h^2\,ds. 
\end{align*}
Apply Gronwall's lemma in Lemma~\ref{gronwalllem} with 
\begin{align*}
& \phi(t)=t^{-\alpha}\norm[0]{\theta_2(t)}^2+\beta_1\norm[0]{\theta_2(t)}_h^2, \; \chi(t)=C, \text{ and }\\
&\psi(t)= C\int_{0}^{t}\big( (\mathcal{I}^{1-\alpha}\widehat{\theta},\widehat{\theta})+(\mathcal{I}^{1-\alpha}\theta_1,\theta_1)+ (\mathcal{I}^{1-\alpha}\lambda,\lambda)\big)\,ds\\
&\qquad \qquad + Ch^4\big( \|\varphi_2(t)\|^2 +\int_{0}^{t}\|\varphi_2'(s)\|^2\,ds+  \|f_2(t)\|^2 +\int_{0}^{t}\|f_2'(s)\|^2\,ds  \big),
\end{align*}
utilize \eqref{bhpositivt}, $\displaystyle \int_{0}^{t}\int_{0}^{s}g(\tau)\,d\tau\,ds\lesssim \int_{0}^{t}g(s)\,ds$ for $\displaystyle \int_{0}^{t}g(s)\,ds>0$ ($g(t)\geq 0$)  and $t\in(0,T]$ for the terms 
corresponding to a double integral in time, and recall $\theta_2(t)=t^2\theta(t)$ to conclude the proof.
\end{proof}
\noindent Lemmas \ref{bhlemma1}-\ref{bhlemma4} bound each term on the right-hand side of the estimate in Lemma \ref{keylem_bhthetafinlthm}. 

\noindent The notations $\displaystyle \widehat{\upsilon}(t)=\int_{0}^{t} \upsilon(s)\,ds$ for the choices $\upsilon=\theta,\rho,\varphi, f$ and $\displaystyle \widehat{\widehat{\upsilon}}(t)=\int_{0}^{t}\int_{0}^{s} \upsilon (\tau)\,d\tau\,ds$ for the choices $\upsilon=\theta,\varphi, f$ are used in the next lemma.
\begin{lemma}[Estimate for $\int_{0}^{t}(\mathcal{I}^{1-\alpha}\widehat{\theta},\widehat{\theta})\,ds$]\label{bhlemma1}
For $t\in(0,T]$  and $\theta$ satisfying  (\ref{bherreqn}), the bounds stated below hold. 
\begin{align*}
\int_{0}^{t}(\mathcal{I}^{1-\alpha}\widehat{\theta},\widehat{\theta})\,ds+\|\widehat{\widehat{\theta}}(t)\|_h^{2}\lesssim  t^{4-\alpha}
\Big(h^{2\gamma_0}t^{-\alpha} + h^2t^{-(\alpha+1)/2}\Big)^2 \norm[0]{u_0}^2, 
\end{align*}
if $u_0\in L^2(\Omega)$ and $f=0$, and 
\begin{align*}
\int_{0}^{t}(\mathcal{I}^{1-\alpha}\widehat{\theta},\widehat{\theta})\,ds+\|\widehat{\widehat{\theta}}(t)\|_h^{2}\lesssim   t^{4-\alpha} \Big(h^{2\gamma_0}{\Lambda_0(\epsilon,t)} +h^2t^{\alpha/2-2}{\Lambda_1(\epsilon,t) }\Big)^2,    
\end{align*}
for all $ \epsilon \in (0,1)$, if $u_0\in D(A)$ and $f$ is such that  $\Lambda_0(\epsilon,t)$ and {$\Lambda_1(\epsilon,t)$}, defined in \eqref{lambda0} and \eqref{bh_thetaconstnts_smth}, are finite.
\end{lemma}
\begin{proof}
 The  proof is split into four steps. The first step derives an intermediate bound, and \textit{Steps 2-4} estimate the terms derived in \textit{Step 1}.

\noindent\textit{Step 1 (An intermediate bound).} 
The definition of $\mathcal{I}^{\beta}(\cdot)$ in \eqref{rlfracintegral} and \eqref{bhsemigroup} establish 
$$\mathcal{I}^{2-\alpha}v'(t)=\mathcal{I}^{1-\alpha}\mathcal{I}(v'(t))=\mathcal{I}^{1-\alpha}(v(t)-v_0) =\mathcal{I}^{1-\alpha}v(t)-\kappa_{2-\alpha}(t)v(0).$$
Integrate (\ref{bherreqn}) over $(0,t)$, utilize the above identity twice, and  $(\rho(0)+\theta(0),\chi_h)=0$ for all $\chi_h \in V_h$ from the second identity in \eqref{bhsemidisscheme} to obtain 
\begin{align*}
(\mathcal{I}^{1-\alpha}\theta,\chi_h)+a_h(\widehat{\theta},\chi_h)&= -(\mathcal{I}^{1-\alpha}\rho,\chi_h)-(\widehat{\varphi},(Q-I)\chi_h)+(\widehat{f},(Q-I)\chi_h) \text{ for all } \chi_h\in V_h.  
\end{align*}
Integrate the above identity in time from $0$ to $t$ and choose $\chi_h=\widehat{\theta}(t)$ to establish
\begin{align*}
\begin{aligned}
(\mathcal{I}^{1-\alpha}\widehat{\theta},\widehat{\theta})+a_h(\widehat{\widehat{\theta}},\widehat{\theta})&=-(\mathcal{I}^{1-\alpha}\widehat{\rho},\widehat{\theta})-(\widehat{\widehat{\varphi}},(Q-I)\widehat{\theta})+(\widehat{\widehat{f}},(Q-I)\widehat{\theta}).
\end{aligned}
\end{align*}  
Integrate the above identity once again in time from $0$ to $t$, apply (\ref{bhcontinuity}) with $\vartheta=1/2$ to $\int_{0}^{t}(\mathcal{I}^{1-\alpha}(-\widehat{\rho}),\widehat{\theta})\,ds$,  utilize $2a_h(\widehat{\widehat{\theta}},\widehat{\theta})=\frac{d}{dt}a_h(\widehat{\widehat{\theta}},\widehat{\widehat{\theta}})$, and \eqref{bhbilinearcoer_cts} to obtain  
\begin{align*}
\begin{aligned}
\int_{0}^{t}(\mathcal{I}^{1-\alpha}\widehat{\theta},\widehat{\theta})\,ds+\beta_1\|\widehat{\widehat{\theta}}\|_h^2&\lesssim \int_{0}^{t}(\mathcal{I}^{1-\alpha}\widehat{\rho},\widehat{\rho})\,ds\\
&+\Big|\int_{0}^{t}(\widehat{\widehat{\varphi}},(Q-I)\widehat{\theta})\,ds\Big|
+\Big|\int_{0}^{t}(\widehat{\widehat{f}},(Q-I)\widehat{\theta})\,ds\Big|.
\end{aligned}
\end{align*}
{The bounds for the terms $\big|\int_{0}^{t}(\widehat{\widehat{\varphi}},(Q-I)\widehat{\theta})\,ds\big|$ and $\big|\int_{0}^{t}(\widehat{\widehat{f}},(Q-I)\widehat{\theta})\,ds\big|$ can be established in an analogous way following the steps of the derivation of the bounds for $|T_1|$ and $|T_2|$ in Lemma \ref{keylem_bhthetafinlthm}. In this case, we obtain
\begin{align}\label{eqn:bound-22}
 \int_{0}^{t}(\mathcal{I}^{1-\alpha}\widehat{\theta},\widehat{\theta})\,ds+\|\widehat{\widehat{\theta}}(t)\|_h^{2} & \lesssim 
\int_{0}^{t}(\mathcal{I}^{1-\alpha}\widehat{\rho},\widehat{\rho})\,ds + h^4 \Big( \|\widehat{\widehat{\varphi}}(t)\|^2 +\int_{0}^{t}\|\widehat{\varphi}(s)\|^2 \,ds \nonumber \\
& \quad+ \|\widehat{\widehat{f}}(t)\|^2 +\int_{0}^{t}\|\widehat{f}(s)\|^2\,ds  \Big)
+ \int_{0}^{t} \|\widehat{\widehat{\theta}}(s)\|_h^{2}\,ds.
\end{align}
Apply Gronwall's lemma in \eqref{eqn:bound-22} to obtain   \begin{align}\label{eqn:bound}
 \int_{0}^{t}(\mathcal{I}^{1-\alpha}\widehat{\theta},\widehat{\theta})\,ds+\|\widehat{\widehat{\theta}}(t)\|_h^{2} & \lesssim 
\int_{0}^{t}(\mathcal{I}^{1-\alpha}\widehat{\rho},\widehat{\rho})\,ds + h^4 \Big( \|\widehat{\widehat{\varphi}}(t)\|^2 +\int_{0}^{t}(\|\widehat{\varphi}(s)\|^2+ { \|\widehat{\widehat{\varphi}}(s)\|^2) }\,ds \nonumber \\
& \quad+ \|\widehat{\widehat{f}}(t)\|^2 +\int_{0}^{t}(\|\widehat{f}(s)\|^2+{\|\widehat{\widehat{f}}(s)\|^2) }\,ds  \Big).
\end{align}
}
\noindent\textit{Step 2 (Estimate for the term
$\int_{0}^{t}(\mathcal{I}^{1-\alpha}\widehat{\rho},\widehat{\rho})\,ds$ on the right-hand side of \eqref{eqn:bound}).}
An application of \eqref{bhritzprjn_ns_smth} results in
\begin{align*}
\norm[0]{\widehat{\rho}(s)} & \le \int_{0}^{s}\|\rho(\tau)\|\,d\tau\\
& \lesssim
\begin{cases}
\displaystyle 
 h^{2\gamma_0} \norm[0]{u_0}\int_{0}^{s}\tau^{-\alpha}\,d\tau\lesssim h^{2\gamma_0}s^{1-\alpha}\norm[0]{u_0}
\text{ for } f=0 \text{ and } u_0\in L^2(\Omega),\\
h^{2\gamma_0}s \: \Lambda_0(\epsilon,s)
\text{ for } f \neq 0 \text{ and } u_0\in D(A).
\end{cases}
\end{align*}
This bound, the definition of $\mathcal{I}^{1-\alpha}$, and  $\tau\le s$ in the computation of the integral in the second step below lead to 
\begin{align*}
\norm[0]{\mathcal{I}^{1-\alpha}\widehat{\rho}(s)}\lesssim \int_{0}^{s}(s-\tau)^{-\alpha}\norm[0]{\widehat{\rho}(\tau)}\,d\tau\lesssim
\begin{cases}
 h^{2\gamma_0}s^{2-2\alpha}\norm[0]{u_0} \text{ for } f=0 \text{ and } u_0\in L^2(\Omega),\\
h^{2\gamma_0}s^{2-\alpha}\Lambda_0(\epsilon,s)
\text{ for } f \neq 0 \text{ and } u_0\in D(A).
\end{cases}
\end{align*}
The H\"older inequality and a combination of the last two displayed bounds show
\begin{align*}
\int_{0}^{t}(\mathcal{I}^{1-\alpha}\widehat{\rho},\widehat{\rho})\,ds \le
\int_{0}^{t}\norm[0]{\mathcal{I}^{1-\alpha}\widehat{\rho}}\norm[0]{\widehat{\rho}}\,ds\lesssim 
\begin{cases}
h^{4\gamma_0}t^{4-3\alpha}\norm[0]{u_0}^2 \text{ for } f=0 \text{ and } u_0\in L^2(\Omega),\\
h^{4\gamma_0}t^{4-\alpha}\Lambda_0(\epsilon,t)^2
 \text{ for } f \neq 0 \text{ and } u_0\in D(A).
 \end{cases}
\end{align*}
\noindent\textit{Step 3 (Estimate for the term
$ (\|\widehat{\widehat{\varphi}}(t)\|^2 +\int_{0}^{t}(\|\widehat{\varphi}(s)\|^2+ \|\widehat{\widehat{\varphi}}(s)\|^2  ) \,ds )$
 on the right-hand side of \eqref{eqn:bound}).}  The approach for the estimates for the nonsmooth  and smooth initial data differs here as direct bounds for the required estimates offer challenges due to the singularity factor $t^{-1}$ (see \eqref{bhhomopfn_func_ns} with $i=1,p=0$)  when $u_0 \in L^2(\Omega)$. However, this issue is resolved below with the help of (\ref{bhsemigroup}).

\noindent{\it Nonsmooth initial data ($f=0$ and $u_0 \in L^2(\Omega)$)}.
Recall  $\varphi=\mathcal{I}^{1-\alpha}u'$ and apply  (\ref{bhsemigroup}) to observe
\begin{align*}
\begin{aligned}
\widehat{\varphi}&=\mathcal{I}(\mathcal{I}^{1-\alpha}u'(t))=\mathcal{I}^{1-\alpha}\mathcal{I}(u'(t))=\mathcal{I}^{1-\alpha}(u(t)-u_0), \; \;
\widehat{\widehat{\varphi}}(t)=\mathcal{I}^{2-\alpha}(u(t)-u_0).
\end{aligned}
\end{align*}
The definition of ${\cal I}^{2-\alpha}$ (cf. \eqref{rlfracintegral}), $(t-s)\le t$, and  (\ref{bhhomopfn_func_ns}) with $i=0, p=0$ lead to 
\begin{align*}
\|\widehat{\widehat{\varphi}}(t)\|^2 = \|\int_{0}^{t} \frac{(t-s)^{1-\alpha}}{\Gamma(2-\alpha)}(u(s)-u_0)\,ds\|^2 \lesssim \Big( t^{1-\alpha}\int_{0}^{t}(\|u(s)\|+\|u_0\|)\,ds \Big)^2\lesssim t^{4-2\alpha}\norm[0]{u_0}^2.  
\end{align*}
This yields
\begin{equation}\label{eqn:trick}
\displaystyle \int_{0}^{t}\|\widehat{\widehat{\varphi}}(s)\|^2\,ds \lesssim \int_{0}^{t} s^{4-2 \alpha} \norm[0]{u_0}^2 \,ds  =
\frac{t^{5-2 \alpha}}{5-2 \alpha}  \norm[0]{u_0}^2  \lesssim 
t^{4-2\alpha}\norm[0]{u_0}^2
\end{equation}
with $t\le T$ (used once) in the last step.
The definition of ${\cal I}^{1-\alpha}$ (cf. \eqref{rlfracintegral}) and  (\ref{bhhomopfn_func_ns}) with $i=0, p=0$ establish
$$\displaystyle \|\widehat{\varphi}(s)\|^2 =\|\int_{0}^{s} \frac{(s-\tau)^{-\alpha}}{\Gamma(1-\alpha)}(u(\tau)-u_0)\,d\tau\|^2\lesssim \Big(\int_{0}^{s}(s-\tau)^{-\alpha}(\|u(\tau)\|+\|u_0\|)\,d\tau\Big)^2
\lesssim s^{2-2\alpha}\norm[0]{u_0}^2,$$
and this shows
$ \displaystyle \int_{0}^{t} \|\widehat{\varphi}(s)\|^2\,ds\lesssim t^{3-2\alpha}\norm[0]{u_0}^2$. 
Altogether, we obtain
\begin{align*}
 \displaystyle   \|\widehat{\widehat{\varphi}}(t)\|^2 + \int_{0}^{t}(\|\widehat{\varphi}(s)\|^2+ \|\widehat{\widehat{\varphi}}(s)\|^2  ) \,ds  \lesssim t^{3-2\alpha}(t+1) \norm[0]{u_0}^2.
\end{align*}
\noindent{\it Smooth initial data ($f \neq 0$ and $u_0 \in D(A)$)}.
Using $\widehat{\widehat{\varphi}}(t)={\cal I}^{3-\alpha}u'(t)$, $(t-s)\le t$, and (\ref{bhhomopfn_dtve_s}) with $p=0$ and  (\ref{bhnonhomo_dtve}) with $q=\epsilon$, we have
\begin{align*}
\|\widehat{\widehat{\varphi}}(t)\|^2&\lesssim \Big(\int_{0}^{t}(t-s)^{2-\alpha}\norm[0]{u'(s)}\,ds\Big)^2
\lesssim 
\Big( t^{2-\alpha} \int_{0}^{t}\big(s^{\alpha-1}\|u_0\|_{D(A)}+ s^{-1+\frac{\alpha\epsilon}{2}}\|f(0)\|_{D(A^{\epsilon/2})} \nonumber 
\\
& \quad + \epsilon^{-1}s^{\frac{\alpha\epsilon}{2}}\|f'\|_{L^{\infty}(0,T;D(A^{\epsilon/2}))}\big)\,ds\Big)^2
\\
&\lesssim \Big( t^2\norm[0]{u_0}_{D(A)}+t^{2-\alpha(1-\epsilon/2)}\norm[0]{f(0)}_{D(A^{\epsilon/2})}+\epsilon^{-1}t^{3-\alpha(1-\epsilon/2)}\norm[0]{f'}_{L^{\infty}(0,T;D(A^{\epsilon/2}))}  \Big)^2. 
\end{align*}
An integration from $0$ to $t$, \eqref{algebraic}, elementary algebraic manipulations, and $t\le T$ (used once) as in \eqref{eqn:trick}, leads to a similar bound for  $\displaystyle \int_{0}^{t}\|\widehat{\widehat{\varphi}}(s)\|^2\,ds$. Analogous arguments with elementary manipulations lead to
\begin{align*}
\int_{0}^{t}\|\widehat{\varphi}(s)\|^2\,ds
&\lesssim \Big( t^\frac{3}{2}\norm[0]{u_0}_{D(A)}+t^{\frac{3}{2}-\alpha(1-\frac{\epsilon}{2})}\norm[0]{f(0)}_{D(A^\frac{\epsilon}{2})}+\epsilon^{-1}t^{\frac{5}{2}-\alpha(1-\frac{\epsilon}{2})}\norm[0]{f'}_{L^{\infty}(0,T;D(A^\frac{\epsilon}{2}))}  \Big)^2.
\end{align*}
A combination of the last two displayed inequalities establishes the required estimate in this step.

\noindent\textit{Step 4 (Estimate for the term
$\big( \|\widehat{\widehat{f}}(t)\|^2 +\int_{0}^{t}(\|\widehat{f}(s)\|^2+ \|\widehat{\widehat{f}}(s)\|^2 )\,ds  \big)$
 on the right-hand side of \eqref{eqn:bound}).}
For {\it nonsmooth}  initial data, since $f$ is chosen as zero, the term is zero. For {\it smooth} initial data, the definitions of $\widehat{f}$ and $\widehat{\widehat{f}}$ show 
\begin{align*}
  \|\widehat{\widehat{f}}(t)\|^2 +\int_{0}^{t}(\|\widehat{f}(s)\|^2 + \|\widehat{\widehat{f}}(s)\|^2 )\,ds  \lesssim \Big( t^{3/2}\|f\|_{L^2(0,t;L^2(\Omega))}+ t\|f\|_{L^2(0,t;L^2(\Omega))}\Big)^2. 
\end{align*}
A combination of  \textit{Steps 2-4} in \eqref{eqn:bound} and algebraic manipulations conclude the proof.
\end{proof}
\noindent The notations $
\upsilon_1(t)=t \upsilon(t)$ for $\upsilon=\theta, \rho,\varphi,f$ and $\displaystyle \widehat{\upsilon}(t)=\int_{0}^{t} \upsilon(s)\,ds$ for $\upsilon=\theta,\theta_1,\rho,\varphi_1, f_1$ are used in the next lemma.
\begin{lemma}[Estimate for $\int_{0}^{t}(\mathcal{I}^{1-\alpha}\theta_1,\theta_1)\,ds$]\label{bhlemma2}
For $\theta$ that satisfies (\ref{bherreqn}) and $t\in(0,T]$, the bounds stated below hold. 
\begin{align*}
\int_{0}^{t}(\mathcal{I}^{1-\alpha}\theta_1,\theta_1)\,ds+\|\widehat{\theta_1}(t)\|_h^{2}\lesssim   t^{4-\alpha}
\Big(h^{2\gamma_0}t^{-\alpha} + h^2t^{-(\alpha+1)/2}\Big)^2 \norm[0]{u_0}^2, 
\end{align*}
if $u_0\in L^2(\Omega)$ and $f=0$, and 
\begin{align*}
\int_{0}^{t}(\mathcal{I}^{1-\alpha}\theta_1,\theta_1)\,ds+\|\widehat{\theta_1}(t)\|_h^{2}\lesssim  t^{4-\alpha}\Big(h^{2\gamma_0}\Lambda_0(\epsilon,t) +h^2t^{\alpha/2-2}{\Lambda_2(\epsilon,t) }\Big)^2,  
\end{align*}
for all $ \epsilon \in (0,1)$, if $u_0\in D(A)$ and $f$ is such that $\Lambda_0(\epsilon,t)$ and $\Lambda_2(\epsilon,t)$,  defined in \eqref{lambda0} and \eqref{bh_thetaconstnts_smth}, are finite.
\end{lemma}
\begin{proof}
 The structure of the proof is similar to that of Lemma \ref{bhlemma1}. However, the majorizations are different and for the readability of the paper, we give the necessary details. The proof is presented in four steps. To prove the desired bounds, we proceed following Lemma \ref{keylem_bhthetafinlthm} as follows. \\
\noindent{\textit{Step 1 (An intermediate bound).}}   Multiply  (\ref{bherreqn}) by $t$, use (\ref{bhminiest1}), and $(\rho(0)+\theta(0),\chi_h)=0$ for all $\chi_h\in V_h$ to arrive at
\begin{align*}
\begin{aligned}
(\mathcal{I}^{1-\alpha}\theta_1',\chi_h)+a_h(\theta_1,\chi_h)&= \alpha(\mathcal{I}^{1-\alpha}\theta,\chi_h)-(\mathcal{I}^{1-\alpha}\rho_1',\chi_h) +\alpha(\mathcal{I}^{1-\alpha}\rho,\chi_h)\\
&\quad -(\varphi_1(t),(Q-I)\chi_h)+(f_1(t),(Q-I)\chi_h) \text{ for all } \chi_h\in V_h.
\end{aligned}
\end{align*}
Integrate the above equality over $(0,t)$, and next choose $\chi_h=\theta_1(t)$ to obtain
\begin{align*}
\begin{aligned}
(\mathcal{I}^{1-\alpha}\theta_1,\theta_1)+a_h(\widehat{\theta_1},\theta_1)&= \alpha(\mathcal{I}^{1-\alpha}\widehat{\theta},\theta_1)-(\mathcal{I}^{1-\alpha}\rho_1,\theta_1) +\alpha(\mathcal{I}^{1-\alpha}\widehat{\rho},\theta_1)\\
&-(\widehat{\varphi_1}(t),(Q-I)\theta_1)+(\widehat{f_1}(t),(Q-I)\theta_1).
\end{aligned}
\end{align*}
An integration over $(0,t)$ once again,  an application of \eqref{bhcontinuity} to the first three terms on the right-hand side, and the choice $\vartheta=\frac{1}{2(2\alpha+1)}$ reveal 
\begin{align*}
\begin{aligned}
&\int_{0}^{t}(\mathcal{I}^{1-\alpha}\theta_1,\theta_1)\,ds+\int_{0}^{t}a_h(\widehat{\theta_1},\theta_1)\,ds\le \frac{1}{2}\int_{0}^{t}(\mathcal{I}^{1-\alpha}\theta_1,\theta_1)\,ds+ C\int_{0}^{t}\Big[(\mathcal{I}^{1-\alpha}\widehat{\theta},\widehat{\theta})+(\mathcal{I}^{1-\alpha}\rho_1,\rho_1)\\
&+(\mathcal{I}^{1-\alpha}\widehat{\rho},\widehat{\rho})\Big]\,ds
+\Big|\int_{0}^{t}(\widehat{\varphi_1}(s),(Q-I)\theta_1)\,ds\Big|+\Big|\int_{0}^{t}(\widehat{f_1}(s),(Q-I)\theta_1)\,ds\Big|.
\end{aligned}
\end{align*}
Apply $2a_h(\widehat{\theta_1}(t),\theta_1(t))=\frac{d}{dt}a_h(\widehat{\theta_1}(t),\widehat{\theta_1}(t))$, \eqref{bhbilinearcoer_cts}, approach of \textit{Step 1} of Lemma \ref{bhlemma1} to bound the last two terms of the above displayed estimate, and Gronwall's lemma to obtain
\begin{align}\label{bhbnd_3}
&\int_{0}^{t}(\mathcal{I}^{1-\alpha}\theta_1,\theta_1)\,ds+\|\widehat{\theta_1}(t)\|_h^{2}\lesssim \int_{0}^{t}(\mathcal{I}^{1-\alpha}\widehat{\theta},\widehat{\theta})\,ds+\int_{0}^{t}(\mathcal{I}^{1-\alpha}\widehat{\rho},\widehat{\rho})\,ds + \int_{0}^{t}(\mathcal{I}^{1-\alpha}\rho_1,\rho_1)\,ds \nonumber \\
& \quad + h^4\Big( \|\widehat{\varphi_1}(t)\|^2 +\int_{0}^{t} (\|\varphi_1(s)\|^2+ \|\widehat{\varphi_1}(s)\|^2) \,ds + \|\widehat{f_1}(t)\|^2 +\int_{0}^{t} (\|f_1(s)\|^2+ \|\widehat{f_1}(s)\|^2)\,ds  \Big).
\end{align}
Observe that the bounds of the first two terms on the right-hand side of \eqref{bhbnd_3} are available from the statement and {\it Step 2} of Lemma \ref{bhlemma1}. Hence in the steps below, we  bound the remaining terms.

\noindent{\textit{Step 2 (Estimate for the term
$\int_{0}^{t}(\mathcal{I}^{1-\alpha}\rho_1,\rho_1)\,ds$ on the right-hand side of \eqref{bhbnd_3}).}}  In view of \eqref{bhritzprjn_ns_smth}, we obtain 
\begin{align*}
\norm[0]{\rho_1(s)}=s\norm[0]{\rho(s)} \lesssim  
\begin{cases}
\displaystyle 
h^{2\gamma_0}s^{1-\alpha}\norm[0]{u_0}
\text{ for } f=0 \text{ and } u_0\in L^2(\Omega),\\
h^{2\gamma_0}s \: \Lambda_0(\epsilon,s)
\text{ for } f \neq 0 \text{ and } u_0\in D(A).
\end{cases}  
\end{align*}
Employ this to establish
\begin{align*}
\norm[0]{\mathcal{I}^{1-\alpha}\rho_1(s)}\lesssim \int_{0}^{s}(s-\tau)^{-\alpha}\norm[0]{\rho_1(\tau)}\,d\tau\lesssim  
\begin{cases}
 h^{2\gamma_0}s^{2-2\alpha}\norm[0]{u_0} \text{ for } f=0 \text{ and } u_0\in L^2(\Omega),\\
h^{2\gamma_0}s^{2-\alpha}\Lambda_0(\epsilon,s)
\text{ for } f \neq 0 \text{ and } u_0\in D(A).
\end{cases} 
\end{align*}
The H\"older inequality plus the last two displayed bounds reveal 
\begin{align*}
\int_{0}^{t}(\mathcal{I}^{1-\alpha}\rho_1,\rho_1)\,ds \le
\int_{0}^{t}\norm[0]{\mathcal{I}^{1-\alpha}\rho_1}\norm[0]{\rho_1}\,ds\lesssim   
\begin{cases}
h^{4\gamma_0}t^{4-3\alpha}\norm[0]{u_0}^2 \text{ for } f=0 \text{ and } u_0\in L^2(\Omega),\\
h^{4\gamma_0}t^{4-\alpha}\Lambda_0(\epsilon,t)^2
 \text{ for } f \neq 0 \text{ and } u_0\in D(A).
 \end{cases}   
\end{align*}
\noindent\textit{Step 3 (Estimate for the term
$ \big( \|\widehat{\varphi_1}(t)\|^2 +\int_{0}^{t}(\|\varphi_1(s)\|^2+\|\widehat{\varphi_1}(s)\|^2) \,ds \big)$
on the right-hand side of \eqref{bhbnd_3}).}  Since the nonsmooth data stability estimate reflects a singularity factor $t^{-1}$ (cf. \eqref{bhhomopfn_func_ns} with $i=1, p=0$), we proceed as follows for this case.

\noindent{\it Nonsmooth initial data ($f=0$ and $u_0 \in L^2(\Omega)$)}. Recall that $\varphi=\mathcal{I}^{1-\alpha}u'$ and apply \eqref{bhminiest1} to observe
\begin{align*}
\varphi_1(s)=s\varphi(s)=s\mathcal{I}^{1-\alpha}u'(s)&= \mathcal{I}^{1-\alpha}(su(s))'-\alpha\mathcal{I}^{1-\alpha}u(s)-s\kappa_{1-\alpha}(s)u_0  \\
&=\mathcal{I}^{1-\alpha}(u(s)+su'(s))-\alpha\mathcal{I}^{1-\alpha}u(s)-s\kappa_{1-\alpha}(s)u_0.
\end{align*}
The definition of $\mathcal{I}^{1-\alpha}$ in \eqref{rlfracintegral} and repeated applications of \eqref{bhhomopfn_func_ns} with $i=0,p=0$ and $i=1,p=0$ yield $\|\varphi_1(s)\|\lesssim s^{1-\alpha}\|u_0\|$  
and hence
$ \displaystyle \|\widehat{\varphi_1}(t)\|^2 \lesssim t^{4-2\alpha}\norm[0]{u_0}^2.  $
Utilize this bound  to arrive at 
\begin{equation*}
\displaystyle \int_{0}^{t}\|\widehat{\varphi_1}(s)\|^2\,ds \lesssim \int_{0}^{t} s^{4-2 \alpha} \norm[0]{u_0}^2 \,ds  \lesssim 
t^{4-2\alpha}\norm[0]{u_0}^2,
\end{equation*}
with $t\le T$ (used once) in the last step.
Analogous arguments show  $ \displaystyle \int_{0}^{t}\|\varphi_1(s)\|^2\,ds \lesssim t^{3-2\alpha}\norm[0]{u_0}^2$. A combination of  these bounds shows
\begin{align*}
\displaystyle   \|\widehat{\varphi_1}(t)\|^2 + \int_{0}^{t} (\|\varphi_1(s)\|^2+\|\widehat{\varphi_1}(s)\|^2) \,ds  \lesssim t^{3-2\alpha}(t+1) \norm[0]{u_0}^2.
\end{align*}

\noindent{\it Smooth initial data ($f \neq 0$ and $u_0 \in D(A)$)}. Apply the definition of $\mathcal{I}^{1-\alpha}$ in \eqref{rlfracintegral}, $s\le t$, and \eqref{bhhomopfn_dtve_s} with $p=0$ and \eqref{bhnonhomo_dtve} with $q=\epsilon$ to obtain
\begin{align} \label{eqn:temp1}
\|\widehat{\varphi_1}(t)\|^2&  =\|\int_{0}^{t}s\mathcal{I}^{1-\alpha}u'(s)\,ds\|^2\lesssim \Big(t\int_{0}^{t}\int_{0}^{s}(s-\tau)^{-\alpha}\|u'(\tau)\|\,d\tau\,ds\Big)^2 \nonumber\\
&\lesssim \Big(t\int_{0}^{t}\int_{0}^{s}(s-\tau)^{-\alpha}\big(\tau^{\alpha-1}\|u_0\|_{D(A)}+ \tau^{-1+\frac{\alpha\epsilon}{2}}\|f(0)\|_{D(A^{\epsilon/2})} \nonumber \\
& \quad + \epsilon^{-1}\tau^{\frac{\alpha\epsilon}{2}}\|f'\|_{L^{\infty}(0,T;D(A^{\epsilon/2}))}\big)\,d\tau\,ds\Big)^2.
\end{align}
For $\alpha\in (0,1)$, recall the following identities involving the beta function $B(\cdot,\cdot)$:
\begin{align} \label{beta}
\begin{aligned}
&\int_{0}^{t}(t-s)^{-\alpha}s^{\alpha-1}\,ds = \int_{0}^{1}(1-s)^{\alpha-1}s^{-\alpha}\,ds=B(\alpha,1-\alpha), \\  & \int_{0}^{t}(t-s)^{-\alpha}s^{\frac{\alpha\epsilon}{2}-1}\,ds= t^{-\alpha(1-\epsilon/2)}\int_{0}^{1}(1-s)^{\frac{\alpha\epsilon}{2}-1}s^{-\alpha}\,ds=B(\alpha\epsilon/2,1-\alpha)t^{-\alpha(1-\epsilon/2)} .
\end{aligned}
\end{align}
Substitute \eqref{beta} in \eqref{eqn:temp1} to obtain
\begin{align*} 
\|\widehat{\varphi_1}(t)\|^2 & \lesssim \Big( t^2\norm[0]{u_0}_{D(A)}+B(\alpha\epsilon/2,1-\alpha)t^{2-\alpha(1-\epsilon/2)}\norm[0]{f(0)}_{D(A^{\epsilon/2})} \nonumber \\
& \quad +\epsilon^{-1}t^{3-\alpha(1-\epsilon/2)}\norm[0]{f'}_{L^{\infty}(0,T;D(A^{\epsilon/2}))}\Big)^2.
\end{align*}
Utilize $t\le T$  to conclude that an analogous bound holds for the term $\int_{0}^{t}\|\widehat{\varphi_1}(s)\|^2\,ds$.
Further, similar calculations as for $\|\widehat{\varphi_1}(t)\|^2$ above and algebraic manipulations with  \eqref{algebraic} reveal
\begin{align*}
&\int_{0}^{t}\|\varphi_1(s)\|^2\,ds= \int_{0}^{t}\|s\mathcal{I}^{1-\alpha}u'(s)\|^2\,ds\lesssim t^2\int_{0}^{t}\Big(\int_{0}^{s}(s-\tau)^{-\alpha}\|u'(\tau)\|\,d\tau\Big)^2\,ds \lesssim \Big( t^{3/2}\norm[0]{u_0}_{D(A)} \\
& + B(\alpha\epsilon/2,1-\alpha)t^{3/2-\alpha(1-\epsilon/2)}\norm[0]{f(0)}_{D(A^{\epsilon/2})}+\epsilon^{-1}t^{5/2-\alpha(1-\epsilon/2)}\norm[0]{f'}_{L^{\infty}(0,T;D(A^{\epsilon/2}))}  \Big)^2.
\end{align*}

\noindent\textit{Step 4 (Estimate for the term
$\big(\|\widehat{f_1}(t)\|^2 +\int_{0}^{t}(\|f_1(s)\|^2+ \|\widehat{f_1}(s)\|^2)\,ds  \big)$ on the right-hand side of \eqref{bhbnd_3}).} Since $f=0$ for {\it nonsmooth}  initial data,  the estimate is trivial. For the case of {\it smooth} initial data, the definitions of $f_1, \widehat{f_1}$ lead to
\begin{align*}
\|\widehat{f_1}(t)\|^2 +\int_{0}^{t}(\|f_1(s)\|^2+ \|\widehat{f_1}(s)\|^2)\,ds  \lesssim \Big( t^{3/2}\|f\|_{L^2(0,t;L^2(\Omega))}+ t\|f\|_{L^2(0,t;L^2(\Omega))} \Big)^2.    
\end{align*}
A combination of the estimates from \textit{Steps 2-4} with \eqref{bhbnd_3} concludes the proof.
\end{proof}
\begin{lemma}[Estimate for $\int_{0}^{t}(\mathcal{I}^{1-\alpha}\lambda,\lambda)\,ds$]\label{bhlemma3}
For $\lambda=-\rho_2'+2\alpha\rho_1+\alpha(1-\alpha)\widehat{\rho}$ with $\rho_2'(t)=(t^2\rho(t))'$, $\rho_1(t)=t\rho(t)$, $\displaystyle \widehat{\rho}(t)=\int_{0}^{t}\rho(s)\,ds$, and $t\in(0,T]$, the following bounds hold. 
\begin{align*}
\int_{0}^{t}(\mathcal{I}^{1-\alpha}\lambda,\lambda)\,ds\lesssim    t^{4-\alpha}
\Big(h^{2\gamma_0}t^{-\alpha}\norm[0]{u_0}\Big)^2, 
\end{align*}
for $u_0\in L^2(\Omega)$ and $f=0$, and 
\begin{align*}
\int_{0}^{t}(\mathcal{I}^{1-\alpha}\lambda,\lambda)\,ds\lesssim    t^{4-\alpha}\Big(h^{2\gamma_0}\mathcal{B}_0(\epsilon,t)
\Big)^2,  
\end{align*}
for all $ \epsilon \in (0,1)$, if $u_0\in D(A)$ and $f$ such that
$\mathcal{B}_0(\epsilon,t)$  defined in  \eqref{bh_thetaconstnts_smth} is finite.
\end{lemma}
\begin{proof}
\noindent{\textit{Step 1 (Nonsmooth data).}}   For $f=0$ and $u_0\in L^2(\Omega)$, Lemma \ref{bhritzprjn_approx} and  (\ref{bhhomopfn_func_ns}) (with $i=0, p=1$ and $i=p=1$) result in
\begin{align*}
\begin{aligned}
\norm[0]{\lambda(s)}&\le \norm[0]{\widehat{\rho}(s)}+3\norm[0]{\rho_1(s)}+s^2\norm[0]{\rho'(s)}\lesssim h^{2\gamma_0}s^{1-\alpha}\|u_0\|+ h^{2\gamma_0}s^2\norm[0]{u'(s)}_{H^{2+\gamma_0}(\Omega)} \\
&\lesssim h^{2\gamma_0}s^{1-\alpha}\|u_0\|.  
\end{aligned}
\end{align*}
This bound leads to 
\begin{align*}
\begin{aligned}
\norm[0]{\mathcal{I}^{1-\alpha}\lambda(s)}
&\lesssim \int_{0}^{s}(s-\tau)^{-\alpha}\|\lambda(\tau)\|\,d\tau\lesssim   h^{2\gamma_0}s^{2-2\alpha}\|u_0\|.  
\end{aligned}
\end{align*}
The H\"older inequality and the two bounds displayed above establish
\begin{align*}
\begin{aligned}
\int_{0}^{t}(\mathcal{I}^{1-\alpha}\lambda,\lambda)\,ds\le  \int_{0}^{t}\norm[0]{\mathcal{I}^{1-\alpha}\lambda}\norm[0]{\lambda}\,ds&\lesssim   h^{4\gamma_0}t^{4-3\alpha}\|u_0\|^2.   
\end{aligned}
\end{align*}

\noindent{\textit{Step 2 (Smooth data).}} For $f\ne 0$ and $u_0\in D(A)$, apply Lemma \ref{bhritzprjn_approx}, \eqref{bhhomopfn_func_s} and (\ref{bhnonhomo_fnest}) with $q=\epsilon$, (\ref{bhhomopfn_dtve_s}) with $p=1$  and (\ref{bhnonhomo_dtve}) with $q=\epsilon$ to obtain 
\begin{align*}
\begin{aligned}
&\norm[0]{\lambda(s)}\le \norm[0]{\widehat{\rho}(s)}+3\norm[0]{\rho_1(s)}+s^2\norm[0]{\rho'(s)}\lesssim h^{2\gamma_0}s\Lambda_0(\epsilon,s) + h^{2\gamma_0}s^2\norm[0]{u'(s)}_{H^{2+\gamma_0}(\Omega)} \\
&\lesssim h^{2\gamma_0}s\Lambda_0(\epsilon,s)+ h^{2\gamma_0}\Big( s\norm[0]{u_0}_{D(A)}+s^{1+\frac{\alpha\epsilon}{2}}\norm[0]{f(0)}_{D(A^{\epsilon/2})}+\epsilon^{-1}s^{2+\frac{\alpha\epsilon}{2}}\norm[0]{f'}_{L^{\infty}(0,T;D(A^{\epsilon/2}))}      \Big).  
\end{aligned}
\end{align*}
An application of this bound in the definition of $\mathcal{I}^{1-\alpha}$ shows
\begin{align*}
\begin{aligned}
&\norm[0]{\mathcal{I}^{1-\alpha}\lambda(s)}
{\lesssim \int_{0}^{s}(s-\tau)^{-\alpha}\|\lambda(\tau)\|\,d\tau }\lesssim    h^{2\gamma_0}s^{2-\alpha}\Lambda_0(\epsilon,s)\\
&\quad + h^{2\gamma_0}\Big( s^{2-\alpha}\norm[0]{u_0}_{D(A)}+s^{2-\alpha(1-\epsilon/2)}\norm[0]{f(0)}_{D(A^{\epsilon/2})}
+\epsilon^{-1}s^{3-\alpha(1-\epsilon/2)}\norm[0]{f'}_{L^{\infty}(0,T;D(A^{\epsilon/2}))}       \Big).  
\end{aligned}
\end{align*}
The H\"older inequality and the two bounds displayed above establish
\begin{align*}
\begin{aligned}
\int_{0}^{t}(\mathcal{I}^{1-\alpha}\lambda,\lambda)\,ds\le \int_{0}^{t}\norm[0]{\mathcal{I}^{1-\alpha}\lambda}\norm[0]{\lambda}\,ds&\lesssim    h^{4\gamma_0}t^{4-\alpha}\Big( \Lambda_0(\epsilon,t)  \\
& +t^{\alpha\epsilon/2}\norm[0]{f(0)}_{D(A^{\epsilon/2})}+\epsilon^{-1}t^{1+\alpha\epsilon/2}\norm[0]{f'}_{L^{\infty}(0,T;D(A^{\epsilon/2}))}  \Big)^2. 
\end{aligned}
\end{align*}
This concludes the proof.
\end{proof}
\begin{lemma}[Estimate for the last term in Lemma \ref{keylem_bhthetafinlthm}]\label{bhlemma4}
For $t\in(0,T]$ and for all $ \epsilon \in (0,1)$, it holds that
\begin{align*}
\begin{aligned}
h^4\Big( \|\varphi_2(t)\|^2 + &\int_{0}^{t}(\|\varphi_2(s)\|^2 +\|\varphi_2'(s)\|^2)\,ds+
 \|f_2(t)\|^2+\int_{0}^{t}(\|f_2(s)\|^2+\|f_2'(s)\|^2 )\,ds  \Big)\\
&\lesssim  t^{4-\alpha}
\begin{cases}
\Big(h^2t^{-(\alpha+1)/2}\norm[0]{u_0}\Big)^2, \;\;   \text{ if }\; u_0\in L^2(\Omega), f=0,\\
\Big(h^2t^{\alpha/2-2}\mathcal{B}_1(\epsilon,t)\Big)^2, \;\;   \text{ if }\; u_0\in D(A), f\ne 0,
\end{cases} 
\end{aligned}
\end{align*}
where $\varphi_2(t)=t^2\varphi(t)$, $\varphi_2'(s)=(s^2\varphi(s))'$, $f_2(t)=t^2f(t)$, $f_2'(s)=(s^2f(s))'$,   and $f$ is such that $\mathcal{B}_1(\epsilon,t)$ defined in \eqref{bh_thetaconstnts_smth} is finite.
\end{lemma}
\begin{proof}
\noindent{\textit{Step 1 (Nonsmooth data).}} For $f=0$ and $u_0\in L^2(\Omega)$, we recall \eqref{bhminiest2} to arrive at
\begin{align*}
\varphi_2(t)&=t^2\mathcal{I}^{1-\alpha}u'(t)= \mathcal{I}^{1-\alpha}(u_2)'(t)-2\alpha\mathcal{I}^{1-\alpha}(u_1)(t)-\alpha(1-\alpha)\mathcal{I}^{2-\alpha}(u(t))-t^2\kappa_{1-\alpha}(t)u_0\\
&=\mathcal{I}^{1-\alpha}(2tu(t)+t^2u'(t))-2\alpha\mathcal{I}^{1-\alpha}(tu(t))-\alpha(1-\alpha)\mathcal{I}^{2-\alpha}(u(t))-t^2\kappa_{1-\alpha}(t)u_0.
\end{align*}
Applications of \eqref{rlfracintegral} and  \eqref{bhhomopfn_func_ns} with $i=p=0$ and   $i=1, p=0$ in the above equality yield $\|\varphi_2(t)\|\lesssim t^{2-\alpha}\|u_0\|$ and this leads to 
\begin{align*}
\|\varphi_2(t)\|^2+\int_{0}^{t}\|\varphi_2(s)\|^2\,ds\lesssim t^{4-2\alpha}\|u_0\|^2.    
\end{align*}
To bound the term $\varphi_2'(s)$, we see that
\begin{align}\label{bhbnd_6}
\varphi_2'(s)=2s\varphi(s)+s^2\varphi'(s) \quad \text{ with }\; \varphi(s)=\mathcal{I}^{1-\alpha}u'(s).
\end{align}
Differentiating $\varphi(s)$ with respect to $s$ we get
\begin{align*}
\begin{aligned}
\varphi'(s)&= \frac{1}{\Gamma(1-\alpha)}\int_{0}^{s}\tau^{-\alpha}u''(s-\tau)\,ds + \frac{s^{-\alpha}u'(0)}{\Gamma(1-\alpha)}
= \mathcal{I}^{1-\alpha}u''(s)+ \kappa_{1-\alpha}(s)u'(0).
\end{aligned}
\end{align*}
Thus, in view of (\ref{bhminiest2}) we arrive at
\begin{align*}
s^2\varphi'(s)&= \mathcal{I}^{1-\alpha}(s^2u')'-2\alpha\mathcal{I}^{1-\alpha}(su')-\alpha(1-\alpha)\mathcal{I}^{2-\alpha}(u')-s^2\kappa_{1-\alpha}(s)u'(0)+ s^2\kappa_{1-\alpha}(s)u'(0)\\
&=\mathcal{I}^{1-\alpha}(s^2u')'-2\alpha\mathcal{I}^{1-\alpha}(su')-\alpha(1-\alpha)\mathcal{I}^{2-\alpha}(u').
\end{align*}
Hence 
$\|s^2\varphi'(s)\| \lesssim \|\mathcal{I}^{1-\alpha}(s^2u')'\|+\|\mathcal{I}^{1-\alpha}(su')\|+\|\mathcal{I}^{2-\alpha}(u')\|$.
The definition of $\mathcal{I}^{1-\alpha}$ and   (\ref{bhhomopfn_func_ns}) first with $i=1$, $p=0$, next with $i=0$, $p=0$, and (\ref{bhhomopfn_ddtve_ns}) lead to 
\begin{align*}
\begin{aligned}
\|s^2\varphi'(s)\| &\lesssim \int_{0}^{s}(s-\tau)^{-\alpha}\big(\norm[0]{\tau u'(\tau)}+\norm[0]{\tau^2u''(\tau)}+\|u(\tau)\|+\|u_0\|\big)\,d\tau\lesssim s^{1-\alpha}\|u_0\|.
\end{aligned}
\end{align*}
This and \textit{Step 3} of  Lemma \ref{bhlemma2}  in (\ref{bhbnd_6}) show 
\begin{align*}
&\norm[0]{\varphi_2'(s)}\lesssim \norm[0]{s^2\varphi'(s)}+\norm[0]{s\varphi(s)}\lesssim  s^{1-\alpha}\|u_0\|.
\end{align*}
This shows
$\displaystyle \int_{0}^{t}\norm[0]{\varphi_2'(s)}^2\,ds\lesssim t^{3-2\alpha}\|u_0\|^2.$

\noindent{\textit{Step 2 (Smooth data).}} For $f\ne 0$ and $u_0\in D(A)$, follow the steps to bound $\|\widehat{\varphi_1}(t)\|^2$ for smooth data in the \textit{Step 3} in Lemma \ref{bhlemma2} and recall \eqref{beta} to arrive at 
\begin{align*}
\|\varphi_2(t)\|^2\lesssim \Big( t^2\norm[0]{u_0}_{D(A)}&+B(\alpha\epsilon/2,1-\alpha)t^{2-\alpha(1-\epsilon/2)}\norm[0]{f(0)}_{D(A^{\epsilon/2})}\\
&\quad+\epsilon^{-1}t^{3-\alpha(1-\epsilon/2)}\norm[0]{f'}_{L^{\infty}(0,T;D(A^{\epsilon/2}))}  \Big)^2.
\end{align*}
It is easy to establish a similar bound for $\int_{0}^{t}\|\varphi_2(s)\|^2\,ds$. 
The approach in \textit{Step 1} of this lemma and the stability properties in \eqref{bhhomopfn_dtve_s} with $ p=0$, \eqref{bhhomopfn_ddtve_s}, \eqref{bhnonhomo_dtve} with $q=\epsilon$, and \eqref{bhnonhomo_ddtve}  show
\begin{align*}
\int_{0}^{t}\|\varphi_2'(s)\|^2\,ds\lesssim  \Big(h^2t^{\alpha/2-2}B_1(\epsilon,t)\Big)^2.  
\end{align*}
The definition of $f_2$ and \eqref{algebraic} establish
\begin{align*}
 \|f_2(t)\|^2 +\int_{0}^{t}(\|f_2(s)\|^2 +\|f_2'(s)\|^2 )\,ds  \lesssim  \Big( t^2\|f(t)\|+t\|f\|_{L^2(0,t;L^2(\Omega))}+t^2\|f'\|_{L^2(0,t;L^2(\Omega))} \Big)^2.   
\end{align*}
A combination of {\it Steps} 1 and 2 establishes the assertion.
\end{proof}
\noindent We are now ready to prove Theorem \ref{bh_thetafnlest}. \\

\noindent{\it{Proof of Theorem \ref{bh_thetafnlest}}.} 
A substitution of the bounds from
Lemmas \ref{bhlemma1}-\ref{bhlemma4} in the key inequality from Lemma \ref{keylem_bhthetafinlthm} and algebraic manipulations establish
\begin{align*}
\norm[0]{\theta(t)}+t^{\alpha/2}\|\theta(t)\|_h\lesssim  
\begin{cases}
h^{2\gamma_0}t^{-\alpha}\norm[0]{u_0} +h^2t^{-(\alpha+1)/2}\norm[0]{u_0}, \;\; \text{ for }\; u_0\in L^2(\Omega), f=0,  \\
h^{2\gamma_0}\mathcal{B}_0(\epsilon,t) +h^2t^{\alpha/2-2}\mathcal{B}_1(\epsilon,t), \;\; \text{ for }\; u_0\in D(A), f\ne 0,
\end{cases}    
\end{align*}
and this concludes the proof. \hfill{$\Box$}

\section{Numerical illustrations}
The numerical experiments in this section validate the theoretical orders of convergences ($OCs$) for the semidiscrete solution established in Section \ref{bhsemierrsec}. 
Let $EOC$ denote the expected orders of convergence with respect to the space variable. The numerical experiments are performed using Freefem++ \cite{hecht12} with the following two sets of problem data on $\Omega=(0,1)^2$. To compute the discrete solution, we first triangulate $\overline{\Omega}$,
and then consider a uniform partition of $[0,t]$ with grid points $t_n=nk$, $n=0,1\ldots, N$, where $k=t/N$ is the time step size and $t$ is the time of interest. The Caputo fractional derivative $\partial_t^{\alpha}u(t)$ at $t=t_n$ is approximated by the L1 scheme (cf. \cite{sunwu06, linxu07}) as shown below:
\begin{align*}
\begin{aligned}
&\partial_t^\alpha u(t_n)= \frac{1}{\Gamma(1-\alpha)}\sum_{j=0}^{n-1}\int_{t_j}^{t_{j+1}}(t_n-s)^{-\alpha}\frac{\partial u(s)}{\partial s}\,ds \\
&\approx \frac{1}{\Gamma(1-\alpha)}\sum_{j=0}^{n-1}\frac{u(t_{j+1})-u(t_j)}{k}\int_{t_j}^{t_{j+1}}(t_n-s)^{-\alpha}\,ds
=\sum_{j=0}^{n-1}l_j\frac{u(t_{n-j})-u(t_{n-j-1})}{k^{\alpha}}\\
&=k^{-\alpha}\Big( l_0u(t_n)-l_{n-1}u(t_0)+\sum_{i=1}^{n-1}(l_j-l_{j-1})u(t_{n-j})\Big),
\end{aligned}
\end{align*}
where the weights $l_j$ are given by $l_j=((j+1)^{1-\alpha}-j^{1-\alpha})/\Gamma(2-\alpha),\;\;j=0,\ldots,N-1$. 

\subsection{Examples}
\begin{itemize}
\item[(i)] Choose the nonsmooth initial data as
\begin{align*}
u_0(x,y)=
\begin{cases}
1, & x\in (0,1/2], \; y\in (0,1),\\
-1, & x\in (1/2,1), \; y\in (0,1),
\end{cases}
\end{align*}
and $f(x,y,t)=0$. In this case, the exact solution $u(x,y,t)$ is not known. 

\item[(ii)] The manufactured exact solution $u$ given by $u(x,y,t)=(t^{\alpha+1}+1)(x(1-x)y(1-y))^2$ leads to the smooth initial data $u_0(x,y)=(x(1-x)y(1-y))^2$ and source function $f(x,y,t)=\Gamma(\alpha+2)t(x(1-x)y(1-y))^2+(t^{\alpha+1}+1)(24(x^2-2x^3+x^4)+24(y^2-2y^3+y^4)+2(2-12x+12x^2)(2-12y+12y^2))$.  Note that the initial data is in $D(A)$ and the source function belongs to the appropriate function space mentioned in Theorem \ref{bhsemidisctesmooth_finl} $(ii)$ for any $\epsilon\in (0,1/4)$. 
\end{itemize}
 For the dG method, the numerical computations are performed with the choice of penalty parameters as $\sigma_{\rm dG}^1=\sigma_{\rm dG}^2=2$, whereas for the $C^0$IP, we choose $\sigma_{\rm IP}=8$. 
\subsection{Order of spatial convergence for nonsmooth initial data}
Since the exact solution is not known in this case, the $OC$ is calculated by the following formula
\begin{align*}
OC= \log(\|W_{2h}-W_{h}\|_*/\|W_{h}-W_{h/2}\|_*)/\log2,   
\end{align*}
where $\|\cdot\|_*$ denotes either the $L^2(\Omega)$- or energy norm and $W_{h}$ is the discrete solution with mesh size $h$. The spatial numerical experiments are performed for case (i) with the mesh size $h=\{1/16,1/32,1/64,1/128, 1/256\}$, and fractional order $\alpha=0.25,0.50,0.75$ keeping the time step size $k$ fixed at  $k=0.001$.

For the Morley method, the numerical results are illustrated in Table \ref{bhnonsmooth_semiM}. The solution plots for the dG method are displayed in Figure \ref{dgnon_fig} on $256\times 256$ mesh with $\alpha=0.25,0.5,0.75$ with fixed $k=0.001$. Finally, for case $C^0$IP method, the empirical results are shown in Table \ref{c0nonsmooth_semi}. The energy norm and $L^2$ norm errors demonstrate linear and quadratic orders of convergence, respectively. These findings agree with the theoretical convergence rates established in Theorem  \ref{bhsemidisctesmooth_finl} $(i)$. 
\vspace{-0.2in}
\begin{table}[h!]
\scriptsize
\caption{The $L^2$-$OC$ and energy-$OC$ for the Morley method in case (i) for $\alpha=0.25, 0.50, 0.75 $ at $t=0.1$  with $k=0.001$}
\centering
\label{bhnonsmooth_semiM}
\begin{tabular}{ c c   c c c c } 
\hline
\rule{0pt}{3ex} 
$h$ & $\alpha $ & errors in $L^2$-norm  & $L^2$-$OC$ & errors in energy norm  & energy-$OC$    \\
\specialrule{.1em}{.05em}{.05em}
\multirow{ 4}{*}{}
1/16 & 0.25 & 4.50057e-05 & $-$ &  9.30188e-03  & $-$   \\ 
1/32 & & 1.18977e-05  & 1.91942 & 4.69592e-03  & 0.98611 \\ 
1/64 & & 3.07517e-06 &  1.95195 & 2.34989e-03  & 0.99881  \\
1/128 & & 7.63686e-07 & 2.00961 & 1.17378e-03 & 1.00143 \\ 
\hline
\multirow{4}{*}{} 1/16 & 0.50 & 5.59721e-05 & $-$ & 1.14995e-02& $-$  \\
1/32 & & 1.47909e-05 & 1.92000  & 5.80462e-03 &  0.98630  \\
1/64 & & 3.82260e-06 & 1.95208  & 2.90464e-03  & 0.99884  \\
1/128 & & 9.49291e-07 & 2.00963  & 1.45088e-03  & 1.00142 \\
\hline
\multirow{4}{*}{} 1/16 & 0.75 & 5.12578e-05 & $-$ & 1.02222e-02& $-$    \\
1/32 & & 1.35183e-05 & 1.92286  & 5.15662e-03  & 0.98720  \\
1/64 & & 3.49206e-06 & 1.95276   & 2.58010e-03 & 0.99899  \\
1/128 & & 8.67171e-07 & 2.00969  & 1.28878e-03  & 1.00142  \\
\hline
$EOC$ & & & 2.0 & & 1.0 \\
\hline
\end{tabular}
\end{table}
\begin{figure}[h!]
    \centering
    \includegraphics[scale=0.18]{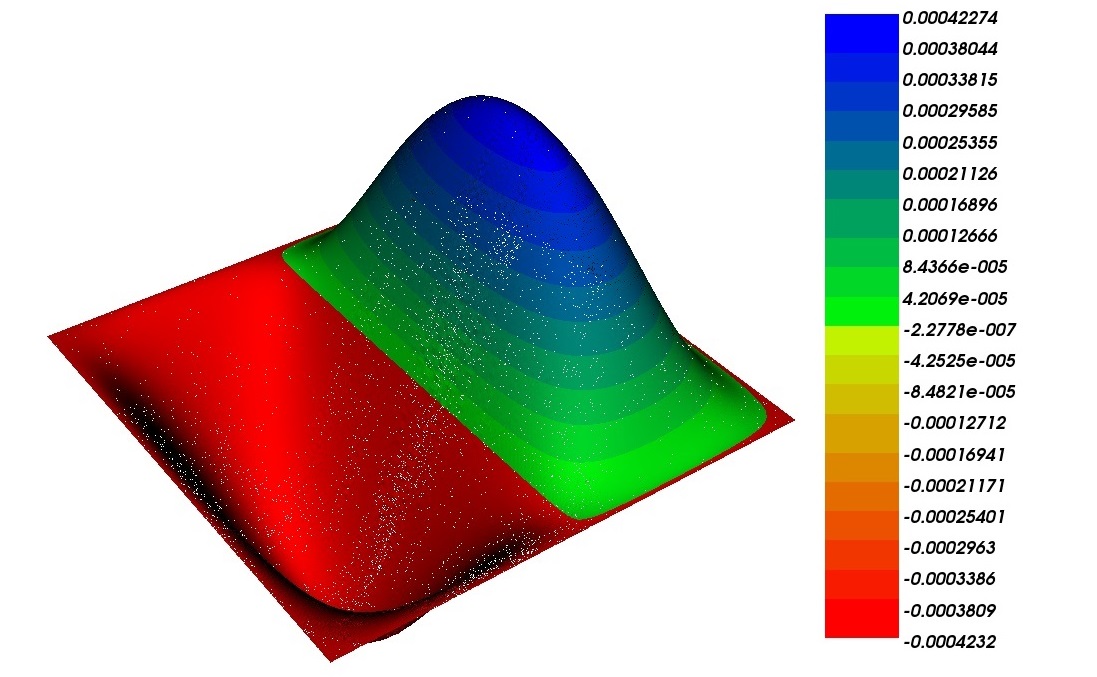}
    \includegraphics[scale=0.18]{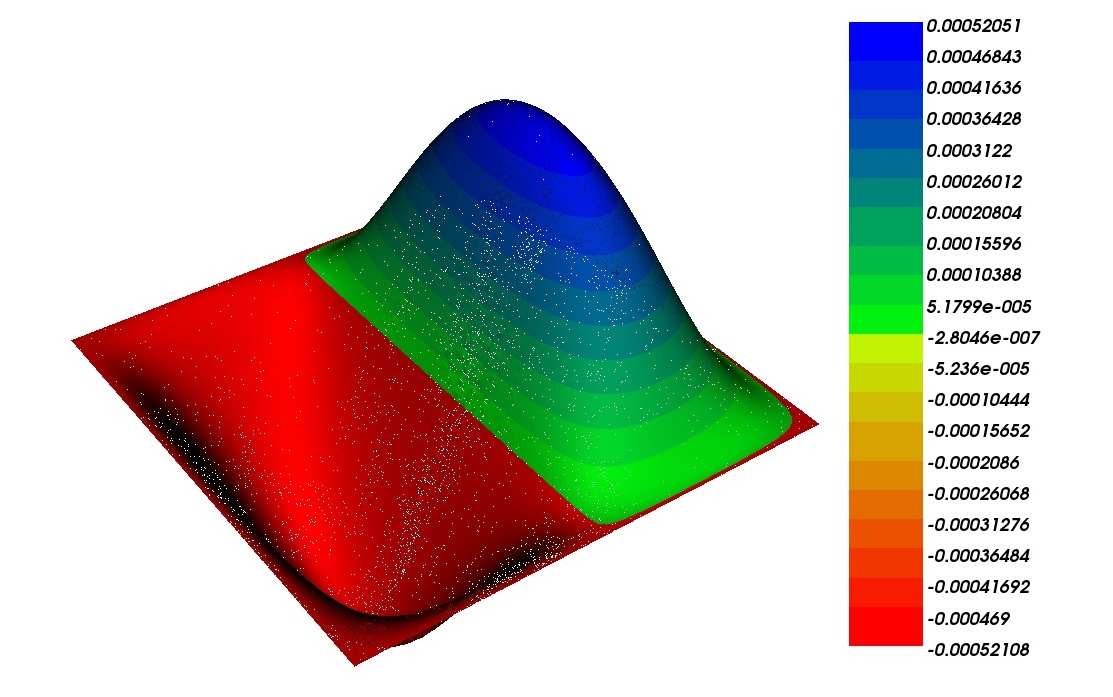}
    \includegraphics[scale=0.18]{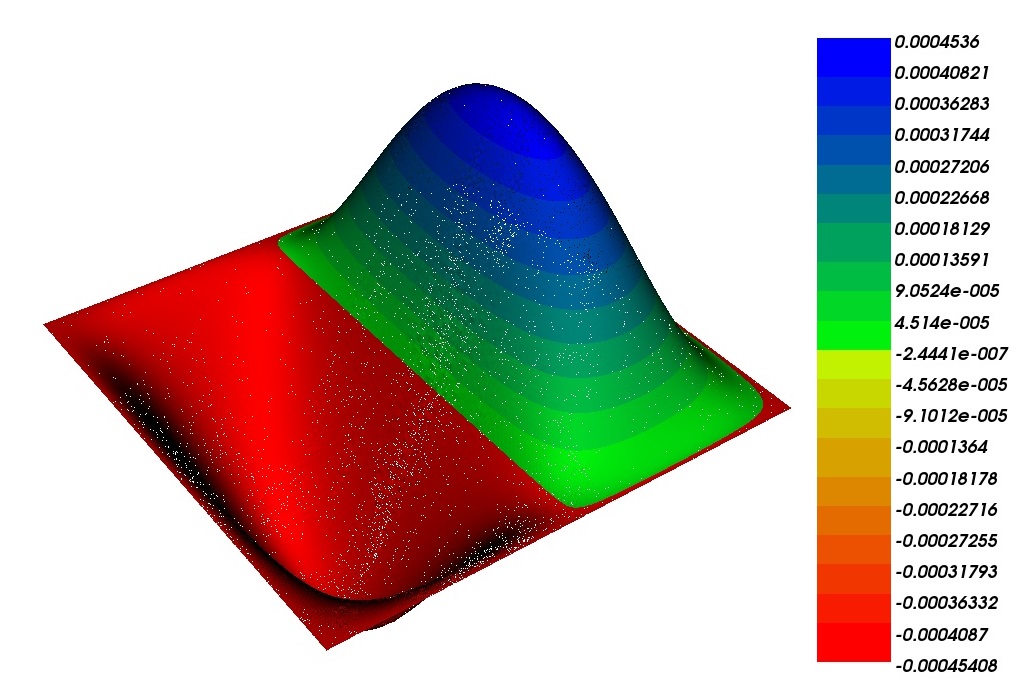}
    \caption{Numerical solutions of dG method on $256\times 256$ mesh with $\alpha=0.25,0.5, 0.75$.}
    \label{dgnon_fig}
\end{figure}
\begin{table}[H]
\scriptsize
\caption{The $L^2$-$OC$ and energy-$OC$ for the $C^0$IP method in case (i) for $\alpha=0.25, 0.50, 0.75 $ at $t=0.1$  with $k=0.001$ and $\sigma_{\rm IP}=8$}
\centering
\label{c0nonsmooth_semi}
\begin{tabular}{ c c   c c c c } 
\hline
\rule{0pt}{3ex} 
$h$ & $\alpha$ & errors in $L^2$-norm & $L^2$-$OC$ & errors in energy-norm  & energy-$OC$    \\
\specialrule{.1em}{.05em}{.05em}
\multirow{ 4}{*}{}
1/16 & 0.25 & 5.51766e-06 & $-$ & 3.13001e-03 & $-$   \\ 
1/32 & & 1.56160e-06 & 1.82103  & 1.59248e-03 & 0.97489 \\ 
1/64 & & 4.02967e-07 & 1.95429  & 7.94704e-04 & 1.00278 \\
1/128 & & 1.05072e-07 & 1.93928 & 4.02570e-04 & 0.98117 \\ 
\hline
\multirow{4}{*}{} 1/16 & 0.50 & 6.79515e-06 & $-$ & 3.85380e-03& $-$  \\
1/32 & & 1.92315e-06 & 1.82103 & 1.96072e-03 &  0.97489  \\
1/64 & & 4.96266e-07 & 1.95429 & 9.78472e-04 &  1.00278  \\
1/128 & & 1.29399e-07 & 1.93929 & 4.95661e-04 & 0.98117 \\
\hline
\multirow{4}{*}{} 1/16 & 0.75 &  5.92698e-06 & $-$  & 3.35774e-03& $-$  \\
1/32 & & 1.67746e-06 & 1.82102 & 1.70834e-03 & 0.97490 \\
1/64 & & 4.32868e-07 & 1.95428 & 8.52520e-04 & 1.00279  \\
1/128 & & 1.12867e-07 & 1.93930 & 4.31861e-04 & 0.98116 \\
\hline
$EOC$ & & & 2.0& & 1.0 \\
\hline
\end{tabular}
\end{table}
\subsection{Order of spatial convergence for smooth data}
The spatial numerical experiments are performed with the mesh size  $h=\{\frac{1}{12},\frac{1}{24},\frac{1}{48},\frac{1}{96}\}$, and fractional order $\alpha=0.25,0.50,0.75$ with a fixed time step size $k=0.001$.
For the Morley method, the results are presented in Table \ref{bhsmooth_semiM}. The convergence results for the dG scheme are demonstrated in Table \ref{dgsmooth_semi}. Further, the exact and numerical solutions plots are also shown in Figure \ref{dg_fig}. Table \ref{c0smooth_semi} and Figure \ref{c0_fig} display the numerical results and the exact and numerical solutions for the $C^0$IP method. From the experiments, we observe that the empirical results are independent of the fractional order $\alpha$. The energy norm and $L^2$ norm errors for the Morley, dG, and $C^0$IP methods are plotted in a log-log scale in Figure \ref{crfigures}. This validates the theoretical $OCs$ of Theorem \ref{bhsemidisctesmooth_finl} $(ii)$. 
\vspace{0.25in}
\begin{table}[h!]
\scriptsize
\caption{The $L^2$-$OC$ and energy-$OC$ for the Morley method in case (ii) for $\alpha=0.25, 0.50, 0.75 $ at $t=0.1$  with $k=0.001$}
\centering
\label{bhsmooth_semiM}
\begin{tabular}{ c c   c c c c } 
\hline
\rule{0pt}{3ex} 
$h$ & $\alpha$ & errors in $L^2$-norm & $L^2$-$OC$ & errors in energy-norm  & energy-$OC$    \\
\specialrule{.1em}{.05em}{.05em}
\multirow{ 4}{*}{}
1/12 & 0.25 & 2.03729e-04 & $-$ & 2.06933e-02 & $-$   \\ 
1/24 & & 5.22317e-05 & 1.96365  & 1.04746e-02 & 0.98226 \\ 
1/48 & & 1.31474e-05 & 1.99015  & 5.25437e-03 & 0.99531  \\
1/96 & & 3.29262e-06 & 1.99747  & 2.62936e-03 & 0.99880 \\ 
\hline
\multirow{4}{*}{} 1/12 & 0.50 & 1.98919e-04 & $-$ & 2.02103e-02& $-$    \\
1/24 & & 5.09999e-05 & 1.96361 & 1.02305e-02 &  0.98222  \\
1/48 & & 1.28375e-05 & 1.99013 & 5.13193e-03 &  0.99530  \\
1/96 & & 3.21502e-06 & 1.99746 & 2.56809e-03 &  0.99880 \\
\hline
\multirow{4}{*}{} 1/12 & 0.75 & 1.96278e-04 & $-$  & 1.99395e-02& $-$  \\
1/24 & & 5.03223e-05 & 1.96363 & 1.00933e-02 & 0.98223  \\
1/48 & & 1.26670e-05 & 1.99013 & 5.06308e-03 & 0.99530 \\
1/96 & & 3.17243e-06 & 1.99741 & 2.53364e-03 & 0.99880 \\
\hline
$EOC$ & & & 2.0& & 1.0 \\
\hline
\end{tabular}
\end{table}
\vspace{-0.42in}
\begin{table}[h!]
\scriptsize
\caption{The $L^2$-$OC$ and energy-$OC$ for the dG method in  case (ii) for $\alpha=0.25, 0.50, 0.75 $ at $t=0.1$  with $k=0.001$ and $\sigma_{\rm dG}^1=\sigma_{\rm dG}^2=2$}
\centering
\label{dgsmooth_semi}
\begin{tabular}{ c c   c c c c } 
\hline
\rule{0pt}{3ex} 
$h$ & $\alpha$ & errors in $L^2$-norm & $L^2$-$OC$ & errors in energy-norm  & energy-$OC$    \\
\specialrule{.1em}{.05em}{.05em}
\multirow{ 4}{*}{}
1/12 & 0.25 & 1.03592e-03 & $-$ & 6.45313e-02 & $-$   \\ 
1/24 & & 2.92835e-04 & 1.82275  & 3.05090e-02 & 1.08076 \\ 
1/48 & & 7.85626e-05 & 1.89817  & 1.41902e-02 & 1.10434  \\
1/96 & & 2.03336e-05 & 1.94998  & 6.97876e-03 & 1.02385 \\ 
\hline
\multirow{4}{*}{} 1/12 & 0.50 & 1.01130e-03 & $-$ & 6.30133e-02& $-$    \\
1/24 & & 2.85915e-04 & 1.82256 & 2.97960e-02 &  1.08054 \\
1/48 & & 7.67092e-05 & 1.89812 & 1.38593e-02 &  1.10426 \\
1/96 & & 1.98541e-05 & 1.94996 & 6.81611e-03 &  1.02383 \\
\hline
\multirow{4}{*}{} 1/12 & 0.75 & 9.97939e-04 & $-$  & 6.21740e-02& $-$  \\
1/24 & & 2.82122e-04 & 1.82263 & 2.93972e-02 & 1.08063  \\
1/48 & & 7.56902e-05 & 1.89814 & 1.36735e-02 & 1.10430  \\
1/96 & & 1.95904e-05 & 1.94996 & 6.72469e-03 & 1.02384 \\
\hline
$EOC$ & & & 2.0& & 1.0 \\
\hline
\end{tabular}
\end{table}
\begin{figure}[H]
    \centering
    \includegraphics[scale=0.20]{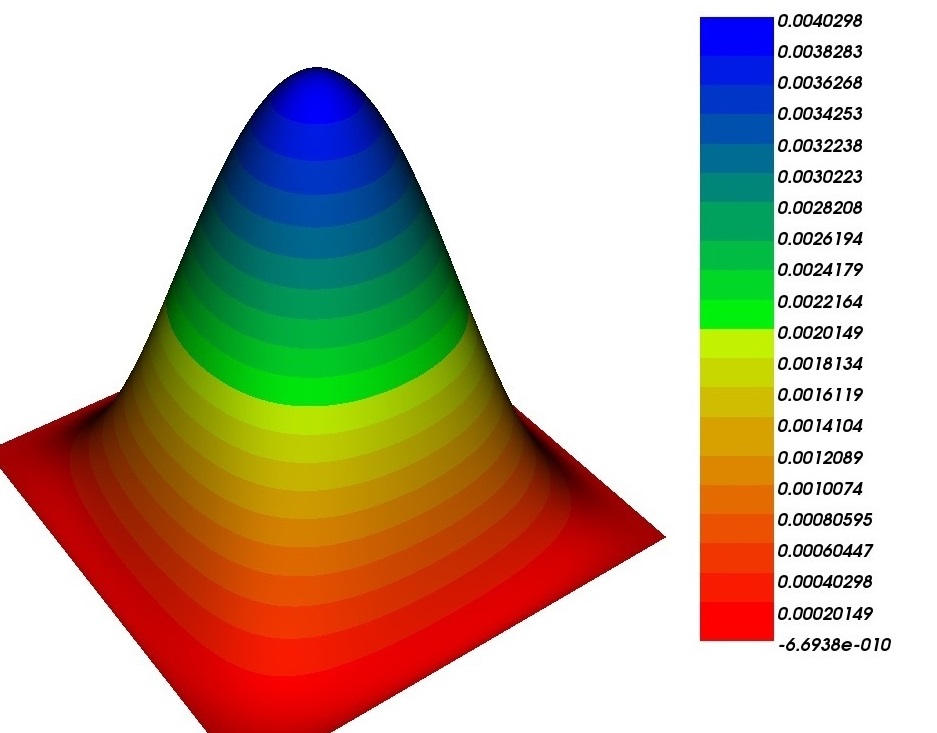}
    \includegraphics[scale=0.20]{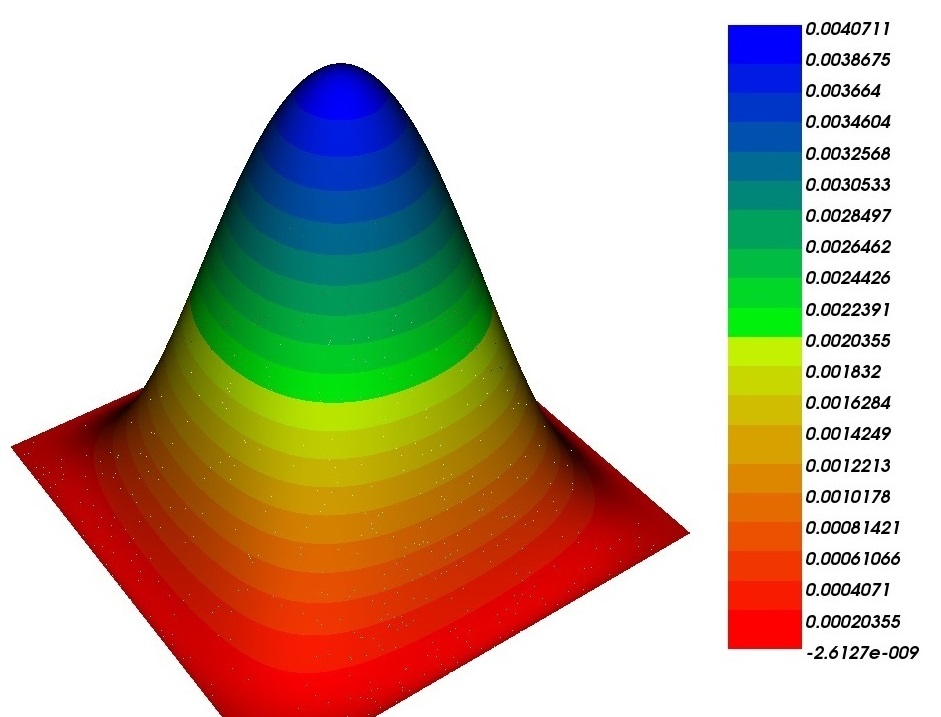}
    \caption{Exact (left) and numerical (right) solutions of dG method on $96\times 96$ mesh, $\alpha=0.5$.}
    \label{dg_fig}
\end{figure}
\vspace{-0.38in}
\begin{table}[h!]
\scriptsize
\caption{The $L^2$-$OC$ and energy-$OC$ for the $C^0$IP method in case (ii) for $\alpha=0.25, 0.50, 0.75 $ at $t=0.1$  with $k=0.001$ and $\sigma_{\rm IP}=8$}
\centering
\label{c0smooth_semi}
\begin{tabular}{ c c   c c c c } 
\hline
\rule{0pt}{3ex} 
$h$ & $\alpha$ & errors in $L^2$-norm & $L^2$-$OC$ & errors in energy-norm  & energy-$OC$    \\
\specialrule{.1em}{.05em}{.05em}
\multirow{ 4}{*}{}
1/12 & 0.25 & 1.15038e-04 & $-$ & 1.72387e-02 & $-$   \\ 
1/24 & & 3.29403e-05 & 1.80419  & 8.87466e-03 & 0.95788 \\ 
1/48 & & 8.67545e-06 & 1.92484  & 4.47807e-03 & 0.98681  \\
1/96 & & 2.21587e-06 & 1.96907  & 2.24645e-03 & 0.99522 \\ 
\hline
\multirow{4}{*}{} 1/12 & 0.50 & 1.12328e-04 & $-$ & 1.68369e-02& $-$    \\
1/24 & & 3.21638e-05 & 1.80421 & 8.66785e-03 &  0.95788  \\
1/48 & & 8.47091e-06 & 1.92485 & 4.37372e-03 &  0.98681  \\
1/96 & & 2.16363e-06 & 1.96906 & 2.19411e-03 &  0.99522 \\
\hline
\multirow{4}{*}{} 1/12 & 0.75 & 1.10834e-04 & $-$  & 1.66111e-02& $-$  \\
1/24 & & 3.17360e-05 & 1.80421 & 8.55157e-03 & 0.95788  \\
1/48 & & 8.35815e-06 & 1.92487 & 4.31505e-03 & 0.98681  \\
1/96 & & 2.13470e-06 & 1.96915 &  2.16467e-03 & 0.99522 \\
\hline
$EOC$ & & & 2.0& & 1.0 \\
\hline
\end{tabular}
\end{table}
\vspace{-0.28in}
\begin{figure}[H]
    \centering
    \includegraphics[scale=0.20]{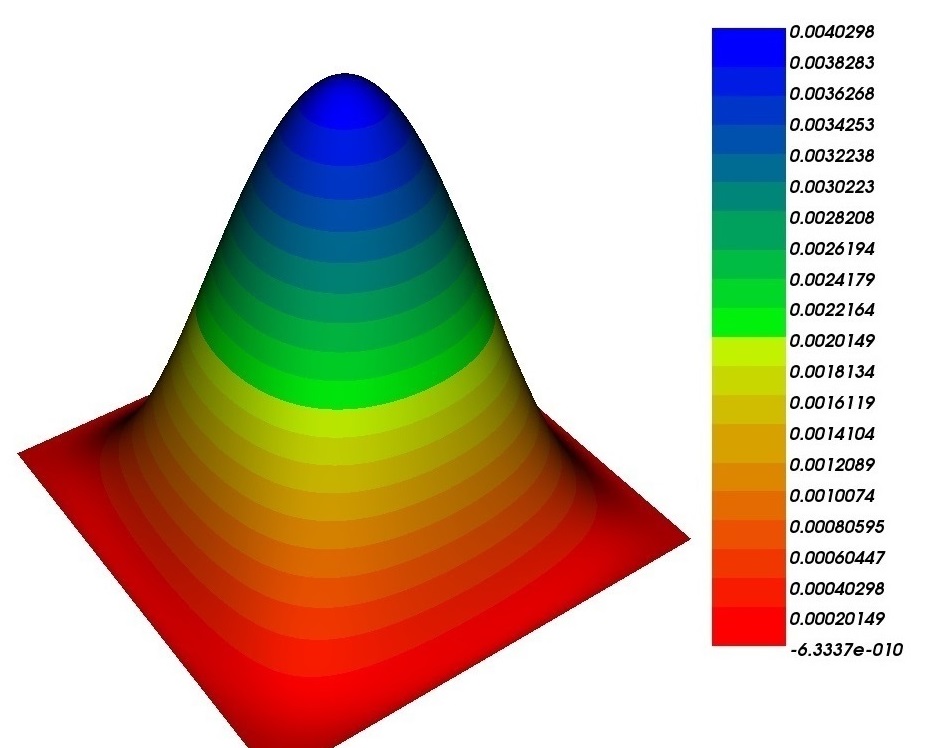}
    \includegraphics[scale=0.20]{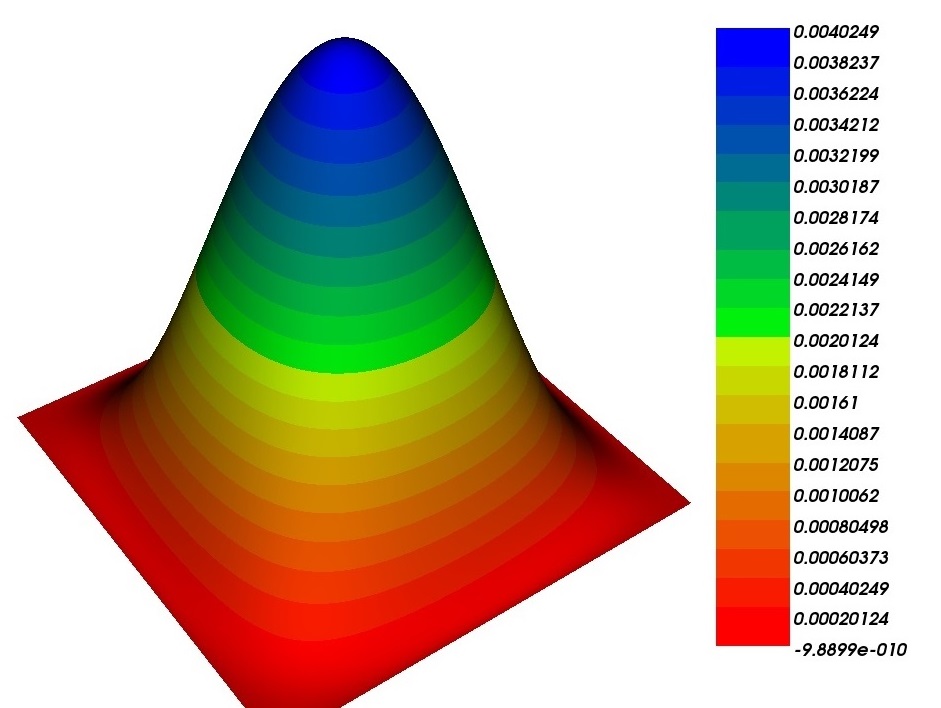}
    \caption{Exact (left) and numerical (right) solutions of $C^0$IP method on $96\times 96$ mesh, $\alpha=0.5$.}
    \label{c0_fig}
\end{figure}
\vspace{-0.3in}
\begin{figure}[H]
    \centering
    \includegraphics[scale=0.5]{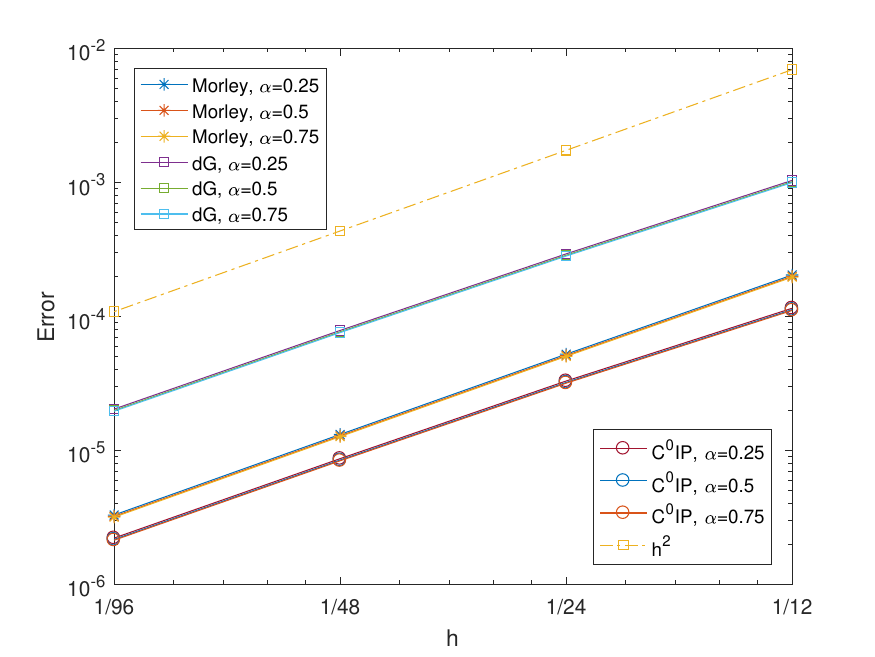}
    \includegraphics[scale=0.5]{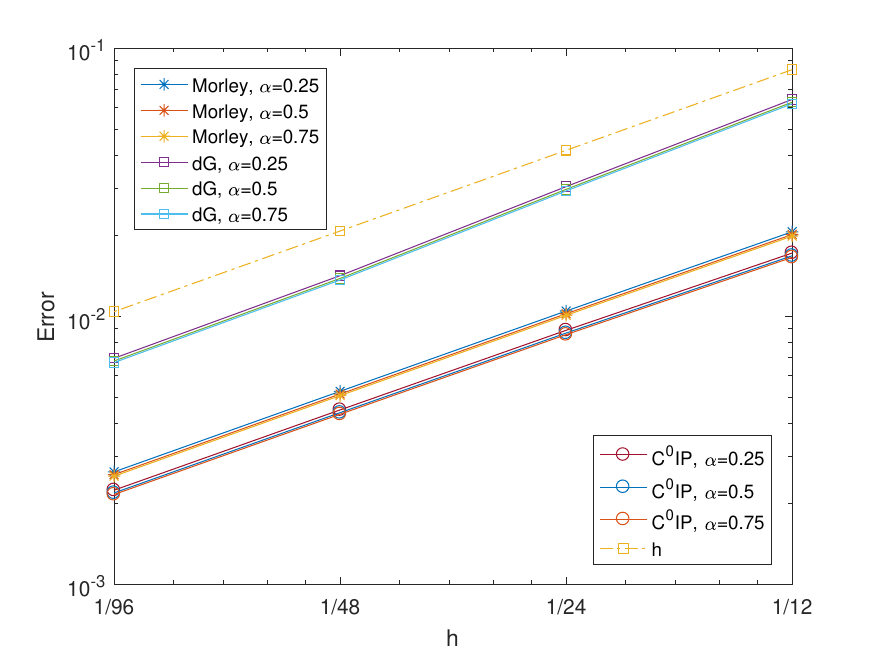}
    \caption{$L^2$ errors (left) and energy errors (right) plots for case (ii) for the Morley, dG, and $C^0$IP schemes at $t=0.1$ with $k=0.001$}
    \label{crfigures}
\end{figure}
\section{Concluding remarks}
 In this paper, we have studied an initial-boundary value problem for a time-fractional biharmonic problem with clamped boundary conditions in a bounded polygonal domain with a Lipschitz continuous boundary in $\mathbb{R}^2$. After defining a weak solution, we have proved the well-posedness result of the problem and derived several regularity results of the solutions of both homogeneous and nonhomogeneous problems which are useful in the error analysis. Using an energy argument a spatially semidiscrete scheme that covers various popular lowest-order nonstandard piecewise quadratic finite element schemes (e.g., the Morley, discontinuous Galerkin, and $C^0$ interior penalty) is developed and analyzed. The convergence analysis is carried out for smooth and nonsmooth initial data cases, including the initial data $u_0\in L^2(\Omega)$. Numerical results are provided to validate the theoretical convergence rates of the discrete solution. One future direction is the convergence analysis for time discretization, which will be addressed in another work.

\noindent{\bf Acknowledgements.}
 The second and third authors acknowledge the support from the IFCAM project “Analysis, Control and Homogenization of Complex Systems”. The second author acknowledges the support from SERB project on finite element methods for phase field crystal equation, code (RD/0121-SERBF30-003).

\bibliography{arxiv}

\begin{bibdiv}
\begin{biblist}

\bib{pg2020}{article}{
      author={Abbaszadeh, M.},
      author={Dehghan, M.},
       title={Direct meshless local {P}etrov-{G}alerkin ({DMLPG}) method for time-fractional fourth-order reaction-diffusion problem on complex domains},
        date={2020},
        ISSN={0898-1221},
     journal={Comput. Math. Appl.},
      volume={79},
      number={3},
       pages={876\ndash 888},
         url={https://doi.org/10.1016/j.camwa.2019.08.001},
}

\bib{maskarikaraa22}{article}{
      author={Al-Maskari, M.},
      author={Karaa, S.},
       title={The time-fractional {C}ahn-{H}illiard equation: analysis and approximation},
        date={2022},
        ISSN={0272-4979},
     journal={IMA J. Numer. Anal.},
      volume={42},
      number={2},
       pages={1831\ndash 1865},
}

\bib{babuskaosborn91}{incollection}{
      author={Babu\v{s}ka, I.},
      author={Osborn, J.},
       title={Eigenvalue {P}roblems},
        date={1991},
   booktitle={Handbook of {N}umerical {A}nalysis, {V}ol. {II}},
      series={Handb. Numer. Anal., II},
   publisher={North-Holland, Amsterdam},
       pages={641\ndash 787},
}

\bib{Bensoussan-etal-2007}{book}{
      author={Bensoussan, A.},
      author={D.~Prato, G.},
      author={Delfour, M.~C.},
      author={Mitter, S.~K.},
       title={Representation and {C}ontrol of {I}nfinite {D}imensional {S}ystems},
     edition={Second},
      series={Systems \& Control: Foundations \& Applications},
   publisher={Birkh\"{a}user Boston, Inc., Boston, MA},
        date={2007},
        ISBN={978-0-8176-4461-1; 0-8176-4461-X},
         url={https://doi.org/10.1007/978-0-8176-4581-6},
}

\bib{blumrannacher80}{article}{
      author={Blum, H.},
      author={Rannacher, R.},
       title={On the boundary value problem of the biharmonic operator on domains with angular corners},
        date={1980},
        ISSN={0170-4214},
     journal={Math. Methods Appl. Sci.},
      volume={2},
      number={4},
       pages={556\ndash 581},
}

\bib{brennersung2005}{article}{
      author={Brenner, S.~C.},
      author={Sung, L.-Y.},
       title={{$C^0$} interior penalty methods for fourth order elliptic boundary value problems on polygonal domains},
        date={2005},
        ISSN={0885-7474,1573-7691},
     journal={J. Sci. Comput.},
      volume={22/23},
       pages={83\ndash 118},
}

\bib{brezis}{book}{
      author={Brezis, H.},
       title={Analyse {F}onctionnelle},
   publisher={Masson, Paris},
        date={1983},
        ISBN={2-225-77198-7},
}

\bib{cgn2015}{article}{
      author={Carstensen, C.},
      author={Gallistl, D.},
      author={Nataraj, N.},
       title={Comparison results of nonstandard {$P_2$} finite element methods for the biharmonic problem},
        date={2015},
        ISSN={2822-7840},
     journal={ESAIM Math. Model. Numer. Anal.},
      volume={49},
      number={4},
       pages={977\ndash 990},
}

\bib{cn2021cmam}{article}{
      author={Carstensen, C.},
      author={Nataraj, N.},
       title={A priori and a posteriori error analysis of the {C}rouzeix-{R}aviart and {M}orley {FEM} with original and modified right-hand sides},
        date={2021},
        ISSN={1609-4840},
     journal={Comput. Methods Appl. Math.},
      volume={21},
      number={2},
       pages={289\ndash 315},
         url={https://doi.org/10.1515/cmam-2021-0029},
}

\bib{cn2022}{article}{
      author={Carstensen, C.},
      author={Nataraj, N.},
       title={Lowest-order equivalent nonstandard finite element methods for biharmonic plates},
        date={2022},
        ISSN={2822-7840},
     journal={ESAIM Math. Model. Numer. Anal.},
      volume={56},
      number={1},
       pages={41\ndash 78},
         url={https://doi.org/10.1051/m2an/2021085},
}

\bib{carsput2020}{article}{
      author={Carstensen, C.},
      author={Puttkammer, S.},
       title={How to prove the discrete reliability for nonconforming finite element methods},
        date={2020},
        ISSN={0254-9409},
     journal={J. Comput. Math.},
      volume={38},
      number={1},
       pages={142\ndash 175},
         url={https://doi.org/10.4208/jcm.1908-m2018-0174},
}

\bib{chenshih98}{book}{
      author={Chen, C.},
      author={Shih, T.},
       title={Finite {E}lement {M}ethods for {I}ntegrodifferential {E}quations},
      series={Series on Applied Mathematics},
   publisher={World Scientific Publishing Co., Inc., River Edge, NJ},
        date={1998},
      volume={9},
        ISBN={981-02-3263-2},
}

\bib{ciarlet78}{book}{
      author={Ciarlet, P.~G.},
       title={The {F}inite {E}lement {M}ethod for {E}lliptic {P}roblems},
      series={Studies in Mathematics and its Applications, Vol. 4},
   publisher={North-Holland Publishing Co., Amsterdam-New York-Oxford},
        date={1978},
        ISBN={0-444-85028-7},
}

\bib{cuifd2021}{article}{
      author={Cui, M.},
       title={A compact difference scheme for time-fractional {D}irichlet biharmonic equation on temporal graded meshes},
        date={2021},
        ISSN={2079-7362},
     journal={East Asian J. Appl. Math.},
      volume={11},
      number={1},
       pages={164\ndash 180},
         url={https://doi.org/10.4208/eajam.270520.210920},
}

\bib{danumjayaetal2021}{article}{
      author={Danumjaya, P.},
      author={Pany, A.~K.},
      author={Pani, A.~K.},
       title={Morley {FEM} for the fourth-order nonlinear reaction-diffusion problems},
        date={2021},
        ISSN={0898-1221,1873-7668},
     journal={Comput. Math. Appl.},
      volume={99},
       pages={229\ndash 245},
}

\bib{duetal2017dg}{article}{
      author={Du, Y.},
      author={Liu, Y.},
      author={Li, H.},
      author={Fang, Z.},
      author={He, S.},
       title={Local discontinuous {G}alerkin method for a nonlinear time-fractional fourth-order partial differential equation},
        date={2017},
        ISSN={0021-9991},
     journal={J. Comput. Phys.},
      volume={344},
       pages={108\ndash 126},
         url={https://doi.org/10.1016/j.jcp.2017.04.078},
}

\bib{evans10}{book}{
      author={Evans, L.~C.},
       title={Partial {D}ifferential {E}quations},
   publisher={Vol. 19, American Mathematical Society, Providence, RI},
        date={2010},
}

\bib{gal2015}{article}{
      author={Gallistl, D.},
       title={Morley finite element method for the eigenvalues of the biharmonic operator},
        date={2015},
        ISSN={0272-4979},
     journal={IMA J. Numer. Anal.},
      volume={35},
      number={4},
       pages={1779\ndash 1811},
         url={https://doi.org/10.1093/imanum/dru054},
}

\bib{gorenflomainardi97}{incollection}{
      author={Gorenflo, R.},
      author={Mainardi, F.},
       title={Fractional calculus: {I}ntegral and differential equations of fractional order},
        date={1997},
   booktitle={Fractals and {F}ractional {C}alculus in {C}ontinuum {M}echanics ({U}dine, 1996)},
      series={CISM Courses and Lect.},
      volume={378},
   publisher={Springer, Vienna},
       pages={223\ndash 276},
}

\bib{grisvard92}{book}{
      author={Grisvard, P.},
       title={Singularities in {B}oundary {V}alue {P}roblems},
      series={Recherches en Math\'{e}matiques Appliqu\'{e}es [Research in Applied Mathematics]},
   publisher={Vol 22, Masson, Paris; Springer-Verlag, Berlin},
        date={1992},
        ISBN={2-225-82770-2},
}

\bib{hadhoud19}{article}{
      author={Hadhoud, A.~R.},
       title={Quintic non-polynomial spline method for solving the time fractional biharmonic equation},
        date={2019},
     journal={Appl. Math. Inf. Sci.},
      volume={13},
      number={3},
       pages={507\ndash 513},
         url={https://doi.org/10.18576/amis/130323},
}

\bib{hecht12}{article}{
      author={Hecht, F.},
       title={New development in freefem++},
        date={2012},
        ISSN={1570-2820},
     journal={J. Numer. Math.},
      volume={20},
      number={3-4},
       pages={251\ndash 265},
         url={https://doi.org/10.1515/jnum-2012-0013},
}

\bib{huangetal22}{article}{
      author={Huang, C.},
      author={An, N.},
      author={Chen, H.},
       title={Local {$H^1$}-norm error analysis of a mixed finite element method for a time-fractional biharmonic equation},
        date={2022},
        ISSN={0168-9274},
     journal={Appl. Numer. Math.},
      volume={173},
       pages={211\ndash 221},
         url={https://doi.org/10.1016/j.apnum.2021.12.004},
}

\bib{huangstynes21}{article}{
      author={Huang, C.},
      author={Stynes, M.},
       title={{$\alpha$}-robust error analysis of a mixed finite element method for a time-fractional biharmonic equation},
        date={2021},
        ISSN={1017-1398},
     journal={Numer. Algorithms},
      volume={87},
      number={4},
       pages={1749\ndash 1766},
         url={https://doi.org/10.1007/s11075-020-01036-y},
}

\bib{jafarietal08}{article}{
      author={Jafari, H.},
      author={Dehghan, M.},
      author={Sayevand, K.},
       title={Solving a fourth-order fractional diffusion-wave equation in a bounded domain by decomposition method},
        date={2008},
        ISSN={0749-159X},
     journal={Numer. Methods Partial Differential Equations},
      volume={24},
      number={4},
       pages={1115\ndash 1126},
         url={https://doi.org/10.1002/num.20308},
}

\bib{jlpz15}{article}{
      author={Jin, B.},
      author={Lazarov, R.},
      author={Pasciak, J.},
      author={Zhou, Z.},
       title={Error analysis of semidiscrete finite element methods for inhomogeneous time-fractional diffusion},
        date={2015},
        ISSN={0272-4979},
     journal={IMA J. Numer. Anal.},
      volume={35},
      number={2},
       pages={561\ndash 582},
         url={https://doi.org/10.1093/imanum/dru018},
}

\bib{karaamustaphapani18}{article}{
      author={Karaa, S.},
      author={Mustapha, K.},
      author={Pani, A.~K.},
       title={Optimal error analysis of a \uppercase{FEM} for fractional diffusion problems by energy arguments},
        date={2018},
     journal={J. Sci. Comput.},
      volume={74},
      number={1},
       pages={519\ndash 535},
}

\bib{kilbasbook}{book}{
      author={Kilbas, A.~A.},
      author={Srivastava, H.~M.},
      author={Trujillo, J.~J.},
       title={Theory and \uppercase{A}pplications of \uppercase{F}ractional \uppercase{D}ifferential \uppercase{E}quations},
   publisher={Elsevier},
     address={Amsterdam},
        date={2006},
      volume={204},
}

\bib{lemcleanmustapha16}{article}{
      author={Le, K.~N.},
      author={McLean, W.},
      author={Mustapha, K.},
       title={Numerical solution of the time-fractional {F}okker-{P}lanck equation with general forcing},
        date={2016},
     journal={SIAM J. Numer. Anal.},
      volume={54},
      number={3},
       pages={1763\ndash 1784},
}

\bib{lietal2022_cnc}{article}{
      author={Li, M.},
      author={Zhao, J.},
      author={Huang, C.},
      author={Chen, S.},
       title={Conforming and nonconforming {VEM}s for the fourth-order reaction-subdiffusion equation: a unified framework},
        date={2022},
        ISSN={0272-4979},
     journal={IMA J. Numer. Anal.},
      volume={42},
      number={3},
       pages={2238\ndash 2300},
         url={https://doi.org/10.1093/imanum/drab030},
}

\bib{linxu07}{article}{
      author={Lin, Y.},
      author={Xu, C.},
       title={Finite difference/spectral approximations for the time-fractional diffusion equation},
        date={2007},
        ISSN={0021-9991},
     journal={J. Comput. Phys.},
      volume={225},
      number={2},
       pages={1533\ndash 1552},
         url={https://doi.org/10.1016/j.jcp.2007.02.001},
}

\bib{liuetal2015}{article}{
      author={Liu, Y.},
      author={Du, Y.},
      author={Li, H.},
      author={He, S.},
      author={Gao, W.},
       title={Finite difference/finite element method for a nonlinear time-fractional fourth-order reaction--diffusion problem},
        date={2015},
        ISSN={0898-1221},
     journal={Comput. Math. Appl.},
      volume={70},
      number={4},
       pages={573\ndash 591},
         url={https://doi.org/10.1016/j.camwa.2015.05.015},
}

\bib{liuetal15}{article}{
      author={Liu, Y.},
      author={Du, Y.},
      author={Li, H.},
      author={Li, J.},
      author={He, S.},
       title={A two-grid mixed finite element method for a nonlinear fourth-order reaction-diffusion problem with time-fractional derivative},
        date={2015},
        ISSN={0898-1221},
     journal={Comput. Math. Appl.},
      volume={70},
      number={10},
       pages={2474\ndash 2492},
         url={https://doi.org/10.1016/j.camwa.2015.09.012},
}

\bib{liuetal14}{article}{
      author={Liu, Y.},
      author={Fang, Z.},
      author={Li, H.},
      author={He, S.},
       title={A mixed finite element method for a time-fractional fourth-order partial differential equation},
        date={2014},
        ISSN={0096-3003},
     journal={Appl. Math. Comput.},
      volume={243},
       pages={703\ndash 717},
         url={https://doi.org/10.1016/j.amc.2014.06.023},
}

\bib{smrksenergy}{article}{
      author={Mahata, S.},
      author={Sinha, R.~K.},
       title={Nonsmooth data optimal error estimates by energy arguments for subdiffusion equations with memory},
        date={2022},
     journal={Adv. Comput. Math.},
      volume={48},
      number={4},
       pages={1\ndash 21},
}

\bib{mustapha18}{article}{
      author={Mustapha, K.},
       title={{FEM} for time-fractional diffusion equations, novel optimal error analyses},
        date={2018},
     journal={Math. Comp.},
      volume={87},
      number={313},
       pages={2259\ndash 2272},
}

\bib{mustaphaschotzau14}{article}{
      author={Mustapha, K.},
      author={Sch\"otzau, D.},
       title={Well-posedness of hp-version discontinuous {G}alerkin methods for fractional diffusion wave equations},
        date={2014},
     journal={IMA J. Numer. Anal.},
      volume={34},
      number={4},
       pages={1426\ndash 1446},
}

\bib{pazybook}{book}{
      author={Pazy, A.},
       title={Semigroups of {L}inear {O}perators and {A}pplications to {P}artial {D}ifferential {E}quations},
      series={Applied Mathematical Sciences},
   publisher={Springer-Verlag, New York},
        date={1983},
      volume={44},
        ISBN={0-387-90845-5},
         url={https://doi.org/10.1007/978-1-4612-5561-1},
}

\bib{podlubny}{book}{
      author={Podlubny, I.},
       title={Fractional {D}ifferential {E}quations},
   publisher={Academic Press, San Diego, CA},
        date={1999},
      volume={198},
}

\bib{sy11}{article}{
      author={Sakamoto, K.},
      author={Yamamoto, M.},
       title={Initial value/boundary value problems for fractional diffusion-wave equations and applications to some inverse problems},
        date={2011},
     journal={J. Math. Anal. Appl.},
      volume={382},
      number={1},
       pages={426\ndash 447},
}

\bib{samkokilbasmarichev93}{book}{
      author={Samko, S.~G.},
      author={Kilbas, A.~A.},
      author={Marichev, O.~I.},
       title={Fractional {I}ntegrals and {D}erivatives},
   publisher={Gordon and Breach Science Publishers, Yverdon},
        date={1993},
}

\bib{sunwu06}{article}{
      author={Sun, Z.-Z.},
      author={Wu, X.},
       title={A fully discrete difference scheme for a diffusion-wave system},
        date={2006},
        ISSN={0168-9274},
     journal={Appl. Numer. Math.},
      volume={56},
      number={2},
       pages={193\ndash 209},
         url={https://doi.org/10.1016/j.apnum.2005.03.003},
}

\bib{veeserzan2018}{article}{
      author={Veeser, A.},
      author={Zanotti, P.},
       title={Quasi-optimal nonconforming methods for symmetric elliptic problems. {I}---{A}bstract theory},
        date={2018},
        ISSN={0036-1429},
     journal={SIAM J. Numer. Anal.},
      volume={56},
      number={3},
       pages={1621\ndash 1642},
         url={https://doi.org/10.1137/17M1116362},
}

\bib{weihe14}{article}{
      author={Wei, L.},
      author={He, Y.},
       title={Analysis of a fully discrete local discontinuous {G}alerkin method for time-fractional fourth-order problems},
        date={2014},
        ISSN={0307-904X},
     journal={Appl. Math. Model.},
      volume={38},
      number={4},
       pages={1511\ndash 1522},
         url={https://doi.org/10.1016/j.apm.2013.07.040},
}

\bib{yazdaniweak2022}{article}{
      author={Yazdani, A.},
      author={Momeni, H.},
      author={Cheichan, M.~S.},
       title={A weak {G}alerkin/finite difference method for time-fractional biharmonic problems in two dimensions},
        date={2022},
        ISSN={0377-0427},
     journal={J. Comput. Appl. Math.},
      volume={410},
       pages={Paper No. 114195, 1\ndash 12},
         url={https://doi.org/10.1016/j.cam.2022.114195},
}

\bib{zhang2020}{article}{
      author={Zhang, H.},
      author={Yang, X.},
      author={Xu, D.},
       title={An efficient spline collocation method for a nonlinear fourth-order reaction subdiffusion equation},
        date={2020},
        ISSN={0885-7474},
     journal={J. Sci. Comput.},
      volume={85},
      number={1},
       pages={Paper No. 7, 18},
         url={https://doi.org/10.1007/s10915-020-01308-8},
}

\bib{zhang_pu2017}{article}{
      author={Zhang, P.},
      author={Pu, H.},
       title={A second-order compact difference scheme for the fourth-order fractional sub-diffusion equation},
        date={2017},
        ISSN={1017-1398,1572-9265},
     journal={Numer. Algorithms},
      volume={76},
      number={2},
       pages={573\ndash 598},
}

\end{biblist}
\end{bibdiv}
\bibliographystyle{alpha}

\newpage
\pagestyle{empty}
\appendix

\section{Appendix}
{In this section, we provide the details of the interpolation and companion operators and outline the proof of the error estimates of the Ritz projection defined in (\ref{bhritzprjn}). 
For details of the proof of the estimates, we refer to \cite{cn2022}.
\begin{lemma}[Morley interpolation and properties \cite{cn2022}]\label{bhmorley_lem}
The (extended) Morley interpolation operator $I_{\rm M}:  H^2(\T) \rightarrow {\rm M}(\mathcal{T})$  is defined in terms of degrees of freedom as follows:
\[
(I_{\rm M}v_{\pw})(z):= 
 |\T(z)|^{-1}
\sum_{K \in \T(z)} (v_\pw|_K)(z)
\text{ and }
 \fint_e\frac{\partial (I_{\rm M}v_{\pw})}{\partial \nu} \,ds := \fint_e \Big\{\!\!\!\Big\{\frac{\partial v_{\pw}}{\partial \nu} \Big\}\!\!\!\Big\}, \;\; v_{\pw}\in H^2(\T),
\]
where, for an interior vertex $z$, $\T(z)$ denotes the set of attached triangles with cardinality $|\T(z)|$.  The rest of the degrees of freedom at the vertices and edges on the boundary have no contribution due to homogeneous boundary conditions. The following properties hold for $I_{\rm M}$:
\begin{align}
\trinl I_\nc v_h   \trinr_\pw &\lesssim  \| v_h \|_h  \text{ for all }\;  v_h \in V_h \label{bhmorley_lem_extn1}  \\
\|v_h- I_\M v_h \|_h & \lesssim  \|v_h - v \|_h \text{ for all }v_h \in V_h \text{ and all }v \in V. \label{bhmorley_lem_extn2}
\end{align}
\end{lemma}
\noindent Note that the standard Morley interpolation operator $I_{\rm M} : V \rightarrow {\rm M}(\mathcal{T})$  defined by $(I_{\rm M} v)(z)=v(z)$ and 
$\fint_e\frac{\partial I_{\rm M} v}{\partial \nu}\,ds=\fint_e\frac{\partial v}{\partial \nu}\,ds$  for any  $z\in \mathcal V(\Omega) $ and $ e\in {\mathcal{E}(\Omega)}$ satisfies (\cite{cn2022})
\begin{align}
& a_{\rm pw}(v-I_{\rm M}v, w_h) = 0  \text{ for all } v \in V \text{ and all } w_h \in P_2(\mathcal{T}),\label{morley_lem_adp1}\\
&\trinl v- I_{\rm M} v \trinr_{\rm pw} \lesssim h^{\gamma} \|v\|_{H^{2+\gamma}(\Omega)}  \text{ for all } v \in  H^{2+\gamma}(\Omega), 0 \le \gamma \le 1. \label{morley_lem_adp2}
\end{align} 
\begin{lemma}[Companion operator and properties \cite{carsput2020,gal2015,veeserzan2018, cn2022}]\label{bhcompanion_lem}
Let ${\rm HCT}(\mathcal{T})$ denote the Hsieh-Clough-Tocher FEM \cite[Chapter 6]{ciarlet78}. There exists a linear mapping $J: {\rm M}(\mathcal{T})\to ({\rm HCT}(\mathcal{T})+P_8(\mathcal{T})) \cap V$ such that any $w_{\rm M}\in {\rm M}(\mathcal{T})$ satisfies
\begin{align*}
\begin{aligned}
&\text{(i) } Jw_{\rm M}(z)=w_{\rm M}(z) \text{ for any } z\in\mathcal{V} \; \text{ (ii) } \nabla ({J}w_{\rm M})(z)=
|\mathcal{T}(z)|^{-1}\sum_{K\in\mathcal{T}(z)}(\nabla w_{\rm M}|_K)(z)
 \text{ for }z\in\mathcal{V}(\Omega) \\
&\text{ (iii) } \fint_e \partial J w_{\rm M}/\partial\nu \,ds=\fint_e \partial w_{\rm M}/\partial\nu \,ds \text{ for any } e\in\mathcal{E}\; \text{ (iv) }  \trinl w_{\rm M}- J w_{\rm M} \trinr_{\rm pw} \lesssim  \min_{v\in V}  \trinl w_{\rm M}- v \trinr_{\rm pw} \\
&\text{ (v) } w_{\rm M}- J w_{\rm M} \perp P_2(\mathcal{T}) \text{ in } L^2(\Omega)\;\text{ (vi) } I_{\rm M}J=\rm id \text{ in } {\rm M}({\mathcal{T}}) \\
&\text{ (vii) }  \sup_{w_\M \in \M(\T) \setminus\{0\}}\trinl w_\M -J w_\M  \trinr_{\pw}/\trinl w_\M \trinr_{\pw}\lesssim 1.
\end{aligned}
\end{align*} 
\end{lemma} 
\begin{lemma}[Transfer operator and its properties \cite{cn2022, cgn2015}]\label{bhtransfer_lem}
The linear transfer operator $I_h: {\rm M}(\mathcal{T}) \rightarrow V_h$   is chosen as the identity map for the Morley and dG methods, and  is defined for $C^0$IP scheme as follows. For all $w_{\rm M} \in {\rm M}(\mathcal{T})$,  for all $z \in \mathcal{V}$, $(I_h w_{\rm M})(z)=w_{\rm M}(z)$, for $z= \text{mid}(e)$ with $e \in \mathcal{E}(\Omega)$, $(I_h w_{\rm M})(z)=\big\{\!\!\!\big\{w_{\rm M} \big\}\!\!\!\big\}(z)$, and for $z= \text{mid}(e)$ with $e \in \mathcal{E}(\partial\Omega)$, $(I_h w_{\rm M})(z)=0$. 
It holds that 
\begin{align}\label{bhtrans_approx}
\|w_{\rm M} - I_h w_{\rm M}\|_h \lesssim  \trinl w_{\rm M} - v \trinr_{\rm pw}
\text{ for all } w_{\rm M} \in {\rm M}(\mathcal{T}) \text{ and }  v \in V.   
\end{align}
\end{lemma}
\begin{lemma}[Relation between $\trinl \cdot \trinr_{\rm 
 pw}:=  a_{\rm pw}(\cdot,\cdot)^{1/2}$ and $\|\cdot\|_h$ \cite{cn2022}]\label{bhdisnorms_rel}
 The norms  $\|\cdot\|_h$ in  $H^2(\mathcal{T})$ and $\trinl \cdot \trinr_{\rm 
 pw}$ in $V+{\rm M}(\mathcal{T})$ are related by 
 \begin{align*}
 ({\rm i})\;\;     \trinl \cdot \trinr_{\rm pw}=\|\cdot \|_h \text{ in } V+{\rm M}(\mathcal{T}) \;\;  \text{ and }\;\;  ({\rm ii})\;\; \trinl \cdot \trinr_{\rm pw} \le \|\cdot \|_h \text{ in }   H^2(\mathcal{T}).
 \end{align*}
\end{lemma}
\noindent  The bilinear form $a_h(\cdot,\cdot):(V_h+\M(\T))  \times  ( V_h+\M(\T))  \rightarrow {\mathbb R}$ in  (\ref{bhsemidisscheme}) has the general form 
\begin{align}\label{bha_hform}
a_h(v_h,w_h):=
a_\pw (v_h,w_h) +  b_h(v_h,w_h) 
+ c_h(v_h,w_h) \text{ for all } v_h,w_h\in  V_h+\M(\T),
 \end{align}
 where $a_\pw(\cdot,\cdot), b_h(\cdot,\cdot), c_h(\cdot,\cdot):(V_h+\M(\T))  \times( V_h+\M(\T))  \rightarrow {\mathbb R}$ are all bounded bilinear forms on $(V_h+\M(\T))$. The sufficient conditions in {\bf (H1)}-{\bf(H4)} from \cite[Subsections 5.4 and 6.1]{cn2022} help to establish the approximation properties of the Ritz projection. It is verified in  \cite{cn2022} that all the lowest-order methods considered in this article satisfy these hypotheses.  
\begin{description}
\item[(H1)] For all $v_\nc \in \M(\T)$, $ w_h \in V_h$, and $v ,w\in V$,
\begin{align*}
a_\pw(v_\nc, w_h - I_\nc w_h) + b_h(v_\nc, w_h- I_\nc w_h) \lesssim
\trinl v_\nc-v \trinr_\pw \|w_h -w\|_h.      
\end{align*}
\item[(H2)] For all $v_\nc,w_\nc \in \M(\T)$ and $ v_h, w_h \in V_h$,
\begin{align*}
\text{ (i) } b_h(v_\nc,w_\nc)=0  \;\; \text{ (ii) }\;
b_h(v_h + v_\nc, w_h+ w_\nc) \lesssim \|v_h+v_\nc\|_h \|w_h+ w_\nc\|_h.
\end{align*}
\item[(H3)] For all $v_h, w_h \in V_h$  and $ v,w \in V$, $c_h(v_h, w_h) \lesssim \| v - v_h\|_h  \|w- w_h\|_h$.
\item[(H4)] For all  $v_h \in V_h$, $w_\nc \in \M(\T)$, and all $v,w \in V$,  
\begin{equation*}
 a_\pw ( v_h -I_\nc v_h, w_\nc)+b_h( v_h, w_\nc)\lesssim \| v-v_h \|_h   \trinl w- w_\nc \trinr_\pw . 
\end{equation*}
\end{description}
\begin{theorem}[Energy norm bound]\label{bhritzenergy_pf}
Under the hypotheses {\bf (H1)}--{\bf (H3)},  the Ritz projection $\mathcal{R}_h$ satisfies 
\begin{align*}
\|u-\mathcal{R}_hu\|_h\lesssim  h^{\gamma} \|u\|_{H^{2+\gamma}(\Omega)}, \;\; \gamma\in [0,\gamma_0] \text{ with } \gamma_0\in (1/2,1].     
\end{align*}
\end{theorem}
\begin{proof}
Let $\zeta_h:=I_h I_{\rm M}u-\mathcal{R}_hu$ and split the error $u-\mathcal{R}_hu$ as a sum of three terms given by 
$u-\mathcal{R}_hu=(u-I_{\rm M}u)+(I-I_h)I_{\rm M}u+\zeta_h.$
Lemma \ref{bhdisnorms_rel} (i) and (\ref{bhtrans_approx}) show
\begin{align}\label{bhritzenergy_eq1}
\begin{aligned}
\|u-\mathcal{R}_hu\|_h &\lesssim  \trinl u - I_\nc u \trinr_\pw  +\|\zeta_h\|_h.
\end{aligned}
\end{align}
For $Q=JI_\M$, the coercivity of $a_h(\cdot,\cdot)$ from \eqref{bhbilinearcoer_cts} and (\ref{bhritzprjn}) reveal
\begin{align*}
\begin{aligned}
&\|\zeta_h\|_h^2\lesssim a_h(\zeta_h,\zeta_h)= a_h(I_h I_{\rm M}u,\zeta_h)-a_h(\mathcal{R}_hu,\zeta_h)=  a_h(I_h I_{\rm M}u,\zeta_h)-a(u,Q\zeta_h)\\
&= a_\pw( u, (I-J)I_\nc \zeta_h )- a_\pw(I_\nc u, I_\nc \zeta_h )+a_h(I_h I_{\rm M}u,\zeta_h)
\end{aligned}   
\end{align*}
with (\ref{morley_lem_adp1}) in the last step. 
Lemma \ref{bhcompanion_lem} (vi) plus
 (\ref{morley_lem_adp1}) show 
 $a_\pw( I_\nc u, (I-J)I_\nc \zeta_h )=0$. 
This, (\ref{bha_hform}), and \textbf{(H2)} (i) in the above displayed expression show
\begin{align}\label{bhzetabnd}
\begin{aligned}
\|\zeta_h\|_h^2&\lesssim a_\pw( u- I_\nc u, (I-J)I_\nc \zeta_h )+  (a_\pw+b_h)  ((I_h-I) I_\nc u , \zeta_h)\\
& \quad +   ((a_\pw+b_h)(I_\nc u, \zeta_h-I_\nc \zeta_h) +  c_h(I_h I_\nc u, \zeta_h))=: T_1+T_2+T_3.    
\end{aligned}    
\end{align}
The Cauchy-Schwarz inequality, Lemma \ref{bhcompanion_lem} (vii), and (\ref{bhmorley_lem_extn1}) establish 
\begin{align*}
T_1=a_\pw(u - I_\nc u, (I-J) I_\nc \zeta_h  ) \le  \trinl u - I_\nc u \trinr_\pw \trinl (I-J) I_\nc \zeta_h \trinr_\pw \lesssim  
\trinl u - I_\nc u \trinr_\pw \|\zeta_h\|_h.
\end{align*}
The Cauchy-Schwarz inequality once again, \textbf{(H2)} (ii), Lemma \ref{bhdisnorms_rel} (ii),  and (\ref{bhtrans_approx}) imply
\begin{align*}
T_2 &=  a_\pw  ((I_h-I) I_\nc u , \zeta_h)+b_h  ((I_h-I) I_\nc u , \zeta_h)\lesssim  \|(I-I_h) I_\nc u \|_h \|\zeta_h\|_h 
\lesssim \|u- I_\nc u \|_\pw \|\zeta_h\|_h.
\end{align*}
To bound the first term in $T_3$, apply  \textbf{(H1)} with $w=0$ 
to arrive at
\begin{align*}
(a_\pw+b_h)(I_\nc u, \zeta_h-I_\nc \zeta_h)\lesssim \trinl u - I_\nc u\trinr_\pw \|\zeta_h\|_h.%
\end{align*}
Utilize \textbf{(H3)} with $w=0$, the triangle inequality, (\ref{bhtrans_approx}), and Lemma \ref{bhdisnorms_rel} (i) to obtain
\begin{align*}
\begin{aligned}
c_h(I_h I_\nc u, \zeta_h)&\lesssim  \| u - I_h I_\nc u\|_h \|\zeta_h\|_h\le (\| u - I_\nc u\|_h+\| (I - I_h) I_\nc u\|_h)\|\zeta_h\|_h \lesssim \trinl u - I_\nc u\trinr_\pw\|\zeta_h\|_h.
\end{aligned}
\end{align*}
A combination of the bounds for $T_1$, $T_2$, and $T_3$ (from last two displayed inequalities) in (\ref{bhzetabnd}) and then in (\ref{bhritzenergy_eq1}), with (\ref{morley_lem_adp2}) to bound $\trinl u - I_\nc u\trinr_\pw$ concludes the proof. 
\end{proof}
\begin{theorem}[Weaker and piecewise Sobolev norm bound]\label{bhritzweak_pf}
Under the hypotheses {\bf(H1)}-{\bf (H4)}, the Ritz projection $\mathcal{R}_h$ satisfies  
\begin{align*}
\norm[0]{u-\mathcal{R}_hu}_{H^s(\mathcal{T})}\lesssim h^{2-s}\norm[0]{u-\mathcal{R}_hu}_h \text{ for all } 2-\gamma \le s\le 2  \text{ and for all } \gamma\in (0,\gamma_0].     
\end{align*}
In particular, we have  
$$ \|u- \mathcal{R}_hu\| \le \|u-\mathcal{R}_hu\|_{H^{2-\gamma}(\mathcal{T})} \lesssim h^{\gamma} \norm[0]{u-\mathcal{R}_h u}_h  \text{ for all } \gamma\in (0,\gamma_0]. $$  
\end{theorem}  
\noindent The proof is based on duality arguments. Here we sketch only the main ideas and for details, we refer to \cite[Section 6]{cn2022}.
Note that the $L^2$ estimates in the last displayed expression cannot be improved further for $P_2$ approximations \cite{cn2021cmam}.

\medskip \noindent{\it Outline of the proof.}  For $u\in V$ and $v=Q \mathcal{R}_hu\in V$, the  duality of $H^{-s}(\Omega)$ and $H^s_0(\Omega)$ reveals 
\begin{align*}
\|u -v\|_{H^{s}(\Omega)} = \sup_{0 \ne \mathcal{F} \in H^{-s}(\Omega)} \frac{\mathcal{F}(u-v)}{\|\mathcal{F}\|_{H^{-s}(\Omega)}}= \widetilde{\mathcal{F}}(u-v).   
\end{align*}
A corollary of the Hahn-Banach theorem ensures that the supremum is attained for some  $\widetilde{\mathcal{F}}\in H^{-s}(\Omega)\subset V^*$ with $\|\widetilde{\mathcal{F}}\|_{H^{-s}(\Omega)} = 1$. The functional $a(z,\cdot)=\widetilde{\mathcal{F}} \in V^*$ has a unique Riesz representation $z \in V$ in the Hilbert space $(V, a)$;  $z\in V$ is the weak solution of  $\Delta^2 z = \widetilde{\mathcal{F}}$.  By virtue of the elliptic regularity \cite{blumrannacher80}, see also Remark \ref{index_remark}, we have $z\in V\cap H^{4-s}(\Omega)$, with $ \|z \|_{ H^{4-s}(\Omega)}\lesssim \|\widetilde{\mathcal{F}}\|_{H^{-s}(\Omega)}\lesssim 1$ for all $2 -\gamma \leq s\leq 2$ and all $\gamma\in(0,\gamma_0]$. 
Hence 
\begin{align}\label{bhweaknm_est}
\|u -v\|_{H^{s}(\Omega)}=a(u-v,z) \text{ and } 
 \trinl  z - I_\nc z \trinr _\pw \lesssim h^{2-s}  \|z \|_{ H^{4-s}(\Omega)} \lesssim h^{2-s}.
\end{align}
\noindent For $  z_h :=I_h I_\nc z \in V_h$ {and} $\xi:= Qz_h \in V
$, the definition of the Ritz projection $\mathcal{R}_h$ (cf. (\ref{bhritzprjn})) and elementary algebra show the key identity \cite[Lemma 6.3]{cn2022}:
\begin{align*}
& a(u-v,z) = a(u-v,z - \xi) +a_\pw(\mathcal{R}_hu-v, \xi-z_h) + a_{\pw} ((I-J)I_\nc \mathcal{R}_hu, z_h -I_\nc z_h) \notag \\
& +  a_{\pw} ((I-I_\nc) \mathcal{R}_hu , I_\nc z_h -\xi)+ a_h(\mathcal{R}_hu,z_h) - a_\pw (\mathcal{R}_hu, I_\nc z_h) + a_\pw ((I-I_\nc)\mathcal{R}_h u,z_h).
\end{align*} 
All the terms on the right-hand side of the  identity above can be bounded  in modulus by a constant times 
$\norm[0]{u-\mathcal{R}_hu}_h  \trinl  z - I_\nc z \trinr _\pw $. The proof of this part is skipped for brevity; we refer to Lemmas 6.4 and 6.5 in \cite{cn2022} for details. Hence with   \eqref{bhweaknm_est}, we have 
$$ a(u-v,z) \lesssim h^{2-s} \norm[0]{u-\mathcal{R}_hu}_h.$$

\noindent To conclude the proof,  split $u-\mathcal{R}_hu$ as
$ u-\mathcal{R}_hu=(u-v)+(v-\mathcal{R}_hu)$.  
Invoke (\ref{bhweaknm_est}) and the above displayed bound to establish
\begin{equation*}
\|u -v\|_{H^{s}(\Omega)}\lesssim  h^{2-s}\|u-\mathcal{R}_hu\|_h.
\end{equation*}
The definition of the Sobolev-Slobodeckii semi-norm shows
$\|u -v\|_{H^{s}(\T)}^2=\sum_{K\in\T} |u -v|_{H^s(K)}^2 \le |u -v |_{H^s(\Omega)}^2   \text{ for }\;  1<s<2.$
This and the last displayed estimate provide
\begin{align*}
\|u -v\|_{H^{s}(\T)}\lesssim  h^{2-s}\|u-\mathcal{R}_hu\|_h  \text{ for all }   2-\gamma \le s \le 2.  
\end{align*}
Since $\mathcal{R}_hu\in P_2(\T)$,  Lemma \ref{bhcompmorlyestmte} provides the estimate
$\|v -\mathcal{R}_hu\|_{H^{s}(\T)}\lesssim  h^{2-s}\|u-\mathcal{R}_hu\|_h.$
A triangle inequality and a combination of the last two estimates conclude the proof. 
 \hfill{$\Box$}

}

\end{document}